\documentclass{amsart}[12pt,a4paper]

\usepackage[notcite,notref]{showkeys}
\usepackage[left=3.0cm,right=3.0cm,top=3.5cm,bottom=3.5cm]{geometry}
\usepackage{mathrsfs}
\usepackage{mathtools}
\usepackage{graphicx}
\usepackage{epstopdf} 
\usepackage{stmaryrd}
\usepackage{subfigure}
\usepackage{float}
\usepackage{bm}
\usepackage{amsfonts,amssymb,amsthm}
\usepackage{hyperref}
\usepackage{xcolor}

\newtheorem{theorem}{Theorem}[section]
\newtheorem{lemma}[theorem]{Lemma}

\newtheorem{corollary}[theorem]{Corollary}
\newtheorem{remark}[theorem]{Remark}

\numberwithin{equation}{section}

\newcommand{\bbR}{\mathbb{R}}
\newcommand{\bbQ}{\mathbb{Q}}

\newcommand{\calT}{\mathcal{T}}
\newcommand{\calE}{\mathcal{E}}

\newcommand{\calB}{\mathcal{B}}

\newcommand{\calO}{\mathcal{O}}

\newcommand{\e}{K}

\def\vn{{\bf n}}

\def\vu{{\bm u}}
\def\ve{{\bm e}}
\def\vv{{\bm v}}
\def\vx{{\bm x}}
\def\vV{{\bm V}}

\def\vQ{{\bm Q}}

\def\vw{{\bm w}}
\def\vzero{\bm{0}}

\newcommand{\hx}{\hat{x}}
\newcommand{\hy}{\hat{y}}

\def\vtn{\widetilde{\bf n}}
\def\vtQ{\widetilde{\bm Q}}

\def\we{\tilde{e}}
\def\vsig{\bm{\Gamma}}
\def\ah{a_{h,1}}
\def\sh{c_{1}}
\def\nnorm{\bm{V}_{h,1}}

\newcommand{\dx}{\,\mathrm{d}x}
\newcommand{\dhx}{\,\mathrm{d}\hat{x}}
\newcommand{\dhy}{\,\mathrm{d}\hat{y}}
\newcommand{\ds}{\,\mathrm{d}s}

\definecolor{liua}{rgb}{1,0,0}
\newcommand{\liul}[1]{{\color{black}#1}}
\newcommand{\wangI}[1]{{\color{black}#1}}
\begin{document}

\title[Weak Galerkin mixed method on curved domains]{A Weak Galerkin Mixed Finite Element Method for second order elliptic equations on 2D Curved Domains}
\author{Yi Liu}
\address{School of Mathematical Sciences and Jiangsu Key Laboratory for NSLSCS, Nanjing Normal University, Nanjing, China}
\email{200901005@njnu.edu.cn}

\author{Wenbin Chen}
\address{School of Mathematical Sciences and Shanghai Key Laboratory for Contemporary Applied Mathematics, Fudan University, Shanghai, China}
\email{wbchen@fudan.edu.cn}

\author{Yanqiu Wang}
\address{School of Mathematical Sciences and Jiangsu Key Laboratory for NSLSCS, Nanjing Normal University, Nanjing, China}
\email[Corresponding author]{yqwang@njnu.edu.cn}

\thanks{Liu and Wang are supported by the NSFC grant 12171244. Chen is supported by NSFC grant 12071090.}


\begin{abstract}
  This article concerns the weak Galerkin mixed finite element method (WG-MFEM) for second order elliptic equations on 2D domains with curved boundary.
  The Neumann boundary condition is considered since it becomes the essential boundary condition in this case.
  It is well-known that the discrepancy between the curved physical domain and the polygonal approximation domain leads to a loss of accuracy
  for discretization with polynomial order $\alpha>1$.
  The purpose of this paper is two-fold. First, we present a detailed error analysis of the original WG-MFEM for solving problems on curved domains,
  which exhibits an $O(h^{1/2})$ convergence for all $\alpha\ge 1$. 
  It is a little surprising to see that even the lowest-order WG-MFEM ($\alpha=1$) experiences a loss of accuracy.
  This is different from known results for the finite element method (FEM) or the mixed FEM,
  and appears to be a combined effect of the WG-MFEM design and the fact
  that the outward normal vector on the polygonal approximation domain is different from the one on the curved domain.
  Second, we propose a remedy to bring the approximation rate back to optimal
  \wangI{by employing two techniques.
    One is a specially designed boundary correction technique.
    The other is to take full advantage of the nice feature that weak Galerkin discretization can be
    defined on polygonal meshes, which allows the curved boundary to be better approximated by multiple short edges
    without increasing the total number of mesh elements.}
  Rigorous analysis shows that a combination of the above two techniques renders optimal convergence for all $\alpha$.
  Numerical results further confirm this conclusion.
  
  \end{abstract}


\maketitle

\section{Introduction}
Many practical problems arising in science and engineering are posed on domains with curved boundaries.
When such problems are approximated on polygonal or polyhedral computational domains,
the geometric difference between the two leads to a loss of approximation accuracy \cite{strang1973change,thomee1971polygonal} for high-order elements.
To \wangI{resolve} this issue, a straight-forward idea is to reduce the geometric error down to the same level of the approximation error.
Popular methods following this track include the isoparametric finite element method \cite{e1,l1}
and the isogeometric analysis \cite{h1,h2}. However, due to their specialized design, neither of them can
be applied to meshes consisting of polygons or polyhedra.

In the past two decades, discretizations on polygonal and polyhedral meshes have gained considerable
attention in the scientific computing community. Various numerical schemes have been proposed, including
the virtual element method (VEM) (see \cite{Veiga13} and references therein),
the discontinuous Galerkin method (DG) \cite{Cockburn09, Gassnera09, mulin2015dg},
and the weak Galerkin method (WG) \cite{wang2013weak, wang2014weak, mwy-wg-stabilization}, to name a few.
Very recently, researchers start to apply these discretizations to curved domains, which requires
\wangI{innovative techniques, with its reason explained above}.
One emerging method is the boundary correction technique, which may have its root date back to a 1972 paper by Bramble, Dupont and Thom\'{e}e \cite{BDH72}.
The idea is to use normal-directional Taylor expansion, in most cases just a linear approximation, to correct function values on the boundary.
Burman et. al. \cite{Burman2018, Burman2019} proposed the technique for a CutFEM discretization in 2018.
It was soon applied to VEM by Bertoluzza et. al. \cite{Bertoluzza2019}.
In 2019, Cheung et. al. \cite{t4} proposed a polynomial extension FEM which is based on an averaged Taylor expansion.
In a series of papers starting from 2018, Main and Scovazzi \cite{Main2018a, Main2018b, Atallah20}
designed a shifted boundary FEM on non-fitted meshes, i.e., boundary nodes of the mesh may not lie on the curved physical boundary.
However, their method uses 1st-order (linear) Taylor expansion and hence only works for linear elements.
Finally, we mention an earlier but closely related work \cite{c3}, where instead of Taylor expansion
the authors used a path integration to achieve a similar `boundary correction' effect.

A totally different track, first proposed for VEM by Beir\~{a}o da Veiga et. al. \cite{Veiga19} in 2019,
is to define the discretization directly on curved mesh elements.
This is possible because of the `skeletal' style design of VEM, where the degrees of freedom (dofs) in the interior and on the boundary of each
mesh element are defined separately.
In \cite{Veiga19}, dofs as well as related shape functions on curved boundary edges
are defined using the parameter in the parametric equation of the curved boundary.
Hence the shape functions on curved edges are no longer polynomials in the physical space.
Later a modification was proposed \cite{Veiga20}
which uses the restriction of physical polynomials to define shape functions on curved boundary edges.
In 2021, Mu applied the idea to the primal WG discretization \cite{Mu21}.
Because of the direct use of the curved boundary, the implementation of these methods requires
\wangI{numerical integration formulae on curved mesh elements, as well as
a mapping between each curved boundary segment and its flat counterpart in the parametric space}.
The theoretical analysis is also more complicated as it has to deal with the parametric mapping.

In this paper, we consider the weak Galerkin mixed finite element method (WG-MFEM) on 2D curved domains.
Since its first debut, the WG method has been quickly applied to various situations
\cite{chen2016weak, mu2014weak, Zhangran, Zhangran1, wang2014weak, wang2016weak, zhang2018weak},
among which \cite{wang2014weak} focuses on the WG-MFEM for second order elliptic equations and
\cite{chen2016weak} uses WG-MFEM to solve coupled Darcy-Stokes equations.
However, there seems to be no rigorous analysis of the original WG-MFEM on curved domains.
The first objective of this paper is to fill this gap.
Although analysis of discretizations in the primal formulation on curved domains has long been
well-known, study of the mixed formulation is relatively rare.
We shall first clarify that in either cases, the main difficulty lies in how to impose the essential boundary condition.
For the mixed formulation of second order elliptic equations, it is the Neumann boundary condition that becomes essential.
A subtlety arises as the Neumann boundary condition involves the outward normal vector on the boundary,
which is {\em different} on the curved boundary and its polygonal approximation.
To our knowledge, such a problem, as well as subsequent mixed-FEM error analysis on curved domains, was first studied by
Bertrand et. al. \cite{bertrand2014first, bertrand2014first1} for the Raviart-Thomas element in 2014.
Both the analysis and numerical results show that the loss of accuracy only occurs for high-order elements.
Therefore it was a little surprising when we found that the original WG-MFEM yields only an $O(h^{1/2})$ convergence for all $\alpha$,
which means an accuracy loss even for the lowest-order WG-MFEM discretization.
This appears to be a combined effect of the WG-MFEM design and the outward normal vector issue mentioned above.

We then propose a remedy to bring the approximation rate back to optimal, which is the second objective of this paper.
\wangI{The remedy employs two techniques. One is a specially designed boundary correction technique that
  treats the difference between outward normal vectors on the curved domain and the polygonal computational domain.
  Unlike the boundary correction \cite{Burman2018, Burman2019, t4, Main2018a, Main2018b, Atallah20} designed for Dirichlet boundary conditions,
  ours mainly deals with the normal component of the flux.
  The implementation is easy and does not involve integration on curved regions.
  But one has to be careful about the discrete compatibility condition, noticing that the solution to the pure Neumann boundary problem is not unique.
  The other technique is a simple strategy proposed in \cite{wgcurved},
  which uses multiple short straight edges to obtain a better geometric approximation to the curved boundary.
  Since the WG discretizations can be defined on polygonal meshes, the curved boundary is then better approximated
  by polygonal elements with multiple short edges.
  Mathematically, this multiple short edge approach}
is still a `linear' approximation and hence the number of short edges required will definitely increase with $\alpha$.
The main advantage of this strategy \wangI{lies in} its simplicity in the implementation.
{\em The algorithm itself remains untouched.} One only needs to provide a new polygonal mesh consisting of multiple short boundary edges
which better approximates the curved domain.
We show that \wangI{combining the above two techniques and} using suitable meshes, the modified WG-FEM reaches optimal approximation rates for all $\alpha$.
 
The paper is organized as follows.
In Section \ref{WF}, we introduce the model problem, the notation and mesh assumptions.
In Section \ref{section3}, we present a rigorous analysis of the original WG-MFEM on curved domains.
The modified WG-FEM and its theoretical analysis are given in Section \ref{section4}.
Finally, numerical results are presented in Section \ref{section5}.

\section{Model problem and mesh assumptions} \label{WF}

Let $\Omega\subset \bbR^2$ be a general domain with Lipschitz-continuous and possibly curved boundary $\partial\Omega$.
Consider the Poisson's equation in its mixed form:
{\it Given $g$ in $L^2(\Omega)$, find functions $\vu$ and $p$ such that}
\begin{equation}
\begin{aligned}
\vu+\nabla p&=\vzero\quad &&\text{in } \Omega,\\\label{primalproblem}
\nabla\cdot\vu&=g\quad &&\text{in } \Omega.
\end{aligned}
\end{equation}

Here we conveniently use bold face characters to denote vectors or vector-valued functions.
Equip system \eqref{primalproblem} with the homogeneous Neumann boundary condition as follows:
\begin{equation}
\vu\cdot \vtn=0\qquad \text{on } \partial\Omega,\label{primalcondition}
\end{equation}
where $\vtn$  is the unit outward normal vector on $\partial\Omega$.
System \eqref{primalproblem}-\eqref{primalcondition} is well-posed as long as the following compatibility condition holds:
\begin{equation} \label{eq:compat}
\int_{\Omega}g\dx=0.
\end{equation}
Moreover, the solution is unique assuming $\int_{\Omega}p\dx=0$.

Denote by $H^m(D)$, $m\geq 0$, the usual Sobolev space defined on an open bounded domain $D\subset\Omega$,
and endow it with the inner-product $(\cdot,\cdot)_{m,D}$, the norm $\|\cdot\|_{m,D}$ and the seminorm $|\cdot|_{m,D}$.
Let $H_0^m(D) = \overline{C_0^{\infty}}^{H^m(D)}$.
When $m=0$, the space $H^0(D)$ is identical to $L^2(D)$.
In this case, we use $(\cdot,\cdot)_D$ to denote the $L^2$ inner-product on $D$.
The above notation also extends to a curve or edge $e$ in $\overline{D}$.
Moreover, denote by $\langle\cdot,\cdot\rangle_{e}$ the duality pair on $e$.
Denote by $H(\text{div},D)$ the space of vector-valued functions with all its components and divergence in $L^2(D)$.

Define spaces
$$
\begin{aligned}
\vV&= \{\vv\in H(\text{div},\Omega)\,|\,\vv\cdot\vtn=0 \text{ on } \partial\Omega\},\\
\Psi&= L^2_0(\Omega)\triangleq\left\{q\in L^2(\Omega)\,\bigg{|}\,\int_{\Omega}q\dx=0\right\}.
\end{aligned}
$$
 Then, the mixed variational formulation of the Poisson's equation with homogeneous Neumann boundary condition can be written as: {\it Find $(\vu,p)\in\vV\times\Psi$ such that}
 \begin{equation}\label{wf}
 \begin{cases}
 a(\vu,\vv)+b(\vv,p)=0\quad&\forall\, \vv\in\vV,\\
 b(\vu,q)=-(g,q)_{\Omega}\quad&\forall\, q\in\Psi,
 \end{cases}
 \end{equation}
  where 
  $$a(\vu,\vv)= (\vu,\vv)_{\Omega}, \qquad b(\vv,q)=-(\nabla\cdot\vv,q)_{\Omega}.$$
  The existence and uniqueness of a weak solution to the mixed problem (\ref{wf}) can be found in \cite{boffi2013mixed}.
  
\smallskip
Let $\calT_h$ be a body-fitted partition of $\Omega$ consisting of polygons.
By `body-fitted' we mean that each boundary edge of $\calT_h$ has its two end points lying on $\partial\Omega$.
Let $\overline{\Omega}_h=\cup_{\e\in\calT_h}\e$.
If $\Omega$ is a polygonal domain in $\bbR^2$, the domains $\Omega_h$ and $\Omega$ are identical in the case of a fitted mesh, i.e.,  $\Omega_h=\Omega$.
But when $\Omega$ has curved boundary, $\Omega_h$ differs from $\Omega$. It's easy to see that when $\Omega$ is convex, we have $\Omega_h\subset\Omega$.
In general, $\Omega_h$ can be viewed as an approximation to $\Omega$.
For the sake of simplicity, we only present theoretical analysis for the case when $\Omega$ is convex.
Generalization to non-convex domains can be done but requires a non-trivial use of extension operators.
In Section \ref{section5}, a numerical example in non-convex domain is presented, which exhibits the same behavior as experiments in convex domains.

Denote by $h_\e$ the diameter of each element $\e\in\calT_h$, and let $h = \max_{\e\in\calT_h} h_\e$.
Denote by $\e_0$ and $\partial\e$ the interior and the boundary of $\e$, respectively. 
Let $\calE_h$ be the set of all edges in $\calT_h$, which are straight edges.
For each edge $e \in\calE_h$, denote by $h_e$ its length.
Let $\calE^I_h$ and $\calE^B_h$ be the set of all interior and boundary edges, respectively, in $\calT_h$.
Denote $s = \max_{e\in\calE_h^B} h_e$. It is obvious that $s\le h$.
Later we shall allow $s$ to be much smaller than $h$ in order to get a more accurate approximation to the curved boundary of $\Omega$.

Note that each $e\in\calE_h^B$ has its two end points lying on $\partial\Omega$.
Denote by $\tilde{e}\subset\partial\Omega$ the (short) section intersected by these two end points, 
and by $M_e$ the crescent-shaped region surrounded by $e$ and $\tilde{e}$, as shown in Figure \ref{sigma}.
It is possible that part of $\partial\Omega$ is indeed flat. In this case $\tilde{e} = e$ and $M_e$ is just empty.
We use $\partial \Omega_h$ to denote the boundary of $\Omega_h$, which consists of all $e\in \calE_h^B$.
Recall that $\vtn$ is the unit outward normal vector on $\partial\Omega$, i.e., on curved edge $\tilde{e}$.
Denote by $\vn$ the unit outward normal vector on $\partial\Omega_h$, i.e., on straight edge $e\in\calE_h^B$.
Later we also need unit outward normal vectors on $\partial K$ for each $K\in \calT_h$.
Since $\partial K$ consists of straight edges, they are still denoted by $\vn$ and should not bring any ambiguity.

Denote by $\calT_h^B$ all mesh elements containing at least one edge in $\calE_h^B$.
It is possible that an element $K\in \calT_h^B$
contains multiple edges in $\calE_h^B$.

\smallskip
Following \cite{mulin2015dg, wgcurved}, we introduce a set of regularity assumptions on the partition $\calT_h$.

\textbf{A1.} There exists a positive constant $C_1$ such that each element $\e\in\calT_h$ is star-shaped with respect to a ball $\calB_{\e}\subset \e$ with radius $\rho_{\e} $ satisfying 
$$\rho_{\e}\geq C_1 h_{\e}.$$

\textbf{A2.} There exists a positive constant $C_2$ such that each element $\e\in\calT_h$ has at least one edge $e$ with length
$$h_e\geq C_2 h_{\e}.$$

\textbf{A3.} The mesh is quasi-uniform, that is, there exists a positive constant $C_3$ such that for every element $\e\in\calT_h$ one has
$$1\leq\dfrac{h}{h_{\e}}\leq C_3.$$

\textbf{A4.} There exist positive constants $C_4$ and $C_4'$ such that: for each edge $e\in\calE_h^B$ and corresponding $\we$,
there is a one-to-one map $\vsig : e\rightarrow\we$ defined in a local coordinate system $\hat{x}$-$\hat{y}$ (see Figure \ref{sigma}) by
$$
\vsig(\hat{x},0)=(\hat{x},\gamma(\hat{x})) \in\we \qquad \forall\,\hat{x}\in [0,h_e],
$$
where the function $\gamma$ satisfies
$$
\sup_{\hat{x}\in[0,h_e]}|\gamma(\hat{x})-\hat{x}|\leq C_4 s^2,
$$
and 
$$\sup_{\hat{x}\in[0,h_e]}|\vn(\hat{x},0)-\vtn(\hat{x},\gamma(\hat{x}))|\leq C_4' s.$$
\begin{figure}[H]
\begin{center}
\includegraphics[width=1.5in]{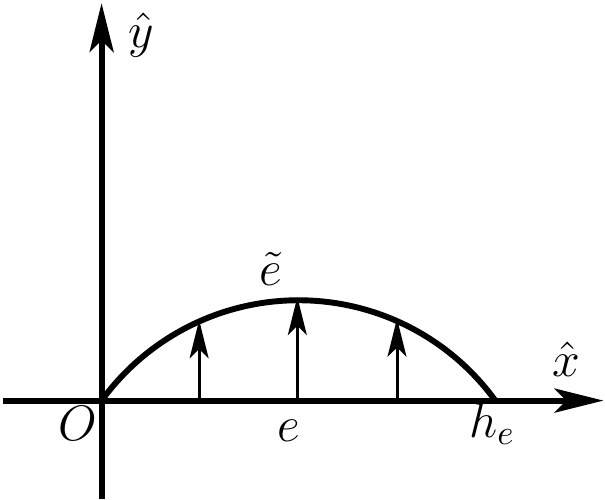}\\
\caption{The map $\vsig$.}
\label{sigma}
\end{center}
\end{figure}

\textbf{A5.} The boundary edges in $\calE_h^B$ are quasi-uniform, i.e., there exists a positive constant $C_5$ such that
$$s\leq C_5 \min_{e\in\calE_h^B}h_e.$$

\textbf{A6.} There exists a positive constant $C_6$ such that: for every $e\in\calE_h^B$ and the unique polygon $\e\in \calT_h^B$ having $e$ as an edge,
one can draw a triangle $P(e)$ with base $e$ and the center of ball $\calB_{\e}$ (defined in assumption \textbf{A1}) as apex;
the height $h_{\e,e}$ of $P(e)$ obviously satisfies $\rho_{\e}\leq h_{\e,e}\leq h_{\e}$;
denote by $\hat{\vx}$ the local coordinate system with $e$ and the height of $P(e)$ as abscissa axis and ordinate axis,
and define a linear transformation $F$ by
$$
F(\hat{\vx})=\left(
\begin{aligned}
h_e^{-1}&&0\\
0&&h^{-1}_{\e,e}\\
\end{aligned}
\right)\hat{\vx},
$$
then the triangle $\calO=F(P(e))$ satisfies
\begin{itemize}
\item$\text{diam}(\calO)=O(1)$;
\item The radius ratio of the circumscribed circle and the inscribed circle of $\calO$ is less than or equal to $C_6$.
\end{itemize}
We further assume that triangles $P(e)$ from all $e\in \partial \e \cap\partial\Omega_h$ form a finite overlapping,
in the sense that each point in $\e$ can only be covered by no more than $M$ such triangles.

\begin{remark}
  \textbf{A1-A2} are polygonal shape regularity conditions proposed in \cite{wang2014weak, mulin2015dg}.
  \textbf{A3} is the quasi-uniform assumption.
  These three are standard mesh assumptions used in WG discretizations.
  It is well-known that \textbf{A4} always holds as long as $\partial\Omega$ is piecewise $C^2$ continuous
  \cite{thomee1971polygonal, bertrand2014first}.
  Conditions \textbf{A5-A6} are extra requirements for the case of curved domains, which were first proposed in \cite{wgcurved}.
  We shall see later how these conditions are used in the analysis.
\end{remark} 
  
\begin{remark}
  Note that conditions \textbf{A1-A6} allow the existence of small edges,
  and hence allow the approximation of the curved boundary by multiple straight short edges.
\end{remark}

\begin{remark} \label{rem:A5A6}
  A less obvious but important consequence of \textbf{A5-A6} is that, each $\e\in\calT_h^B$ contains
  at most $O(\frac{h}{s})$ edges in $\calE_h^B$, i.e., boundary edges in $\calE_h^B$ should never be in a zigzag formation.
\end{remark}

For the sake of brevity, throughout this article, we write $x\lesssim y$ and $x\gtrsim y$ in place of $x\leq Cy$ and $x\geq Cy$, respectively,
for a positive constant $C$ independent of the discretization parameters.
Moreover, $x\approx y$ means that there exist positive constants $c,C$ such that $c y\leq x\leq C y$. When required, the constants will be written explicitly.

In the entire paper, we always assume that the mesh $\calT_h$ satisfies assumptions \textbf{A1-A6}.
This guarantees the trace inequality, the inverse inequality and a few other important lemmas on the curved domains.

\begin{lemma} \label{lem:trace}
(Trace inequality \cite{mulin2015dg, wgcurved, wang2014weak}) For $\e\in\calT_h$ and $v\in H^1(\e)$, one has
\begin{equation*}
\begin{aligned}
\|v\|_{0,\partial \e}^2\lesssim h^{-1}\|v\|_{0,\e}^2+h\|\nabla v\|_{0,\e}^2.
\end{aligned}
\end{equation*}
\end{lemma}

Given a non-negative integer $j$, we denote by $P_j(\e)$ the space of polynomials with degree less than or equal to $j$ on $\e\in\calT_h$.
\begin{lemma} \label{inv}
(Inverse inequality \cite{mulin2015dg, wgcurved, wang2014weak}) For $\e\in\calT_h$ and $v\in P_j(\e)$, one has
\begin{equation*}
\begin{aligned}
\|\nabla v\|_{0,  \e}\lesssim h^{-1}\|v\|_{0,\e},
\end{aligned}
\end{equation*}
where the hidden constant in $\lesssim$ may depend on $j$ but not on $h$ or the shape of $\e$.
\end{lemma}

Note that if $\e\in \calT_h^B$, polynomials in $P_j(\e)$ can be naturally extend to any adjacent crescent-shaped region $M_e$, or vice versa.
\begin{lemma} \label{lem:inv}
  For $\e\in\calT_h^B$ and $v\in P_j(\e)$, one has
  $$
  \sum_{e\in\partial K\cap\partial\Omega_h} \|v\|_{0,M_e}^2 \lesssim h^{-1} s^2 \|v\|_{0,K}^2.
  $$
\end{lemma}
\begin{proof}
  Take the center of $\calB_K$, as defined in Assumption \textbf{A1}, and draw a ball with radius $h_K+C_4s^2$.
  Denote the larger ball by $\calB_K'$. Then, using Assumption \textbf{A4} and the fact that each $\e\in\calT_h^B$
  contains at most $O(\frac{h}{s})$ edges in $\calE_h^B$ (see Remark \ref{rem:A5A6}), we have
  $$
    \sum_{e\in\partial K\cap\partial\Omega_h} \|v\|_{0,M_e}^2 \lesssim \sum_{e\in\partial K\cap\partial\Omega_h} s^3 \|v\|_{L^\infty(\calB_K')}^2 \\
    \lesssim \frac{h}{s} s^3 \|v\|_{L^\infty(\calB_K')}^2.
    $$
    Finally, using Assumption \textbf{A1}, a scaling argument and the fact that $s\le h \le O(1)$, we have
    $$
    \|v\|_{L^\infty(\calB_K')}^2 \lesssim \|v\|_{L^\infty(\calB_K)}^2 \lesssim h^{-2} \|v\|_{0,\calB_K}^2 \le h^{-2} \|v\|_{0,K}^2.
    $$
    This completes the proof of the lemma.
\end{proof}

Following \cite{bramble1994robust}, we have
\begin{lemma}\label{lem:OmegaOmegah}
  For $v\in H^1(\Omega)$, one has
\begin{equation}\label{omegaomegah}
\sum_{e\in\calE_h^B} \|v\|_{0, M_e}^2  \lesssim s^2 \|v\|_{1,\Omega}^2.
\end{equation}
Moreover, when $v\in H_0^1(\Omega)$, for each $e\in\calE_h^B$ one has
\begin{equation}\label{omegaomegah1}
     \|v\|_{0, M_e}  \lesssim s^2\|v\|_{1,M_e}, \qquad\qquad
     \|v\|_{0,e} \lesssim s \|v\|_{1,M_e}.
\end{equation}
\end{lemma}
\begin{proof}
Inequality (2.9) in \cite{bramble1994robust} states that
\begin{equation*}
\begin{aligned}
\|v\|_{0, M_e} ^2 \lesssim s^2\|v\|_{0,\we}^2+s^4\|v\|_{1,M_e}^2.
\end{aligned}
\end{equation*}
Taking the summation and applying the trace inequality on $\Omega$, we get (\ref{omegaomegah}).
For $v\in H_0^1(\Omega)$, the first inequality in (\ref{omegaomegah1}) follows immediately from $\|v\|_{0,\partial\Omega}=0$.
The second inequality in (\ref{omegaomegah1}) is exactly inequality (2.11) in \cite{bramble1994robust}.
\end{proof}

\medskip

The following version of the Bramble-Hilbert lemma has been proved in Chapter 4 of  \cite{brenner2008mathematical}.
\begin{lemma} \label{bh}
  (Bramble-Hilbert) For $\e\in\calT_h$ and $v\in H^m(\e)$, there exists an averaged Taylor polynomial $v^m_{\e}$ of
  degree less than or equal to $m-1$ satisfying
$$
\begin{aligned}
|v-v^m_{\e}|_{s,  \e}\lesssim h_{\e}^{m-s}|v|_{m,\e},\quad \text{for}\ s\ =\ 0, \ldots, m. 
\end{aligned}
$$
where the hidden constant in $\lesssim$ may depend on $m$ but not on $h$ or the shape of $\e$.
\end{lemma}

\section{Weak Galerkin discretization}{\label{section3}}

Now we introduce the WG-MFEM discretization for system \eqref{primalproblem}-\eqref{primalcondition}
in a form presented in \cite{chen2016weak}.
On each $\e\in \calT_h$, denote by $P_j(\e_0)$ or $P_j(\e)$ the set of polynomials
with degree less than or equal to $j$. 
Likewise, on each $e \in \calE_h$, let $P_j(e)$ be the set of polynomials of degree
no more than $j$. We define the weak Galerkin spaces
$$
\begin{aligned}
\bm{W}_h =\big\{&\vv = \{\vv_0, \vv_b\} \in [L^2(\Omega_h)]^2 \times [L^2(\calE_h)]^2 \text{ such that} &\\
&\vv_0|_{\e_0} \in [P_{\alpha} (\e_0)]^2 \text{ for } \e \in \calT_h,&\\
&\vv_b|_e = v_b\vn_e \text{, where } v_b \in P_\beta (e) \text{, for } e \in \calE_h\big\},&\\
\vV_h =\big\{&\vv\in \bm{W}_h \textrm{ satisfying } \vv_b|_e = 0 \text{ for } e \in \calE_h^B\big\},&
\end{aligned}
$$
where  $\alpha$ and $\beta$ are given non-negative integers and $\vn_e$ is a prescribed normal direction on each edge
$e\in\calE_h$. 
A key feature of the weak Galerkin discretization is that a function $\vv\in \vV_h$ takes separate
values $\vv_0$ on the interior of each $\e$ and $\vv_b$ on edges.
Define
$$
\Psi_h = \big\{q \in L^2_0(\Omega_h) \text{ such that }  q|_\e \in P_{\sigma} (\e)\ \text{for}\ \e \in \calT_h \big\},
$$
where $\sigma$ is a given non-negative integer. Moreover, assume that
\begin{equation} \label{eq:alphabetagamma}
\beta-1 \leq \sigma \leq \beta = \alpha.
\end{equation}
Condition \eqref{eq:alphabetagamma} is imposed to make sure that the weak Galerkin discretization has desired stability and approximation properties,
as will become clear in the analysis to be given later.

On each $\e\in\calT_h$, define the weak divergence $\nabla_w\cdot\vv\in P_\beta(\e)$
for $\vv\in\vV_h$ by 
$$
(\nabla_w\cdot\vv,q)_{\e}= -(\vv_0,\nabla q)_{\e}+\langle\vv_b\cdot\vn,q\rangle_{\partial\e}\qquad \forall\, q\in P_\beta(\e).
$$

Define the discrete bilinear forms $a_h:\: \vV_h\times\vV_h\to \bbR$ and $b_h:\: \vV_h\times\Psi_h\to \bbR$ by
$$
\begin{aligned}
a_h(\vu, \vv) &= \left( \vu_0, \vv_0\right)_{\Omega_h}+\rho \sum_{K \in \calT_h} h_K^{-1}\left\langle\left(\vu_0-\vu_b\right) \cdot \vn,\left(\vv_0-\vv_b\right) \cdot \vn\right\rangle_{\partial K},&\\
b_h(\vv, q) &= -(\nabla_w\cdot \vv, q)_{\Omega_h},&\\
\end{aligned}
$$
in which $\rho$ is a positive constant. 
\begin{remark}
  The second part in $a_h(\cdot,\cdot)$ is a stabilization term. However, different from the discontinuous Galerkin method,
  the stabilization parameter $\rho$ in the weak Galerkin discretization can be chosen arbitrarily without affecting the approximation results.
  In practice, one can simply set $\rho=1$.
  For convenience, we denote the stabilization part by
 \begin{equation}  \label{eq:stabterm}
   c(\vu, \vv) = \rho \sum_{K \in \calT_h} h_K^{-1}\left\langle\left(\vu_0-\vu_b\right) \cdot \vn,\left(\vv_0-\vv_b\right) \cdot \vn\right\rangle_{\partial K}.
   \end{equation}
\end{remark}

The weak Galerkin formulation for system \eqref{primalproblem}-\eqref{primalcondition} can now be written as follows: {\it Find
$\vu_h \in \vV_h$ and $p_h \in \Psi_h$ such that}
\begin{equation}
\begin{cases}
a_h(\vu_h, \vv) + b_h(\vv, p_h)=0\quad &\forall\, \vv\in \vV_h,\\
b_h(\vu_h, q) =-(g, q)_{\Omega_h}\quad &\forall\, q \in \Psi_h. \label{dispro}
\end{cases}
\end{equation}
\begin{remark} \label{rem:discretecompat}
  Recall that for the continuous problem \eqref{primalproblem}-\eqref{primalcondition} to be well-posed, a compatibility condition \eqref{eq:compat} is necessary
  which requires $g$ to be mean-value free on $\Omega$.
  One may wonder whether the discrete problem requires a similar compatibility condition or not.
  We point out that the compatibility mechanism works differently for finite-dimensional problems
  \wangI{and the key is to have
    \begin{equation} \label{eq:bhcompatibility}
      b_h(\vv, 1)=0\qquad \forall\, \vv\in \vV_h,
    \end{equation}
  which is obviously true in our case according to the definitions of $b_h(\cdot,\cdot)$ and $\nabla_w\cdot$}.
  Note that $(1,q)_{\Omega_h}=0$ for all $q \in \Psi_h$. Hence the second equation in \eqref{dispro} is the same as
  \begin{equation} \label{eq:discretecompat}
  b_h(\vu_h, q) =-(\tilde{g}, q)_{\Omega_h}\quad \forall q \in \Psi_h,
  \end{equation}
  where $\tilde{g} = g-\frac{1}{|\Omega_h|}\int_{\Omega_h} g\dx$ is mean-value free on $\Omega_h$.
  Equation \eqref{eq:discretecompat} is now `compatible' in the traditional sense as it also holds for $q\equiv const$.
  From the theoretical point of view, there is no difference should one choose to use \eqref{dispro} or to replace its second equation by \eqref{eq:discretecompat}.

  For finite-dimensional problems, what the `compatibility' condition may actually affect is the implementation procedure.
  In practice, it is not convenient to compute a basis for $\Psi_h$, which needs to be mean-value free on $\Omega_h$.
  One usually drops the mean-value free condition while at the same time expecting the resulting stiffness matrix $M$ to have a rank $1$ deficiency
  \wangI{(under the discrete inf-sup condition to be proved later)}.
  Now, if the `compatible' equation \eqref{eq:discretecompat} is used in the implementation, it guarantees that the right-hand side vector of the linear system
  is orthogonal to $ker(M^T)$. Hence the system is solvable, i.e., compatible.
  Using elementary linear algebra, one immediately sees that even if the `non-compatible' second equation of \eqref{dispro} is used in the implementation,
  it just yields a right-hand side vector not orthogonal to $ker(M^T)$.
  In this case, we know that the kernel consists of exactly $q\equiv const$. Therefore a pure algebraic post-process after assembling the entire linear system,
  i.e., making the right-hand vector orthogonal to $ker(M^T)$, can easily \wangI{resolve} this issue and render the linear system `compatible'.

  Due to the above explanation, we do not need to worry about the `compatibility' of the discrete system
  either theoretically or in the implementation, \wangI{as long as \eqref{eq:bhcompatibility} holds}.
\end{remark}

To analyze the well-posedness and approximation properties of the weak Galerkin discretization (\ref{dispro}),
we first define the norm on $\vV_h$ as follows:
$$
\begin{aligned}
\|\vv\|_{\vV_h}=&\left(\| \vv_0\|_{\Omega_h}^{2}+\rho \sum_{\e \in \calT_h} h_{\e}^{-1}\left\|\left(\vv_0-\vv_b\right) \cdot \vn\right\|_{\partial \e}^{2}\right)^{1 / 2}\quad \forall\, \vv\in\vV_h.
\end{aligned}
$$
It is obvious that $a_h(\vv,\vv)=\|\vv\|_{\vV_h}^2$ for all $\vv\in\vV_h$.
We shall show that $\|\cdot\|_{\bm{V_h}}$ is a well-defined norm. 
Indeed, if $\|\vv\|_{\vV_h}=0$ for some $\vv\in\vV_h$; i.e.,
$$
\left( \vv_0, \vv_0\right)_{\Omega_h}+\rho \sum_{K \in \calT_h} h_K^{-1}\left\langle\left(\vv_0-\vv_b\right) \cdot \vn,\left(\vv_0-\vv_b\right) \cdot \vn\right\rangle_{\partial K}=0,$$
one has $\vv_0\equiv\vzero$ on each element $\e$ and $(\vv_0-\vv_b)\cdot\vn=0$ on each edge $e\in\calE_h$.
This leads to $0=\vv_b\cdot\vn= v_b\vn_e\cdot\vn$  and consequently $\vv_b\equiv\vzero$ on each edge $e\in\calE_h$.
Therefore $\|\cdot\|_{\bm{V_h}}$ is a norm on $\vV_h$.

On each $K \in \calT_h$, denote by $\vQ_0,\ \bbQ_h$ and $\pi_h$ the $L^{2}$ projections
onto $\left[P_{\alpha}(K)\right]^{2}$, $P_{\sigma}(\e)$ and $P_{\beta}(K)$, respectively.
On each $e\in \calE_h$, denote by $Q_b$ the $L^{2}$ projection onto $P_{\beta}(e)$.
On the entire $\Omega_h$, we use the same notation to denote the combination of the above local projections.
For $\vv\in [H^1(\Omega_h)]^2$, define $\vQ_b\vv=(Q_b(\vv\cdot\vn_e))\vn_e$ for each $e\in \calE_h$.
Then we define a projection $\vQ_h :\: [H^1(\Omega_h)]^2\to \bm{W}_h$ by 
$$
  \vQ_h\vv= \{\vQ_0\vv, \vQ_b\vv\}.
$$
One has 
\begin{lemma}\label{Projection} (Lemma 3.12 in \cite{chen2016weak}) Assume that $\alpha$ and $\beta$ satisfy \eqref{eq:alphabetagamma}, the following commutative property holds:
\begin{equation} \label{eq:commutative}
\begin{aligned}
\nabla_{w} \cdot\left(\vQ_h \vv\right) &=\pi_h(\nabla \cdot \vv) \qquad \forall\, \vv \in [H^1(\Omega_h)]^2.
\end{aligned}
\end{equation}
\end{lemma}

The above commutative property is one of the key features of the weak Galerkin discretization
and has played an important role in its theoretical analysis \cite{chen2016weak, wang2014weak}.
However, the situation is quite different when $\Omega$ is a curved domain,
because for $\vv\in [H^1(\Omega)]^2$ satisfying $\vv\cdot\vtn|_{\partial\Omega} = 0$
one does not get $\vQ_h\vv \cdot \vn|_{\partial\Omega_h}=0$. 
The projection $\vQ_h$ only maps $\vv$ into $\bm{W}_h$, but not $\vV_h$.

To partly remedy this, we introduce 
$$
\vtQ_b\vv = 
\begin{cases}
  \vzero \quad & \textrm{on }e\in \calE_h^B,\\
  \vQ_b \vv & \textrm{on all other edges},
\end{cases}
$$
and define a modified projection $\vtQ_h:\: [H^1(\Omega_h)]^2\to \vV_h$ by
$$
\vtQ_h \vv= \{\vQ_0\vv, \vtQ_b\vv\}.
$$
Obviously, the price to pay is that $\vtQ_h$ no longer satisfies the highly desired commutative property \eqref{eq:commutative}.
This brings a lot trouble to the theoretical analysis of the weak Galerkin formulation \eqref{dispro}, as we will see later.

Using the Bramble-Hilbert lemma \ref{bh}, we easily get the following result:
\begin{lemma}\label{Qh}
Let $\vu\in [H^{r+1}(\Omega)]^2$ and $p\in H^{t+1}(\Omega)$, where $0\leq r\leq \alpha$ and $0\leq t\leq\sigma$. Then
\begin{equation*}
\begin{aligned}
\sum_{\e\in\calT_h}\|\vu-\vQ_0\vu\|^2_{ \e}+\sum_{\e\in\calT_h}h_{\e}^2\|\nabla(\vu-\vQ_0\vu)\|^2_{ \e}\lesssim h^{2(r+1)}|\vu|^2_{r+1,\Omega_h},\\
\sum_{\e\in\calT_h}\|p-\bbQ_hp\|^2_{ \e}+\sum_{\e\in\calT_h}h_{\e}^2\|\nabla(p-\bbQ_hp)\|^2_{ \e}\lesssim h^{2(t+1)}|p|^2_{t+1,\Omega_h}.
\end{aligned}
\end{equation*}
\end{lemma}

\subsection{ Existence and uniqueness of the discrete solution}\label{subsection: solution}

Since $a_h(\cdot, \cdot)$ is indeed an inner-product on $\vV_h$,
according to the standard theory of mixed finite element methods \cite{boffi2013mixed},
System (\ref{dispro}) admits a unique solution as long as the following discrete inf-sup condition holds:
\begin{equation}
\sup _{\vv \in \vV_h} \frac{|b_h(\vv,q)|}{\|\vv\|_{\vV_h}} \gtrsim\|q\|_{0,\Omega_h}\qquad \forall\, q \in \Psi_h.\label{inf-sup}
\end{equation}
The standard way to prove the discrete inf-sup condition is to use the continuous inf-sup condition together with a stable projection onto the discrete space,
as is the case in \cite{wang2014weak}. However, following this strategy is not easy on domains with curved boundary.
Below we shall explain why. Firstly, the continuous inf-sup condition must be defined on $\Omega$ instead of $\Omega_h$
so that the constant in it does not depend on $h$ or $s$. However, a function $q\in \Psi_h$ is mean value free on $\Omega_h$ but not necessarily on $\Omega$,
and hence can not be used directly in the continuous inf-sup condition.
Secondly, as pointed out earlier, the projection $\vQ_h$ satisfies the commutative property \eqref{eq:commutative} but only maps functions into $\bm{W}_h$,
while $\vtQ_h$ maps functions into $\vV_h$ at the price of violating the commutative property.
Neither of them can serve directly in a traditional proof of the discrete inf-sup condition.
In fact, the proof of the inf-sup condition (\ref{inf-sup}) turns out to be non-trivial:

\begin{lemma} \label{lbb}The discrete $\inf$-$\sup$ condition (\ref{inf-sup}) holds when $h$ is sufficiently small.
\end{lemma}
\begin{proof}
  Each $q \in \Psi_h$ is a function defined on $\Omega_h$. 
  Note that it can be naturally extended to $\Omega$ by filling the gap $\Omega\backslash\Omega_h$ with the same polynomial values on neighboring mesh elements.
  For simplicity, we still denote this extension by $q$.
  Now $q$ is mean-value free on $\Omega_h$. Define $\bar{q}=q-\frac{1}{|\Omega|}\int_{\Omega\setminus\Omega_h} q \dx$.
  Obviously $\bar{q}$ is mean-value free on $\Omega$. Moreover, by the triangle inequality, Lemma \ref{lem:inv},
  and the fact that $s\le h \le O(1)$, one has
\begin{equation*}
\begin{aligned}
\|\bar{q}\|_{0,\Omega}^2&= 
 \bigg(\|q\|_{0,\Omega_h}^2+\sum_{e\in\calE_h^B}\|q\|_{0,M_e}^2\bigg)+\frac{1}{|\Omega|}\left(\int_{\Omega\setminus\Omega_h} q\dx\right)^2\\
&\lesssim \|q\|_{0,\Omega_h}^2+h^{-1}s^2\|q\|_{0,\Omega_h}^2+|\Omega\backslash\Omega_h| \int_{\Omega\backslash\Omega_h} |q|^2 \dx\\
&= \|q\|_{0,\Omega_h}^2+h^{-1}s^2\|q\|_{0,\Omega_h}^2+ |\Omega\backslash\Omega_h| \sum_{e\in\calE_h^B}\|q\|_{0,M_e}^2\\
&\lesssim \|q\|_{0,\Omega_h}^2+h^{-1}s^2\|q\|_{0,\Omega_h}^2\lesssim \|q\|_{0,\Omega_h}^2.\\
\end{aligned}
\end{equation*}

It can be shown (see, e.g.,\cite{girault2012finite}) that  there exists a $\bm{w} \in\left[H_0^{1}(\Omega)\right]^{2}$ such that $\nabla \cdot \bm{w}=\bar{q}$
and $\|\bm{w}\|_{1,\Omega} \lesssim\|\bar{q}\|_{0,\Omega} \lesssim \|q\|_{0,\Omega_h}$.
By Lemma \ref{Projection} and the fact that $ \sigma \le \beta,$ we have
$$
\left(\nabla_{w} \cdot \vQ_h\bm{w}, q\right)_{\Omega_h}=\left(\pi_h(\nabla \cdot \bm{w}), q\right)_{\Omega_h}=(\nabla \cdot \bm{w}, q)_{\Omega_h}=(\bar{q},q)_{\Omega_h}=\|q\|_{0,\Omega_h}^{2}.
$$
Combine the above equation with the definitions of $\vQ_h$, $\vtQ_h$ and $\nabla_w\cdot$, one gets
$$
\begin{aligned}
\left|\left(\nabla_{w} \cdot \vtQ_h\bm{w}, q\right)_{\Omega_h} \right|
&=\left| \left(\nabla_w \cdot \vQ_h\bm{w}, q\right)_{\Omega_h}-\sum_{\e\in\calT_h^B}\langle \vQ_b\bm{w}\cdot\vn,q\rangle_{\partial\e\cap\partial\Omega_h} \right|\\
&\geq \|q\|_{0,\Omega_h}^{2}-\left( \sum_{K \in \calT_h^B} h_K^{-1}\|\vQ_b \bm{w} \cdot \vn\|_{0,\partial K\cap\partial\Omega_h}^{2}\right)^{1 / 2}
\left( \sum_{K \in \calT_h^B} h_K\|q\|_{0,\partial K\cap\partial\Omega_h}^{2}\right)^{1 / 2}.
\end{aligned}
$$

Check the right-hand side of the above inequality. First, by the trace inequality and the inverse inequality
(lemmas \ref{lem:trace}-\ref{inv}), it is clear that
\begin{equation} \label{eq:q}
\left( \sum\limits_{K \in \calT_h^B} h_K\|q\|_{0,\partial K\cap\partial\Omega_h}^{2}\right)^{1 / 2}\lesssim\|q\|_{0,\Omega_h}.
\end{equation}

Next, using $\bm{w}|_{\partial\Omega}=\vzero$ and Lemma \ref{lem:OmegaOmegah}, we have
\begin{equation}\label{w}
\begin{aligned}
\left(\sum_{K \in \calT_h^B} h_K^{-1}\|\vQ_b \bm{w} \cdot \vn\|_{0,\partial K\cap\partial\Omega_h}^{2}\right)^{1 / 2} & \le \left( \sum_{K \in \calT_h^B} h_K^{-1}\|\bm{w} \cdot \vn\|_{0,\partial K\cap\partial\Omega_h}^{2}\right)^{1 / 2}\\
   & \lesssim h^{-\frac{1}{2}}s \|\bm{w}\|_{1,\Omega}
\lesssim h^{-\frac{1}{2}}s\|q\|_{0,\Omega_h}.
\end{aligned}
\end{equation}
Hence when $h$ is sufficiently small,
\begin{equation*}
\begin{aligned}
\left|(\nabla_w\cdot \vtQ_h\bm{w} ,q) \right|\ge \left(1-O(h^{-\frac{1}{2}}s) \right)\|q\|^2_{0,\Omega_h} \ge \left(1-O(h^{\frac{1}{2}}) \right)\|q\|^2_{0,\Omega_h}\gtrsim \|q\|^2_{0,\Omega_h}.
\end{aligned}
\end{equation*}
Using the fact that $\vQ_0 \vw = \vQ_b\vQ_0 \vw$ and the approximation property of $\vQ_0$, one gets
$$
\begin{aligned}
h_K^{-1}\left\|\left(\vQ_0 \bm{w}-\vQ_b \bm{w}\right) \cdot \vn\right\|_{0,\partial K}^{2} & \le h_K^{-1}\left\|\vQ_0 \bm{w}-\vQ_b \bm{w}\right\|_{0,\partial K}^{2} \\
& \le h_K^{-1}\left\|\vQ_0 \bm{w}-\bm{w}\right\|_{0,\partial K}^{2} \\
& \lesssim\|\nabla \bm{w}\|_{0,K}^{2}.
\end{aligned}
$$
We then have
$$
\begin{aligned}
\| \vtQ_h\bm{w}\|_{\vV_h}=&\left(
\|\vQ_0 \bm{w}\|_{0,\Omega_h}^{2}+\rho \sum_{K \in \calT_h} h_K^{-1}\|\left(\vQ_0 \bm{w}-\vtQ_b \bm{w}\right) \cdot \vn\|_{0,\partial K}^{2}\right)^{1 / 2}\\
\leq&\left(
\|\vQ_0 \bm{w}\|_{0,\Omega_h}^{2}+\rho \sum_{K \in \calT_h} h_K^{-1}\|\left(\vQ_0 \bm{w}-\vQ_b \bm{w}\right) \cdot \vn\|_{0,\partial K}^{2}\right)^{1 / 2}\\
&+\left(\rho \sum_{K \in \calT_h^B} h_K^{-1}\|\vQ_b \bm{w} \cdot \vn\|_{0,\partial K\cap\partial\Omega_h}^{2}\right)^{1 / 2}\\
&\lesssim \|\vw\|_{1,\Omega} + h^{-\frac{1}{2}}s\|\vw\|_{1,\Omega} \\
& \lesssim \|q\|_{0,\Omega_h}.
\end{aligned}
$$

Combining the above gives
$$
\sup _{\vv \in \vV_h} \frac{|b_h(\vv,q)|}{\|\vv\|_{\vV_h}} \ge \frac{|b_h(\vtQ_h\vw,q)|}{\|\vtQ_h\vw\|_{\vV_h}} \gtrsim \|q\|_{0,\Omega_h}.
$$
This completes the proof of the lemma.
\end{proof}

By the standard theory of the mixed finite elements \cite{boffi2013mixed}, we know that system \eqref{dispro} admits a unique solution.
Moreover, the discrete inf-sup condition ensures that the discrete operator in (\ref{dispro}) is stable in the sense that that unique solution to
$$
\begin{cases}
a_h(\vu_h, \vv) + b_h(\vv, p_h)=F(\vv) \quad &\forall \vv \in \vV_h,\\
b_h(\vu_h, q)  =G(q) \quad &\forall q \in \Psi_h, 
\end{cases}
$$
satisfies
\begin{equation} \label{eq:stability}
  \|\vu_h\|_{\vV_h} + \|p_h\|_{\Psi_h} \lesssim \|F\|_{\vV_h'} + \|G\|_{\Psi_h'}.
\end{equation}
The stability result \eqref{eq:stability} is essential to the error analysis to be given next.

\subsection{Error analysis}

To analyze the approximation error of the weak Galerkin discretization \eqref{dispro},
we first give two lemmas that help to simplify the derivation of error equations.
The proof of these two lemmas use quite standard techniques and will be given in Appendix \ref{appendix1}.
\begin{lemma}\label{lemma1}
The solution $\vu$ and $p$ to system (\ref{primalproblem})-\eqref{primalcondition} satisfy
$$
a_h\left(\vtQ_h\vu, \vv\right)+b_h\left(\vv,  \bbQ_hp\right)=c(\vQ_h \vu, \vv)+l_s(\vv) -l_{\mathrm{div}}(\vv)\qquad\forall\, \vv \in \vv_h,
$$
where $c(\cdot,\cdot)$ is the stabilization term defined in \eqref{eq:stabterm}
and the linear functionals  $l_{s}(\cdot)$ and $l_{\mathrm{div}}(\cdot)$ are defined by
$$
\begin{aligned}
l_s(\vv) &=\rho \sum_{K \in \calT_h^B} h_K^{-1}\left\langle\vu \cdot \vn,\left(\vv_0-\vv_b\right) \cdot \vn\right\rangle_{\partial K\cap\partial\Omega_h}, \\
l_{\mathrm{div}}(\vv) &=\sum_{K \in \calT_h}\langle(\vv_0-\vv_b) \cdot \vn, p-\bbQ_h p\rangle_{\partial K}.
\end{aligned}
$$
\end{lemma}

\begin{lemma}\label{lemma2}
The solution $\vu$ to problem (\ref{primalproblem})-\eqref{primalcondition} satisfies
$$
b_h\left( \vtQ_h  \vu, q\right)=-(g, q)_{\Omega_h} + l_b(q)\qquad\forall\,q \in \Psi_h,
$$
where the linear functional  $l_b(\cdot)$ is defined by
$$l_b(q)=\sum_{\e\in\calT_h^B}\langle  {\vu}\cdot {\vn},q\rangle_{\partial\e\cap\partial\Omega_h} .$$
\end{lemma}

With the aid of lemmas \ref{lemma1}-\ref{lemma2}, we can easily derive the error equations.
Let $(\vu, p)$ be the solution to problem (\ref{primalproblem})-\eqref{primalcondition},
and $(\vu_h,p_h)$ be the solution to the weak Galerkin discretization (\ref{dispro}).
Define 
\begin{equation} \label{eq:errdef}
 \ve_{ {\vu}}=\vtQ_h  {\vu}-\vu_h=\left\{\vQ_0  {\vu}- {\vu}_0, \vtQ_b  {\vu}- {\vu}_b\right\},\quad e_{ p}=\bbQ_h  p-p_h.
\end{equation}
Using lemmas \ref{lemma1}-\ref{lemma2}, we clearly have
\begin{equation}\label{erroreuqation1}
\begin{cases}
a_h\left(\ve_{ {\vu}}, \vv\right)+b_h\left(\vv, e_{ p}\right) = c(\vQ_h \vu, \vv)+l_s(\vv)-l_{\text{div}}(\vv)  \qquad &\forall\,  \vv \in \vV_h, \\
b_h\left(\ve_{ {\vu}}, q\right) = l_b(q) \quad &\forall\, q \in \Psi_h.
\end{cases}
\end{equation}

We shall first derive the upper bounds for the right-hand side of (\ref{erroreuqation1}).

\begin{lemma} \label{Berize}
  Let $u\in \vV\cap [H^{ r+1}(\Omega)]^2$ and $p\in\Psi\cap H^{t+1}(\Omega)$ be the solution to problem (\ref{primalproblem})-\eqref{primalcondition},
  where $0 \le r\le \alpha$ and $0 \le t \le \sigma$.
  Then
$$
\begin{aligned}
|c( {\vQ}_h\bm{ {u}},\vv)| &\lesssim  h^{r}\|\vu\|_{r +1,\Omega}\|\vv\|_{\vV_h} \quad &&\forall\, \vv\in\vV_h,\\
|l_{\operatorname{div}}(\vv)| &\lesssim h^{t+1}\|p\|_{t+1,\Omega}\|\vv\|_{\vV_h} \quad &&\forall\, q\in \Psi_h.
\end{aligned}
$$
\end{lemma}
\begin{proof}
Note that $\vQ_0\vu = \vQ_b(\vQ_0\vu)$ on each edge in $\calT_h$.
Hence by the Cauchy-Schwarz inequality, the trace inequality (Lemma \ref{lem:trace}) and the approximation property of $\vQ_0$, we have
\begin{equation} \label{eq:sbound}
\begin{aligned}
|c\left(\vQ_h  \vu, \vv\right)| &= \left|\rho  \sum_{K \in \calT_h } h_K^{-1}\left\langle\left(\vQ_0  \vu -\vQ_b \vu \right) \cdot \vn,\left( \vv_0-\vv_b\right) \cdot \vn\right\rangle_{\partial K}\right| \\
&\le \rho\left(\sum_{K \in \calT_h} h_K^{-1}\left\|\vQ_0  \vu- \vQ_b\vu\right\|_{0,\partial K}^{2}\right)^{1 / 2}
     \left(\sum_{K \in \calT_h} h_K^{-1}\left\| (\vv_0-\vv_b)\cdot\vn \right\|_{0,\partial K}^{2}\right)^{1 / 2} \\
&\le \rho\left(\sum_{K \in \calT_h} h_K^{-1}\left\|\vQ_0  \vu- \vu\right\|_{0,\partial K}^{2}\right)^{1 / 2}
     \left(\sum_{K \in \calT_h} h_K^{-1}\left\|\left( \vv_0-\vv_b\right) \cdot \vn\right\|_{0,\partial K}^{2}\right)^{1 / 2} \\
&\lesssim  h^{r}\|\vu\|_{r+1, \Omega}\|\vv\|_{\vV_h}.
\end{aligned}
\end{equation}

Similarly, the upper bound for $l_{\mathrm{div}}(\vv)$ follows directly from Lemma \ref{lem:trace} and the approximation property of $\bbQ_h$, that is,
$$
\begin{aligned}
|l_{\mathrm{div}}(\vv)| &= \left|\sum_{K \in \calT_h}\langle(\vv_0-\vv_b) \cdot \vn,  p-\bbQ_h  p\rangle_{\partial K}\right| \\
&\lesssim\left(\sum_{K \in \calT_h} h_K^{-1}\left\|(\vv_0-\vv_b) \cdot \vn\right\|_{0,\partial K}^{2}\right)^{1 / 2}h^{t +1}\| p\|_{t +1, \Omega_h}\\
&\lesssim h^{t+1}\|p\|_{t+1, \Omega}\|\vv\|_{\vV_h}.
\end{aligned}
$$
This completes the proof of the lemma.
\end{proof}

\begin{lemma}\label{lemma3.13}
 Assume that $\vu\in \vV\cap [H^2(\Omega)]^2$ is the solution to problem (\ref{primalproblem})-\eqref{primalcondition}, then 
 \begin{equation}\label{lbq}
\begin{aligned}
|l_b(q)| &\lesssim h^{-\frac{3}{2}} s^2 \|\vu\|_{2,\Omega}\|q\|_{0,\Omega_h}&\forall\,  q\in\Psi_h,\\
|l_s(\vv)| &\lesssim h^{-\frac{1}{2}} s \|\vu\|_{2,\Omega}\|\vv\|_{\vV_h}&\forall\, \vv\in\vV_h.\\
\end{aligned}
\end{equation}
\end{lemma}

\begin{figure}[H]
\begin{center}
\includegraphics[width=1.5in]{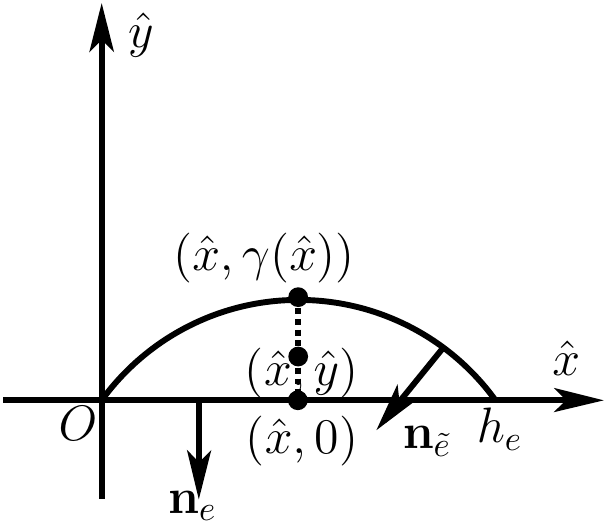}\\
\caption{$M_e$ in the local coordinate system.}
\label{maping2}
\end{center}
\end{figure}

\begin{proof}
  For each $e\in\calE_h^B$, the crescent-shaped region $M_e$ is surrounded by the straight edge $e$ and the curved edge $\tilde{e}$.
  Let $\vn_{e}$ and $\vn_{\we}$ be unit normal vectors on $e$ and $\we$, respectively, with directions shown in Figure \ref{maping2}.
  It is just for convenience that we draw both $\vn_{e}$ and $\vn_{\we}$ as downward pointing, and the direction has no affect to the final result.
  
  As mentioned in the beginning of the proof of Lemma \ref{lbb}, a function $q\in\Psi_h$ can be naturally extended to a function defined on $\Omega$,
  which is still denoted by $q$ for simplicity.
  Using Lemma \ref{lem:inv} and the inverse inequality we get
  $$
  \sum_{e\in\calE_h^B}\|q\|_{0,M_e}^2 \lesssim h^{-1}s^2\|q\|_{0,\Omega_h}^2, \qquad
  \sum_{e\in\calE_h^B}\|\nabla q\|_{0,M_e}^2 \lesssim h^{-3}s^2\|q\|_{0,\Omega_h}^2.
  $$
Combining the above and noticing that the solution $\vu$ to problem (\ref{primalproblem}) satisfies $\vu\cdot\vn_{\we}|_{\we}=0$, we have
$$
\begin{aligned}
  |l_b(q)| = \left|-\sum_{e\in\calE_h^B} \langle \vu\cdot\vn_e, q \rangle_{e} \right| 
    &=\left|-\sum_{e\in\calE_h^B} \bigg( \langle \vu\cdot\vn_e, q \rangle_{e}  - \langle \vu\cdot\vn_{\we}, q \rangle_{\we} \bigg)\right| \\
    &=\left|-\sum_{e\in\calE_h^B} \bigg( (\nabla\cdot\vu, q)_{M_e} + (\vu, \nabla q)_{M_e} \bigg)\right| \\
    &\leq \|\nabla\cdot\vu\|_{0,\Omega\setminus\Omega_h} \|q\|_{0,\Omega\setminus\Omega_h} + \|\vu\|_{0,\Omega\setminus\Omega_h} \|\nabla q\|_{0,\Omega\setminus\Omega_h}\\
    &\lesssim s\|\vu\|_{2,\Omega}\, h^{-\frac{1}{2}}s\|q\|_{0,\Omega_h}  + s\|\vu\|_{1,\Omega} \, h^{-\frac{3}{2}}s\|q\|_{0,\Omega_h} \\
    &\lesssim h^{-\frac{3}{2}}s^2\|\vu\|_{2,\Omega}\|q\|_{0,\Omega_h},
\end{aligned}
$$
where in the second last step we have used Lemma \ref{lem:OmegaOmegah}.
This completes the proof of the first inequality in \eqref{lbq}.

For any $\vv=\{\vv_0,\vv_b\}\in\vV_h$, note that $((\vv_0-\vv_b)\cdot\vn_e)|_{e}$ is a polynomial of $\hat{x}$
in the local coordinate system $\hat{x}$-$\hat{y}$ as shown in Figure \ref{maping2}.
This polynomial form can be viewed as the restriction of a two-variable polynomial $((\vv_0-\vv_b)\cdot\vn_e)^*$, which depends only on $\hat{x}$, on edge $e$.
Then
$$
\begin{aligned}
  \sum_{e\in\calE_h^B}\|((\vv_0-\vv_b)\cdot\vn_e)^*\|_{0,M_e}^2
  &= \sum_{e\in\calE_h^B}  \int_0^{h_e} \int_0^{\gamma(\hat{x})} |((\vv_0-\vv_b)\cdot\vn_e)^*|^2 \dhy\dhx\\
  &\lesssim \sum_{e\in\calE_h^B}  \int_0^{h_e} s^2 |(\vv_0-\vv_b)\cdot\vn_e|^2 \dhx\\
  &\lesssim h s^2 \|\vv\|_{\vV_h}^2.
\end{aligned}
$$
Similarly, 
$$
\begin{aligned}
\sum_{e\in\calE_h^B}\|\nabla((\vv_0-\vv_b)\cdot\vn_e)^*\|_{0,M_e}^2
&= \sum_{e\in\calE_h^B} \int_0^{h_e} \int_0^{\gamma(\hat{x})}  |\frac{\partial ((\vv_0-\vv_b)\cdot\vn_e)^*}{\partial \hat{x}}|^2 \dhy\dhx \\
&\lesssim \sum_{e\in\calE_h^B}s^2|(\vv_0-\vv_b)\cdot\vn_e|^2_{1,e} \\
&\lesssim \sum_{e\in\calE_h^B} \|(\vv_0-\vv_b)\cdot\vn_e\|^2_{0,e} \\
&\lesssim h\|\vv\|_{\vV_h}^2.
\end{aligned}
$$

Combining the above and using the fact that $\vu\cdot {\vn}_{\we}|_{\we}=0$ together with Lemma \ref{lem:OmegaOmegah}, we have
$$
\begin{aligned}
|l_s(\vv)|&= \left|\rho \sum_{\e\in\calT_h^B} h_K^{-1} \langle  \vu\cdot \vn, (\vv_0-\vv_b)\cdot\vn\rangle_{\partial K\cap \partial\Omega_h} \right|\\
 &\lesssim h^{-1}{\sum_{e\in\calE_h^B} \bigg| \langle\vu\cdot {\vn}_e,(\vv_0-\vv_b)\cdot\vn_e\rangle_{e} -\langle\vu\cdot {\vn}_{\we},((\vv_0-\vv_b)\cdot\vn_e)^*\rangle_{\we}} \bigg|\\
&=h^{-1}{\sum_{e\in\calE_h^B} \bigg| (\nabla\cdot{\vu},((\vv_0-\vv_b)\cdot\vn_e)^*)_{M_e} + ({\vu},\nabla ((\vv_0-\vv_b)\cdot\vn_e)^*)_{M_e}} \bigg|\\
& \leq h^{-1} \bigg( s\|\vu\|_{2,\Omega} \, h^{\frac{1}{2}} s \|\vv\|_{\vV_h} + s\|\vu\|_{1,\Omega} \,h^{\frac{1}{2}} \|\vv\|_{\vV_h}  \bigg) \\
&\lesssim h^{-\frac{1}{2}} s\|\vu\|_{2,\Omega}\|\vv\|_{\vV_h}.
\end{aligned}
$$
This completes the proof of the lemma.
\end{proof}

Now, we are able to present the error estimate:
\begin{theorem} \label{th1}  Let $\vu$ and $p$ satisfy the conditions in Lemma \ref{Berize} and Lemma \ref{lemma3.13}.
  The error $\ve_{\vu}$ and $e_p$ satisfy
$$
\left\|\ve_{ {\vu}}\right\|_{\vV_h}+\left\|e_{ p}\right\|_{0,\Omega_h} \lesssim h^r\|\vu\|_{r+1, \Omega} + h^{t+1}\|p\|_{t+1, \Omega } + h^{-\frac{1}{2}} s\|\vu\|_{2,\Omega},
$$
where $0 \le r\le \alpha$ and $0 \le t \le \sigma$.
\end{theorem}
\begin{proof} 
  The theorem follows immediately from the mixed finite element theory,
  Inequality (\ref{eq:stability}), the error equation (\ref{erroreuqation1}), lemmas \ref{Berize}-\ref{lemma3.13},
  and the fact that $s\le h$.
\end{proof}

The error bound in Theorem \ref{th1} consists of two parts:
an approximation error $h^r\|\vu\|_{r+1, \Omega} + h^{t+1}\|p\|_{t+1, \Omega}$
and a consistency error $h^{-\frac{1}{2}} s\|\vu\|_{2,\Omega}$.
The consistency error comes purely from the discrepancy between $\Omega$ an $\Omega_h$.
It vanishes if $\Omega=\Omega_h$.
Here we are only interested in the case when $\Omega$ has curved boundary and hence $\Omega\neq\Omega_h$.
To fully examine the effect of the consistency error, we set
$\alpha=\beta \ge 1$, $\sigma=\alpha-1$ and assume that the exact solution satisfies $\vu\in [H^{\alpha+1}(\Omega)]^2$, $p\in H^{\alpha}(\Omega)$,
so that the approximation error reaches its optimal order $O(h^{\alpha})$.
Then by adjusting $s$, we get the following results:
\begin{corollary}\label{lemma3.15}  
  Assuming the polynomial approximation error reaches the optimal $O(h^{\alpha})$ with $\alpha\ge 1$.
  On domains with curved boundary, when $s=O(h)$, the errors $\bm{e}_{\vu}$ and $e_p$ satisfy
$$
\left\|\ve_{ {\vu}}\right\|_{\vV_h}+\left\|e_{ p}\right\|_{0,\Omega_h} \lesssim O(h^{\frac{1}{2}}).
$$
\end{corollary}

Numerical results in Section \ref{section5} will show that the result in Corollary \ref{lemma3.15} is indeed sharp.

\begin{corollary} \label{lemma3.16}
  Assuming that the polynomial approximation error reaches the optimal $O(h^{\alpha})$ with $\alpha\ge 1$.
  On domains with curved boundary, when $s=O(h^{\alpha+\frac{1}{2}})$, the errors $\bm{e}_{\vu}$ and $e_p$ reach an optimal
$$
\left\|\ve_{ {\vu}}\right\|_{\vV_h}+\left\|e_{ p}\right\|_{0,\Omega_h} \lesssim O(h^{\alpha}).
$$
\end{corollary}

When $h$ is small, setting $s=O(h^{\alpha+\frac{1}{2}})$ means to approximate the curved boundary $\partial\Omega$ with
multiple short edges. Numerical results in Section \ref{section5}
will show that for moderate $\alpha$ (our experiments take $\alpha=1,\,2,\,3$), this approach seem to work quite well.


\section{A modified weak Galerkin method}{\label{section4}}
In this section, we construct a modified weak Galerkin discretization for problem (\ref{primalproblem})-\eqref{primalcondition},
\wangI{for which the consistency error improves from $O(h^{-\frac{1}{2}} s)$ to $O(h^{-\frac{1}{2}} s^2)$.
}

Previously we have defined $\vn$ and $\vtn$ as the outward normal on $e\in \calE_h^B$ and $\we$, respectively.
By using the map $\vsig$ as shown in Figure \ref{sigma}, the definition of $\vtn$ can be pulled to $e$ through $\vtn\circ \vsig$.
We still denote this pullback by $\vtn$, i.e., $\vtn$ is now also well-defined on $e\in \calE_h^B$.
For simplicity of notation, on interior edges $e\in \calE_h^I$ we just set $\vtn=\vn$. 
Now we introduce the following modified bilinear forms
\begin{equation}
  \begin{aligned}
    \sh(\vu, \vv) &=\rho \sum_{K \in \calT_h} h_K^{-1}\left\langle\left(\vu_0-\vu_b\right) \cdot \vtn,\left(\vv_0-\vv_b\right) \cdot \vtn\right\rangle_{\partial K}, \\
    \ah(\vu, \vv) &= \left( \vu_0, \vv_0\right)_{\Omega_h} + \sh(\vu, \vv),
\end{aligned}
\end{equation}
 while the definition $b_h(\vv, q) = -(\nabla_w\cdot \vv, q)$ remains unchanged.
\begin{remark}
  The definition of $b_h(\cdot,\cdot)$ depends on $\nabla_w\cdot$ which in turn depends on the choice of outward normal vectors on each $\e\in \calT_h$.
  We emphasize that the definition of $\nabla_w\cdot$ still uses $\vn$ instead of $\vtn$.
  Hence $b_h(\cdot,\cdot)$ remains unchanged.
\end{remark}

The modified weak Galerkin formulation for problem (\ref{primalproblem})-\eqref{primalcondition} reads as follows:
{\it Find $\vu_h \in \vV_h$ and $p_h \in \Psi_h$ such that}
\begin{equation}\label{dispro2}
\begin{cases}
\ah(\vu_h, \vv) + b_h(\vv, p_h)=0 \quad &\forall\, \vv \in \vV_h,\\
b_h(\vu_h, q) -\langle\vu_0\cdot\vn -\overline{\vu_0\cdot\vn } ,q\rangle_{\partial\Omega_h}  =-(g, q)_{\Omega_h}\quad &\forall\, q \in \Psi_h, 
\end{cases}
\end{equation}
where $\langle\vu_0\cdot\vn-\overline{\vu_0\cdot\vn } ,q\rangle_{\partial\Omega_h}$ in the second equation is a boundary correction term with
 \liul{$$\overline{\vu_0\cdot\vn }|_e = \frac{1 }{|e|}\int_{e}\vu_0\cdot\vn \ds,\quad \forall e\in\calE_h^B.$$}

\begin{remark} \label{rem:overline}
  \wangI{The term $\overline{\vu_0\cdot\vn}$ is added for two purposes.
    First, it ensures that the left-hand side of the second equation of \eqref{dispro2} vanishes for $q=const$.
    Then, under the discrete inf-sup condition to be proved later,
    the kernels of both the stiffness matrix for \eqref{dispro2} and its transpose are $q\equiv const$.
    As discussed earlier in Remark \ref{rem:discretecompat},
    one does not need to worry about the `compatibility' issue for the modified problem \eqref{dispro2}.
    Second, it gives an $O(s)$ asymptotic rate to the boundary correction term, which will play an important role in the current proof
    of the discrete inf-sup condition.}
\end{remark}

Similar to the analysis in Section \ref{section3}, its not hard to see that
$$
\|\vv\|_{\nnorm} =\sqrt{\ah(\vv,\vv)}
$$
is a well-defined norm on $\vV_h$.
Moreover, we have
\begin{lemma} \label{lem:normequivance}
  One has $\|\vv\|_{\nnorm}\approx \|\vv\|_{\vV_h}$ for all $\vv\in \vV_h$.
\end{lemma}
\begin{proof}
By the mesh assumptions \textbf{A4} and \textbf{A6},
the trace and the inverse inequality (lemmas \ref{lem:trace}-\ref{inv}),
the facts that $\vv_b=\vzero$ on $\partial\Omega_h$ and each $\e\in\calT_h^B$
  contains at most $O(\frac{h}{s})$ edges in $\calE_h^B$,
one gets
  $$
    \rho \sum_{\e \in \calT_h} h_{\e}^{-1}\left\|\left(\vv_0-\vv_b\right) \cdot (\vtn-\vn)\right\|_{0,\partial \e\cap \partial\Omega_h}^2 \lesssim \|\vv_0\|_{0,\Omega_h}^2.
  $$
    The lemma then follows from the above inequality, the definitions of $\|\cdot\|_{\vV_h}$ and $\|\cdot\|_{\nnorm}$, and the triangle inequality.
\end{proof}

Denote the left-hand side of the second equation in \eqref{dispro2} by
$$
b_{h,1}(\vu_h, q) = b_h(\vu_h, q) -\langle\vu_0\cdot\vn -\overline{\vu_0\cdot\vn } ,q\rangle_{\partial\Omega_h}.
$$
System \eqref{dispro2} is non-symmetric. Non-symmetric mixed systems have been studied in \cite{Bernardi88,Ciarlet03, Nicolaides82}.
We follow the theoretical framework in \cite{Nicolaides82}. Since $\ah(\cdot,\cdot)$ is coercive on the entire $\vV_h$, 
stability of \eqref{dispro2} only requires that both $b_h(\cdot,\cdot)$ and $b_{h,1}(\cdot,\cdot)$ satisfy the discrete inf-sup condition
(see \cite{Nicolaides82} and also Remark 4.2.7 in \cite{boffi2013mixed}), which we prove in the following lemma:
\begin{lemma}\label{lb1lb2}
  \wangI{When $h$ and $\frac{s}{h}$ are sufficiently small,}
  both $b_h(\cdot,\cdot)$ and $b_{h,1}(\cdot,\cdot)$ satisfy the discrete $\inf$-$\sup$ condition under the modified norm $\|\cdot\|_{\nnorm}$, i.e.,
  \begin{equation}\label{inf-supnew}
  \begin{aligned}  
    \sup _{\vv \in \vV_h} \frac{|b_h(\vv,q)|}{\|\vv\|_{\nnorm}} &\gtrsim\|q\|_{0,\Omega_h} \qquad \forall\, q \in \Psi_h ,\\
    \sup _{\vv \in \vV_h} \frac{|b_{h,1}(\vv,q)|}{\|\vv\|_{\nnorm}} &\gtrsim\|q\|_{0,\Omega_h} \qquad \forall\, q \in \Psi_h.
    \end{aligned}
  \end{equation}
\end{lemma}
\begin{proof}
  Using lemmas \ref{lbb} and \ref{lem:normequivance}, one immediately gets the discrete inf-sup condition for $b_h(\cdot,\cdot)$ \wangI{when $h$ is sufficiently small}.
  As for $b_{h,1}(\cdot,\cdot)$, through a careful examination of the proof of Lemma \ref{lbb}, 
  \wangI{it is not hard
  to see that one only needs to prove 
  (using the same $\vw$ as in the proof of Lemma \ref{lbb}) 
  \begin{equation} \label{eq:lbbmod}
  \begin{aligned}
  \left(\sum_{K \in \calT_h^B} h_K^{-1}\|\vQ_0 \bm{w} \cdot \vn - \overline{\vQ_0 \bm{w} \cdot \vn }\|_{0,\partial K\cap\partial\Omega_h}^{2}\right)^{1 / 2}
  \left(\sum_{K \in \calT_h^B} h_K\|q\|_{0,\partial K\cap\partial\Omega_h}^{2}\right)^{1 / 2}
  &\le C(h,s) \|q\|_{0,\Omega_h}^2 \\
  &< \|q\|_{0,\Omega_h}^2,
  \end{aligned}
  \end{equation}
  with $C(h,s)$ strictly bounded below $1$.
  By \eqref{eq:q}, this reduces to proving a bound for
  $\sum_{K \in \calT_h^B} h_K^{-1}\|\vQ_0 \bm{w} \cdot \vn- \overline{\vQ_0 \bm{w} \cdot \vn }\|_{0,\partial K\cap\partial\Omega_h}^{2}$,
  which can be achieved by using the trace inequality and the inverse inequality}
    \liul{$$
  \begin{aligned}
    \left(\sum_{K \in \calT_h^B} h_K^{-1}\|\vQ_0 \bm{w} \cdot \vn - \overline{\vQ_0 \bm{w} \cdot \vn }\|_{0,\partial K\cap\partial\Omega_h}^{2} \right)^{\frac{1}{2}}
    &\lesssim  \left(\sum_{K \in \calT_h^B} h_K^{-1}\left(\sum_{e\subset\partial K\cap\partial\Omega_h}s^2\|\nabla(\vQ_0 \bm{w}) \|_{0,e}^{2} \right)\right)^{\frac{1}{2}}\\
    &\lesssim \left( s^2\sum_{K \in \calT_h^B} h_K^{-1}\|\nabla(\vQ_0 \bm{w}) \|_{0,\partial K\cap\partial\Omega_h}^{2} \right)^{\frac{1}{2}}\\
   &\lesssim \left( s^2\sum_{K \in \calT_h^B} h^{-2}\|\nabla(\vQ_0 \bm{w}) \|_{0,K}^{2} \right)^{\frac{1}{2}}\\
   &\lesssim \frac{s}{h}\|\bm{w}\|_{1,\Omega},
  \end{aligned}
  $$}
   
  \wangI{Thus the coefficient $C(h,s)$ in \eqref{eq:lbbmod} is of order $O(\frac{s}{h})$. It can be strictly bounded below $1$
    if $\frac{s}{h}$ is sufficiently small.
  }
  The rest of the proof follows directly from the proof of Lemma \ref{lbb}.
\end{proof}

\wangI{
  \begin{remark} \label{rem:sh}
    When the curved boundary of $\partial\Omega$ is approximated by multiple short edges,
    we usually have $\frac{s}{h}$ being sufficiently small and hence Lemma \ref{lb1lb2} holds.
    The proof does not work for the case of $s=h$, i.e., when no refinement is imposed on boundary edges.
    However, several numerical experiments to be presented in Section \ref{section5} show that
    the modified discrete scheme appears to be stable even when $s=h$.
    We suspect that a better discrete inf-sup condition, with less restrictions on $s$ and $h$, can be proved using more advanced skills.
  \end{remark}

  \begin{remark}
    The proof of Lemma \ref{lb1lb2} depends on the $O(s)$ approximation provided by $\overline{\vQ_0 \bm{w} \cdot \vn }$,
    which is one of the reasons for introducing this term, as pointed out in Remark \ref{rem:overline}.
    An interesting question is whether Lemma \ref{lb1lb2} holds or not with the piecewise average $\overline{\vQ_0 \bm{w} \cdot \vn }$
    replaced by a global average on $\partial\Omega_h$, which also satisfies the compatibility condition discussed in Remark \ref{rem:overline}.
    Our numerical experiments (not reported in this article) suggest that this is possible.
    But it remains to be proved and the proof appears to be non-trivial.
  \end{remark}
}


According to Theorem 3.1 in \cite{Nicolaides82}, the inf-sup conditions \eqref{inf-supnew} ensure that the discrete problem (\ref{dispro2}) admits a unique solution
\wangI{when $h$ and $\frac{s}{h}$ are sufficiently small}.
Moreover, the discrete operator in (\ref{dispro2}) is stable in the sense that the unique solution to
$$
\begin{cases}
\ah(\vu_h, \vv) + b_h(\vv, p_h)=F(\vv) \quad &\forall\, \vv \in \vV_h,\\
b_{h,1}(\vu_h, q)  =G(q) \quad &\forall\, q \in \Psi_h, 
\end{cases}
$$
satisfies
\begin{equation} \label{eq:stabilityNew}
  \|\vu_h\|_{\nnorm} + \|q\|_{\Psi_h} \lesssim \|F\|_{\vV_h'} + \|G\|_{\Psi_h'}.
\end{equation}

\subsection{Error analysis}
Next we analyze the error of the modified discretization (\ref{dispro2}).
Again we first derive the error equations.
An argument similar to the one used in the proof of Lemma \ref{lemma1} shows that
\begin{lemma}\label{lemma11}
The solution $\vu$ and $p$ to problem (\ref{primalproblem}) satisfy
$$
  \ah\left(\vtQ_h\vu, \vv\right)+b_h\left(\vv,  \bbQ_hp\right)=\sh (\vQ_h \vu, \vv)+l_{s,1}(\vv)-l_{\mathrm{div}}(\vv)
  \qquad \forall\, \vv \in \vV_h,
$$
where $l_{\mathrm{div}}(\cdot)$ is defined as in Lemma \ref{lemma1} and $l_{s,1}(\cdot)$ is defined by
$$
l_{s,1}(\vv) =\rho \sum_{K \in \calT_h} h_K^{-1}\left\langle \vQ_b\vu \cdot \vtn,\left(\vv_0-\vv_b\right) \cdot \vtn \right\rangle_{\partial K\cap\partial\Omega_h}.
$$
\end{lemma}
\begin{proof}
  The proof is similar to the one for Lemma \ref{lemma1} given in Appendix \ref{appendix1}.
  One simply replaces \eqref{2.11} by
  $$
\begin{aligned}
\sh(\vtQ_h\vu, \vv) &= \rho \sum_{K \in \calT_h} h_K^{-1}\langle(\vQ_0\vu-\vtQ_b\vu) \cdot \vtn,(\vv_0-\vv_b) \cdot \vtn\rangle_{\partial K},&\\
&=\sh(\vQ_h\vu, \vv) + \rho \sum_{K \in \calT_h} h_K^{-1}\langle{\vQ}_b\vu \cdot \vtn,(\vv_0-\vv_b) \cdot \vtn\rangle_{\partial K\cap\partial\Omega_h},&\\
&= \sh(\vQ_h\vu, \vv) + l_{s,1}(\vv).
\end{aligned}
$$
\end{proof}

Let $\vu$, $p$ be the solution to problem (\ref{primalproblem}),
and $\vu_h$, $p_h$ be the solution to the modified weak Galerkin formulation (\ref{dispro2}).
Consider the error $\ve_{\vu}$ and $e_p$ defined in \eqref{eq:errdef}.
By lemmas \ref{lemma11} and \ref{lemma2}, the error to problem \eqref{dispro2} satisfies
\begin{equation}\label{erroreuqation2}
\begin{cases}
\ah\left(\ve_{\vu}, \vv\right)+b_h\left(\vv, e_p\right) =\sh(\vQ_h \vu, \vv)+l_{s,1}(\vv)-l_{\text{div}}(\vv) \quad &\forall \vv\in\vV_h, \\
b_{h,1}\left(\ve_{\vu}, q\right)= l_{b,1}(q) +l_{b,2}(q)\quad &\forall q \in \Psi_h,
\end{cases}
\end{equation}
where
$$
\begin{aligned}
l_{b,1}(q) &= \langle (\vu-\vQ_0\vu)\cdot\vn - \overline{(\vu-\vQ_0\vu)\cdot\vn }, q\rangle_{\partial\Omega_h},\\
l_{b,2}(q) &= \langle\overline{\vu\cdot\vn },q\rangle_{\partial\Omega_h}.\\
\end{aligned}
$$

We first derive an upper bound for $\sh(\cdot,\cdot)$, $l_{s,1}(\cdot)$ and $l_{b,1}(\cdot)$.
\begin{lemma} \label{Berize2}For any $ 0 \le r\le \alpha$ and $0 \le t \le \sigma$,
  assume that $\vu\in \vV\cap [H^{\max \{r+1,3\}}(\Omega)]^2, p\in\Psi\cap H^{t+1}(\Omega)$ are the solutions to problem (\ref{primalproblem})-\eqref{primalcondition},
  and $(\vv,q) \in (\vV_h,\Psi_h)$; then we have
\begin{equation*}
\begin{aligned}
\sh(\vQ_h\vu, \vv)&\lesssim  h^r\|\vu\|_{r+1,\Omega} \|\vv\|_{\nnorm},\\
l_{s,1}(\vv)&\lesssim \left( {h^{-\frac{1}{2}}s^2\|\vu\|_{3,\Omega}} + h^r\|\vu\|_{r+1,\Omega} \right) \|\vv\|_{\nnorm}, \\
l_{b,1}(q) &\lesssim h^r\|\vu\|_{r+1,\Omega} \|q\|_{0,\Omega_h}.
\end{aligned}
\end{equation*}
\end{lemma}
\begin{proof}
  The bound for $\sh(\cdot,\cdot)$ follows from a trivial modification of \eqref{eq:sbound}.
  The bound for $l_{b,1}(\cdot)$ follows from the Cauchy-Schwarz inequality, the trace inequality, the approximation property of $\vQ_0$
  and Inequality \eqref{eq:q}.
  We are left with estimating $l_{s,1}(\cdot)$. Note that
  $$
  \begin{aligned}
    l_{s,1}(\vv)
    &= \rho \sum_{K \in \calT_h^B} h_K^{-1}\left\langle \vQ_b\vu \cdot \vtn,\left(\vv_0-\vv_b\right) \cdot \vtn \right\rangle_{\partial K\cap \partial\Omega_h}\\
    &\lesssim \left(\sum_{\e\in\calT_h^B}h^{-1}_K\| \vQ_b \vu\cdot \vtn\|_{0,\partial\e\cap\partial\Omega_h}^2\right)^{1/2} \|\vv\|_{\nnorm},\\
    &\lesssim \left(\sum_{\e\in\calT_h^B}h^{-1}_K\| \vu\cdot \vtn\|_{0,\partial\e\cap\partial\Omega_h}^2
       + \sum_{\e\in\calT_h^B}h^{-1}_K\| (\vu-\vQ_b\vu)\cdot \vtn\|_{0,\partial\e\cap\partial\Omega_h}^2 \right)^{1/2} \|\vv\|_{\nnorm} \\
    &\lesssim \left(\sum_{\e\in\calT_h^B}h^{-1}_K\| \vu\cdot \vtn\|_{0,\partial\e\cap\partial\Omega_h}^2 + h^{2r}\|\vu\|_{r+1}^2\right)^{1/2} \|\vv\|_{\nnorm},
  \end{aligned} 
  $$
  where in the last step we used the same argument as in \eqref{eq:sbound},
  together with the fact that $\|\vu-\vQ_b\vu\|_e \le \|\vu-\vQ_0\vu\|_e$ on each edge.
  
  Now we only need to estimate $\sum_{\e\in\calT_h^B} h^{-1}_K \|\vu\cdot\vtn\|_{0,\partial\e\cap\partial\Omega_h}^2$.
  We use the same argument as in the proof of Lemma 2.4.4 in \cite{wgcurved}.
  Reference \cite{wgcurved} is written in Chinese. For reader's convenience, we present the full detail of the estimate below.
  The key is to use $\vu\cdot \vtn|_{\partial\Omega}=0$ and the map $\vsig$ defined in Assumption {\bf A4}.
  In the local coordinate system $\hx$-$\hy$ as shown in Figure \ref{sigma}, the value of $\vu\cdot\vtn$
  at a point $(\hx,0)\in e$ satisfies
  $$
  \begin{aligned}
    \left|(\vu\cdot\vtn)(\hx,0)\right|
    &= \left|\left(\vu(\hx,\gamma(\hx)) - \int_0^{\gamma(\hx)} \frac{\partial \vu}{\partial \hy}\dhy  \right) \cdot\vtn(\hx,\gamma(\hx))\right| \\
    &=\left|\left(\int_0^{\gamma(\hx)} \frac{\partial \vu}{\partial \hy}\dhy\right)  \cdot\vtn(\hx, \gamma(\hx))\right|\\
    &\lesssim s^2 \|\nabla \vu\|_{L^{\infty}(\Omega)} .
  \end{aligned}
  $$
  Therefore, on each $e\in \calE_h^B$, one has
  $$
    \|\vu\cdot\vtn\|_e^2 = \int_0^{h_e} \left|(\vu\cdot\vtn)(\hx,0)\right|^2 \dhx \lesssim s^5 \|\nabla \vu\|_{L^{\infty}(\Omega)}^2. 
  $$
  Hence
  $$
  \sum_{\e\in\calT_h^B} h^{-1}_K \|\vu\cdot\vtn\|_{0,\partial\e\cap\partial\Omega_h}^2
   \lesssim  h^{-1} s^4 \|\nabla \vu\|_{L^{\infty}(\Omega)}^2 \lesssim h^{-1} s^4 \|\vu\|^2_{3,\Omega}.
  $$
  Combining the above, this completes the proof of the lemma. 
\end{proof}

Then we derive an upper bound for $l_{b,2}(\cdot)$.
\begin{lemma} \label{Berize3}
  For any $ 0 \le r\le \alpha$ and $0 \le t \le \sigma$,
  assume that 
  $\vu$ is the solution to (\ref{primalproblem})-\eqref{primalcondition},
  and $(\vv,q) \in (\vV_h,\Psi_h)$; then we have
\begin{equation*}
\begin{aligned}
l_{b,2}(q) &\lesssim h^{-\frac{1}{2}}s^2\|\vu\|_{3,\Omega} \|q\|_{0,\Omega_h}.
\end{aligned}
\end{equation*}
\end{lemma}
\begin{proof}
 We first derive an upper bound for $ \overline{\vu\cdot\vn }|_e$ on each $e\in \calE_h^B$ using integration by parts
  {\liul{\begin{equation*}
\begin{aligned}
 \big|\,\overline{\vu\cdot\vn }|_e\,\big|&= \left|\frac{1}{|e|}\int_e\vu\cdot\vn \ds\right| 
 =\frac{1}{|e|}\left|\int_e\vu\cdot\vn \ds-\int_{\tilde{e}}\vu\cdot\vtn \ds\right| \\
 &= \frac{1}{|e|}\left|\int_{M_e}\nabla\cdot\vu \dx\right|\\
 &\lesssim s^{-1}|M_e|\|\vu\|_{1,\infty,\Omega}\\
 &\lesssim s^2\|\vu\|_{3,\Omega}.
\end{aligned}
\end{equation*}}}
Combining the above with the trace and the inverse inequalities, we have the bound for $l_{b,2}(q)$ 
{\liul{\begin{equation*}
\begin{aligned}
l_{b,2}(q)& =\langle\overline{\vu\cdot\vn },q\rangle_{\partial\Omega_h}
\leq \sup_{e\in\calE_h^B} \big|\,\overline{\vu\cdot\vn }|_e\,\big|\int_{\partial\Omega_h}|q|ds\\
&\lesssim  s^2\|\vu\|_{3,\Omega}\|q\|_{0,\partial\Omega_h}\\
&\lesssim h^{-\frac{1}{2}} s^2\|\vu\|_{3,\Omega}\left(\sum_{K\in\calT_h^B}h_K\|q\|_{0,\partial K \cap \partial\Omega_h}^2\right)^{\frac{1}{2}}\\
&\lesssim h^{-\frac{1}{2}} s^2\|\vu\|_{3,\Omega}\|q\|_{\Omega_h}.
\end{aligned}
\end{equation*}}}
\end{proof}

Now, we are able to derive the error estimate:
\begin{theorem}\label{th2}
  Let $\vu$ and $p$ satisfy the condition in Lemma \ref{Berize2}. The error $\bm{e}_{\vu}$ and $e_p$ satisfy
$$
\left\|\ve_{ \vu}\right\|_{\nnorm}+\left\|e_{ p}\right\|_{0,\Omega_h} \lesssim  h^{r }\|\vu\|_{r +1, \Omega }+h^{t +1}\|p\|_{t +1, \Omega }+h^{-\frac{1}{2}}s^2\|\vu\|_{3,\Omega}.
$$
where $0 \le r\le \alpha $ and $0 \le t \le \sigma$.
\end{theorem}
\begin{proof}
  The estimate follows from the theoretical framework in \cite{Nicolaides82}, the stability result \eqref{eq:stabilityNew},
  the error equation (\ref{erroreuqation2}), lemmas \ref{Berize} and \ref{Berize2}-\ref{Berize3}.
\end{proof}

\wangI{Similar to Corollary \ref{lemma3.16}, one can adjust $s$ to improve the error estimate. A simple calculation shows that
  an optimal $O(h^{\alpha})$ convergence requires $s= O(\sqrt{h^{\alpha+\frac{1}{2}}})$, which is less demanding comparing to
  the $s= O(h^{\alpha+\frac{1}{2}})$ in Corollary \ref{lemma3.16}. However, we still need to check the requirements
  for the discrete inf-sup condition stated in Lemma \ref{lb1lb2}.
  When $\alpha \ge 2$, by taking $s= O(h^{\alpha+\frac{1}{2}})$ one gets $\frac{s}{h}\le O(h^{\frac{1}{4}})$
  which goes to $0$ as $h$ goes to $0$, i.e., $\frac{s}{h}$ is sufficiently small whe $h$ is small.
  It is different for the case $\alpha=1$, where $O(\sqrt{h^{\alpha+\frac{1}{2}}}) = O(h^{\frac{3}{4}}) > O(h)$.
  Since we always have $s\le h$, optimal convergence holds as long as $\frac{s}{h}$ is sufficiently small.
 
\begin{corollary}\label{lemma4.10}
  Assuming the polynomial approximation error reaches the optimal $O(h^{\alpha})$ with $\alpha\ge 1$.
  For $\alpha=1$, when $h$ and $\frac{s}{h}$ are sufficiently small, the errors $\bm{e}_{\vu}$ and $e_p$ reach an optimal
  $$
  \left\|\ve_{ {\vu}}\right\|_{\nnorm}+\left\|e_{ p}\right\|_{0,\Omega_h}\lesssim O(h).
  $$
  For $\alpha \ge 2$, when $s= O(\sqrt{h^{\alpha+\frac{1}{2}}})$ and $h$ is sufficiently small,
  the errors $\bm{e}_{\vu}$ and $e_p$ reach an optimal  
  $$
  \begin{aligned}
    \left\|\ve_{ {\vu}}\right\|_{\nnorm}+\left\|e_{ p}\right\|_{0,\Omega_h} \lesssim O(h^{\alpha}).
  \end{aligned}
  $$
\end{corollary}

\begin{remark} \label{rem:sh2}
  As mentioned previously in Remark \ref{rem:sh}, we suspect that the condition of $\frac{s}{h}$ being sufficiently small is not necessary,
  but are not yet able to prove it. Indeed, numerical results in Section \ref{section5} will show that the estimate in Corollary \ref{lemma4.10} is
  sharp in the case $\alpha=1$ with just $s=O(h)$.
\end{remark}

}




\section{Numerical examples}{\label{section5}}

In this section, we present a sequence of numerical examples to validate the accuracy of the original weak 
Galerkin method (\ref{dispro}) and the modified weak Galerkin method (\ref{dispro2}).
All experiments are done with polynomial degrees $\alpha=\beta=j$ and  $\sigma=j-1$, for $j\geq 1$,
which, according to theorems \ref{th1} and \ref{th2}, are the best choice to minimize the approximation error.
Throughout this section, we shall address such scheme using $P_j$-$P_j$-$P_{j-1}$, with various $j$s.

\medskip

\noindent{\bf{ Example 5.1. Rectangular domain}}

We first quickly present numerical results on the square domain $\Omega= (0,1)\times(0,1)$, which shall serve as a comparison group.
Uniform triangular meshes are used.
The exact solution is, 
$$
\vu(x,y)=\left(
\begin{aligned}
\pi\sin(\pi x)\cos(\pi y)\\
\pi\cos(\pi x)\sin(\pi y)
\end{aligned}\right),\,
  \qquad
p(x,y) = \cos(\pi x)\cos(\pi y),
$$
which satisfies a homogeneous Neumann boundary condition $\vu\cdot \vn=0\text{ on }\partial\Omega$
and $\int_{\Omega}p\dx = 0$.
In Figure \ref{p1p2ret},
it can be seen that the $P_j$-$P_j$-$P_{j-1}$ original WG schemes exhibits an optimal $(j+1)$th order convergence, as predicted.
\begin{figure}[H]
\begin{center}
\includegraphics[width=2in]{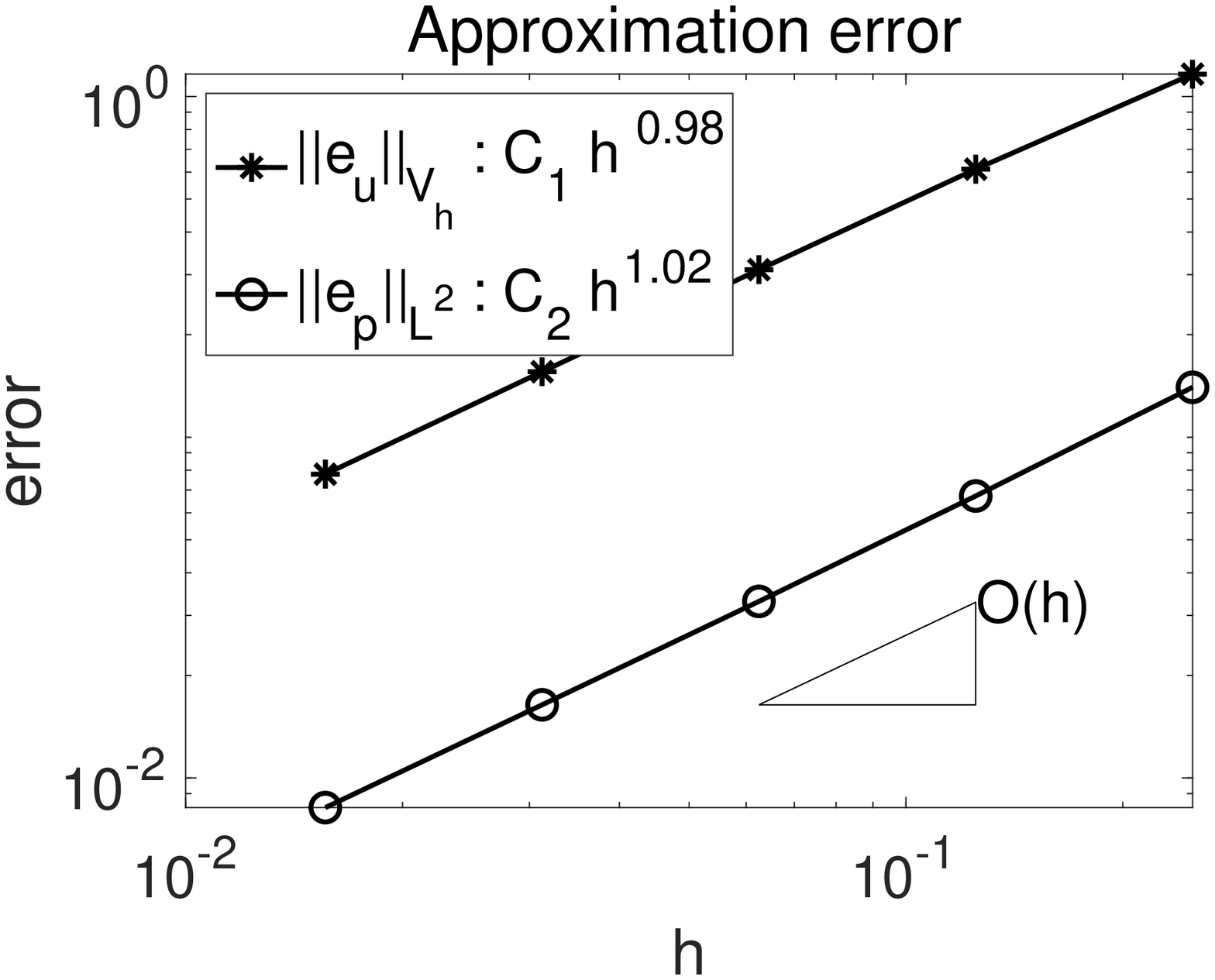} \includegraphics[width=2in]{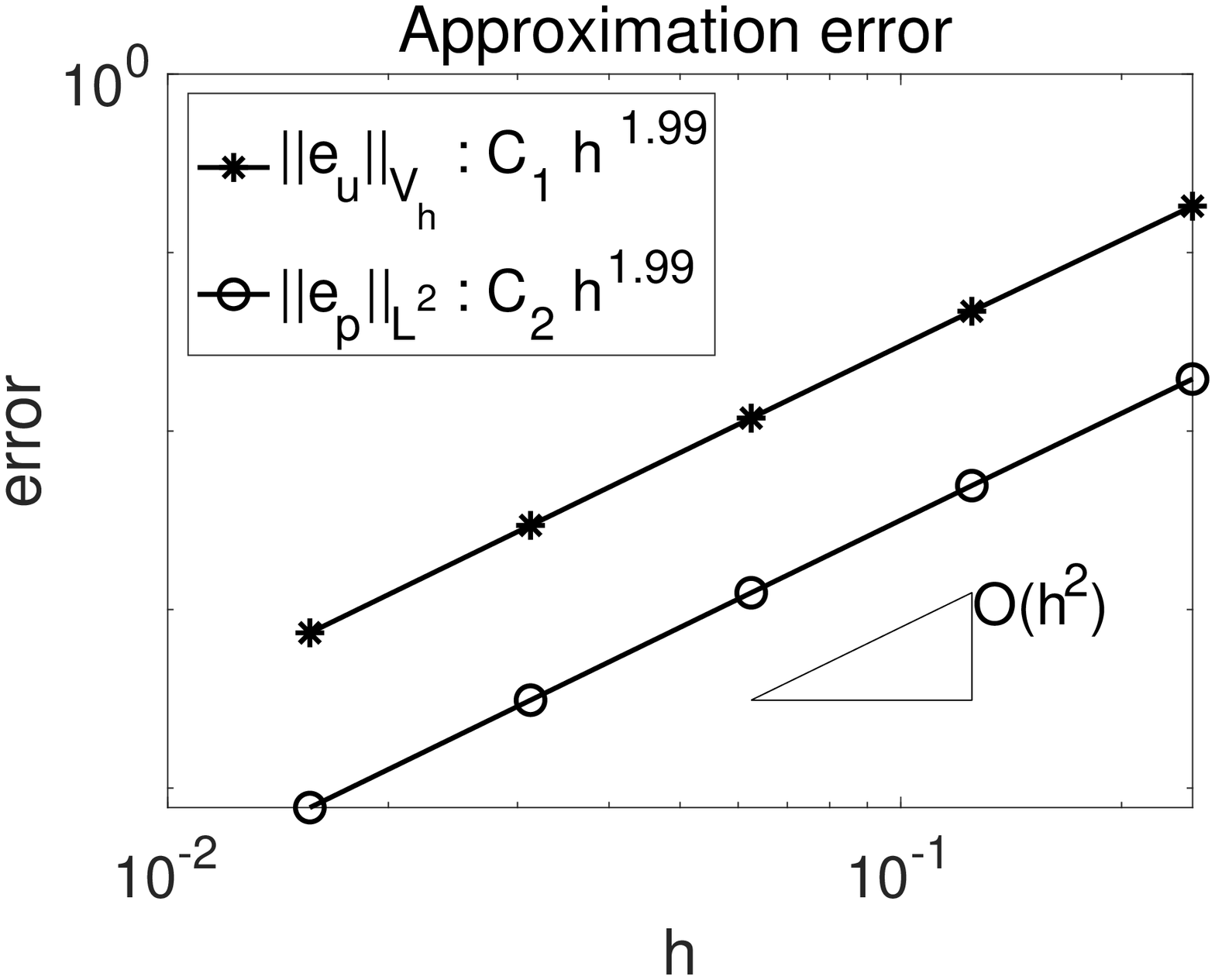}\\
\includegraphics[width=2in]{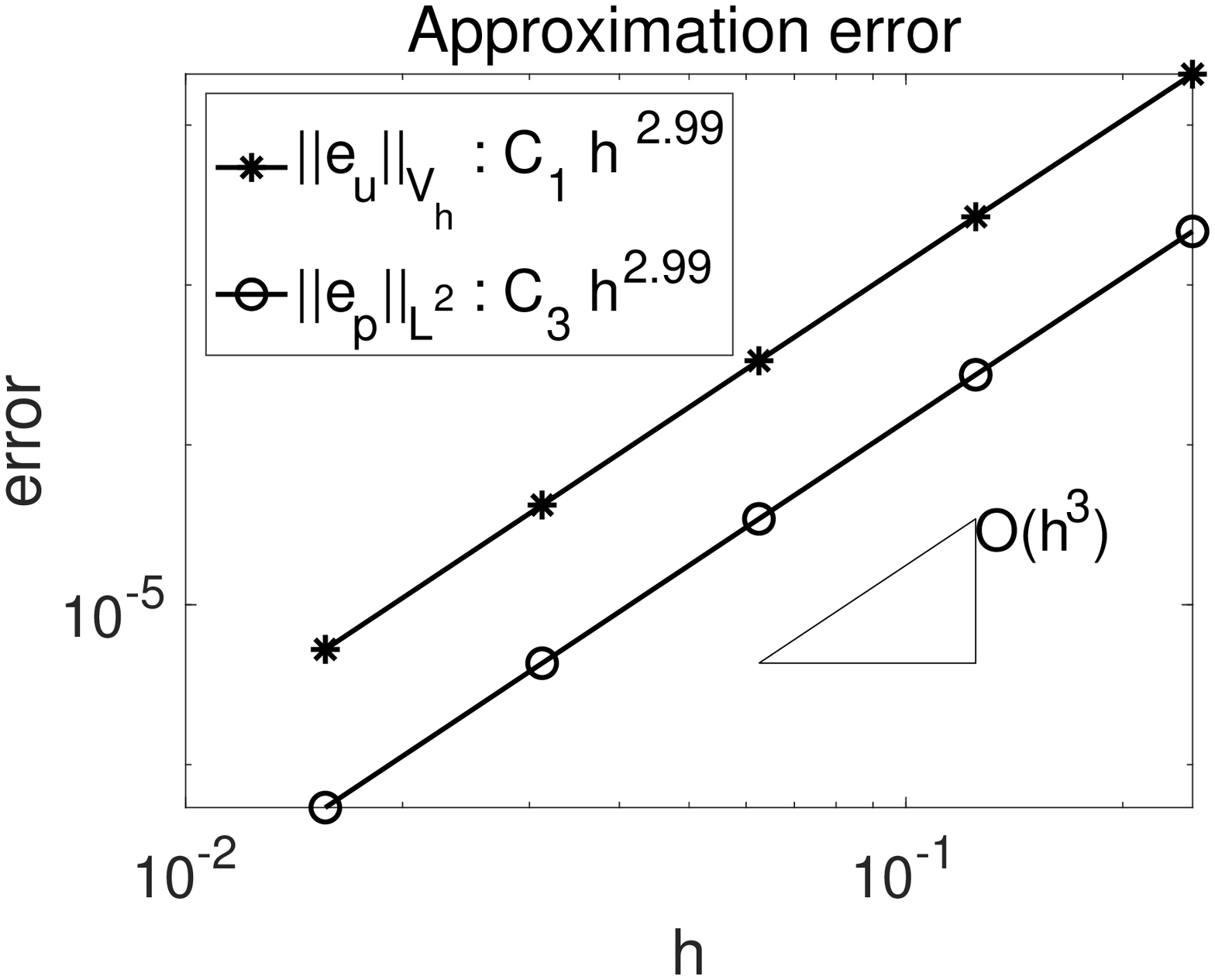}\includegraphics[width=2in]{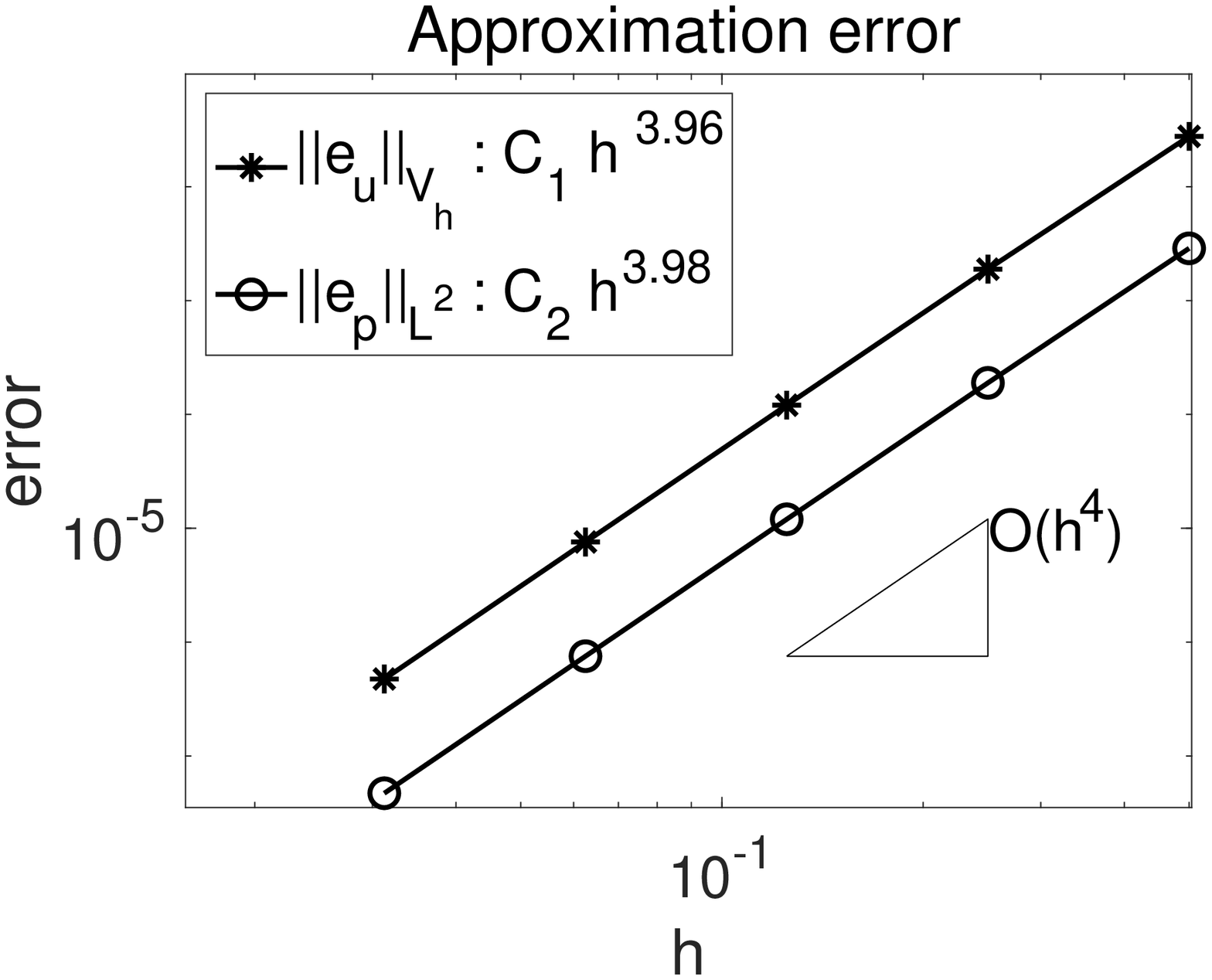}
\caption{Convergence rates (from left to right; top to bottom)
  for the $P_1$-$P_1$-$P_0$, $P_2$-$P_2$-$P_1$, $P_3$-$P_3$-$P_2$ and $P_4$-$P_4$-$P_3$ original WG scheme (\ref{dispro}) on a rectangular domain.}
\label{p1p2ret}
\end{center}
\end{figure}

\noindent{\bf{ Example 5.2. Unit disk}}

The domain is a unit disk $\Omega = \{(x,y)\,|\,x^2+y^2< 1\}$,
and the exact solution is 
$$
\vu(x,y)=\left(
\begin{aligned}
3x^2&+y^2-3\\
 &2xy
\end{aligned}\right),\,
\qquad
p(x,y) = 3x-x(x^2+y^2),
$$
which satisfies a homogeneous Neumann boundary condition and $\int_{\Omega}p\dx = 0$.

Three meshes, as shown in Figure \ref{circlemesh}, are considered.
The leftmost is a plain triangular mesh with $s=O(h)$.
According to \wangI{corollaries \ref{lemma3.15}, \ref{lemma4.10} and Remark \ref{rem:sh2}}, on this mesh we expect
to have $O(h^{\frac{1}{2}})$ convergence for all original WG schemes, optimal $O(h)$ convergence for the lowest-order modified WG scheme,
and $O(h^{\frac{3}{2}})$ convergence for all high-order modified WG schemes.
The mesh in the middle is a polygonal mesh with each curved side further divided into $\lceil h^{{\frac{1}{2}-j}}\rceil$ short straight edges,
where $\lceil\cdot\rceil$ means the ceiling.
In other words, we have $\frac{h}{s}=O(h^{{\frac{1}{2}-j}})$ which is equivalent to $s=O(h^{j+\frac{1}{2}})$.
In Figure \ref{circlemesh} the illustration is given for $j=2$.
According to Corollary \ref{lemma3.16}, this is the requirement for the original WG to exhibit optimal convergence.
The rightmost is a polygonal mesh with each curved side divided into $\lceil h^{\frac{3-2j}{4}}\rceil$ short straight edges,
or in other words, $s=O(\sqrt{h^{j+\frac{1}{2}}})$.
Again the illustration is given for $j=2$.
One can immediately see that it uses less short edges than the mesh in the middle.
According to Corollary \ref{lemma4.10}, this guarantees that the modified WG scheme has optimal convergence.

\begin{figure}[H]
\begin{center}
\includegraphics[width=2in]{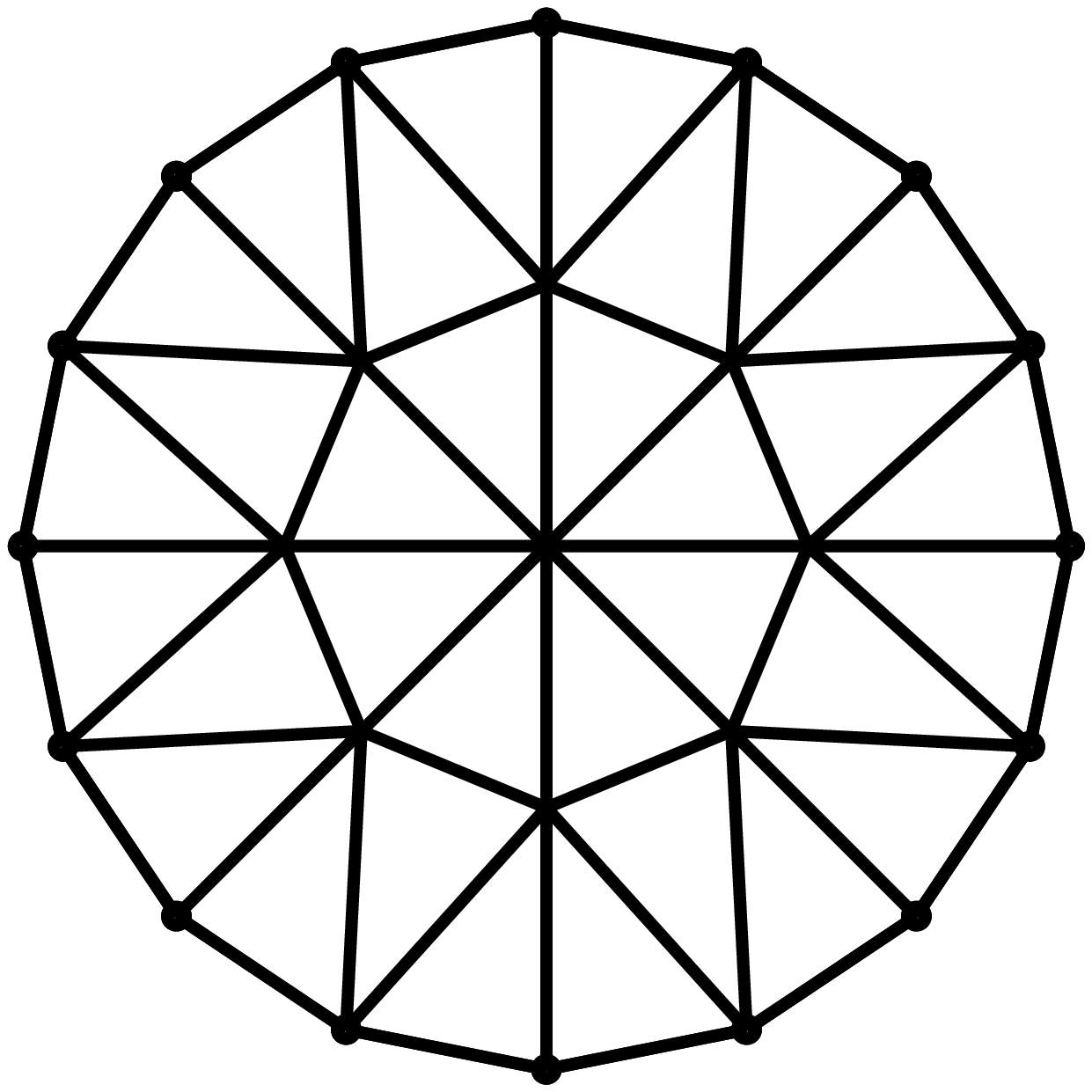} \includegraphics[width=2in]{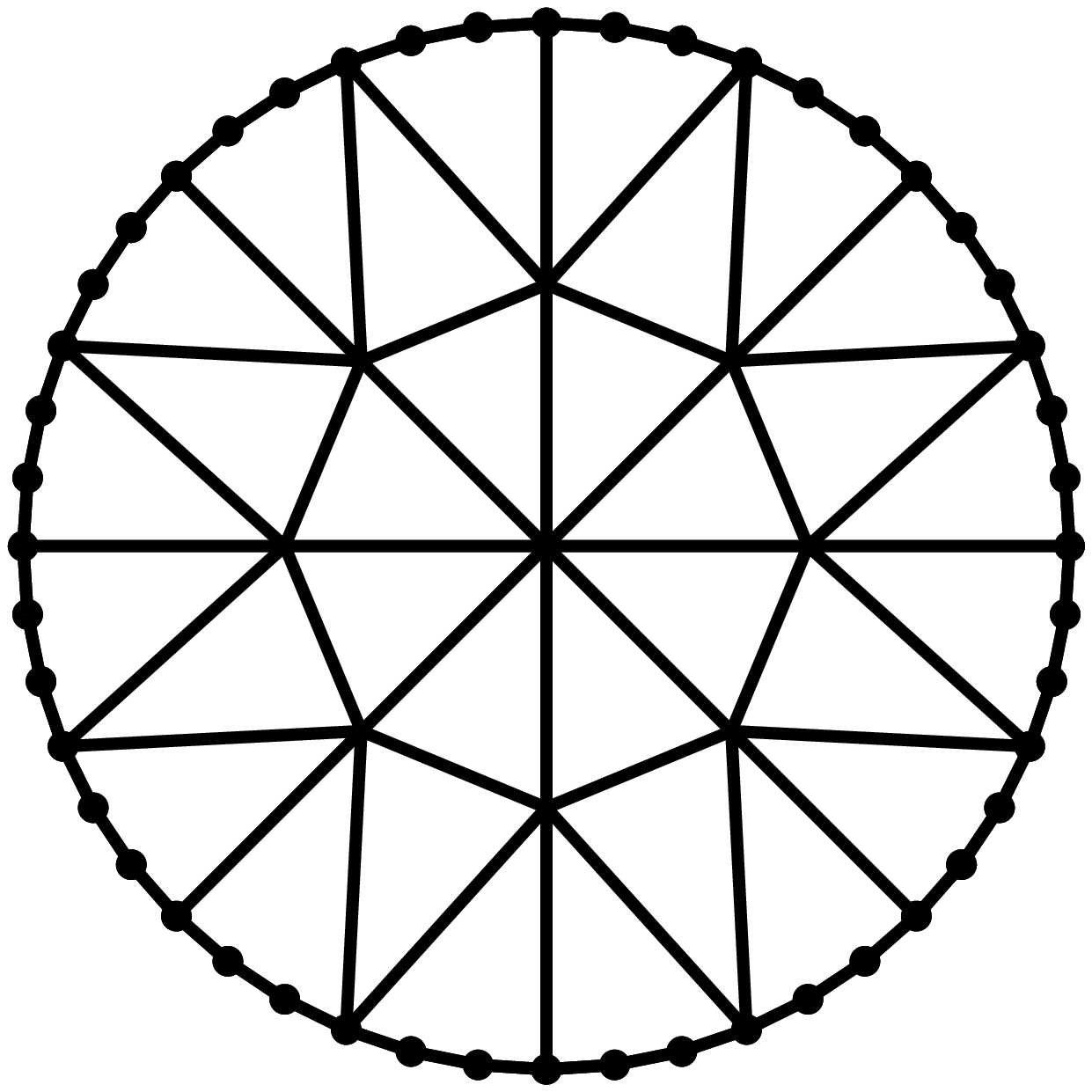} \includegraphics[width=2in]{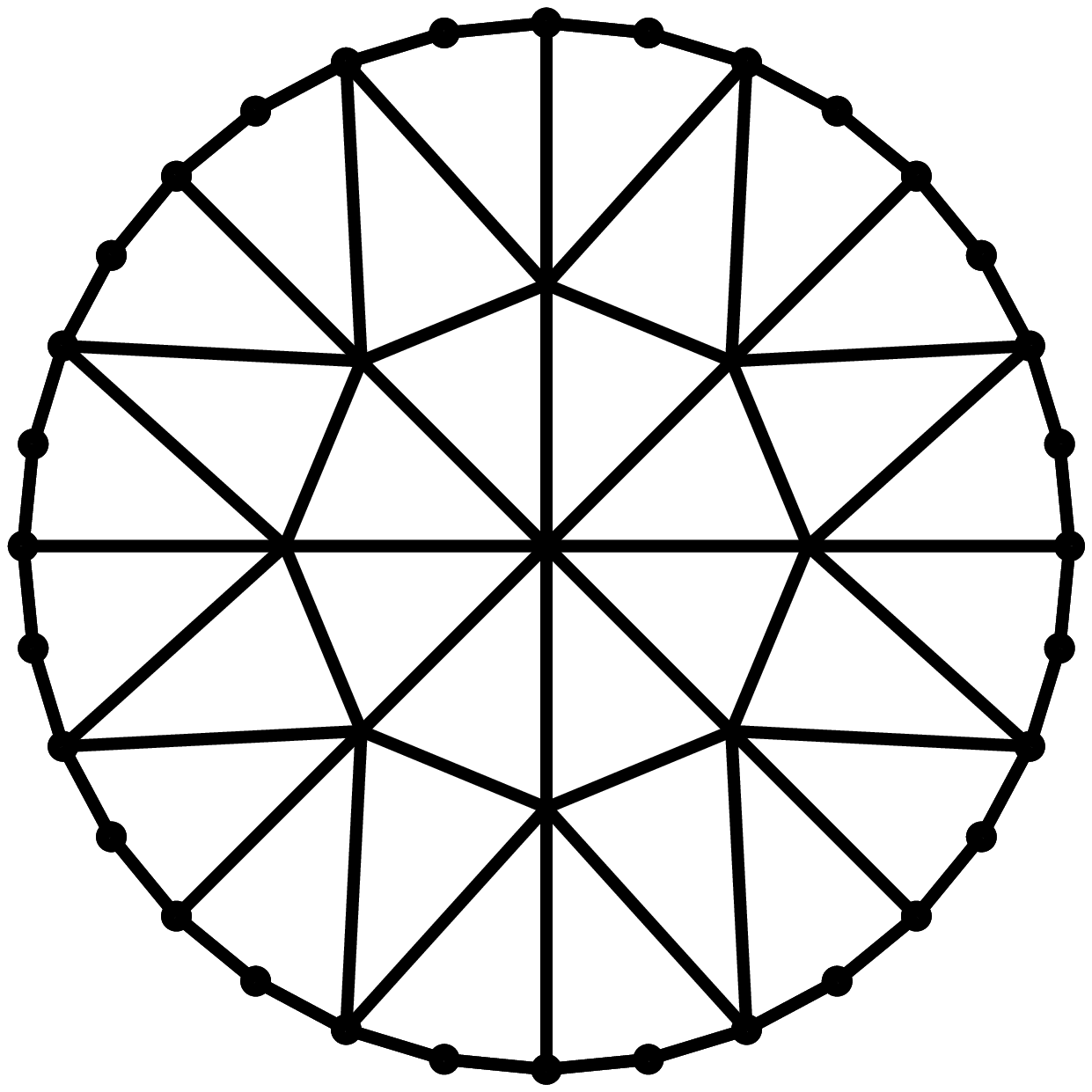} 
\caption{Triangular and polygonal meshes on the unit disk.}
\label{circlemesh}
\end{center}
\end{figure}
 
In figures \ref{circlewglong}-\ref{p1p2p3p4}, convergence rates of the original and modified WG schemes 
on the left mesh of Figure \ref{circlemesh} are reported.
The results agree with the conclusion in corollaries \ref{lemma3.15} and \ref{lemma4.10}.
Higher-order schemes with $j>2$ (we tested $j$ up to $4$) behave the same as the case of $j=2$, and hence are omitted to save space.
We notice an interesting fact that $\|e_{ p}\|_{0,\Omega_h}$ appears to have a better convergence rate than $\|\ve_{ {\vu}}\|_{\bm{V}_h}$,
which awaits further study in the future.

\begin{figure}[H]
\begin{center}
\includegraphics[width=2in]{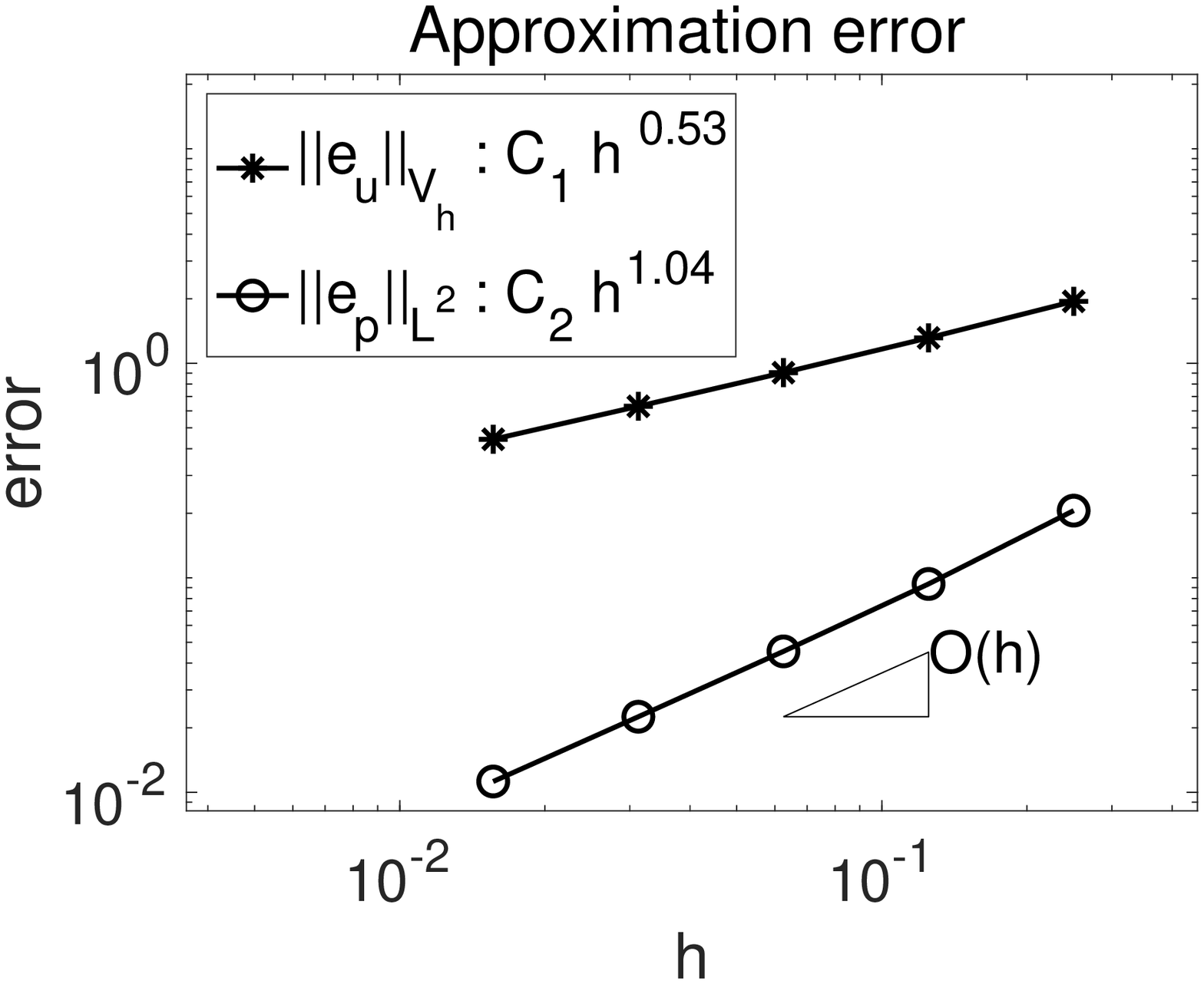} \includegraphics[width=2in]{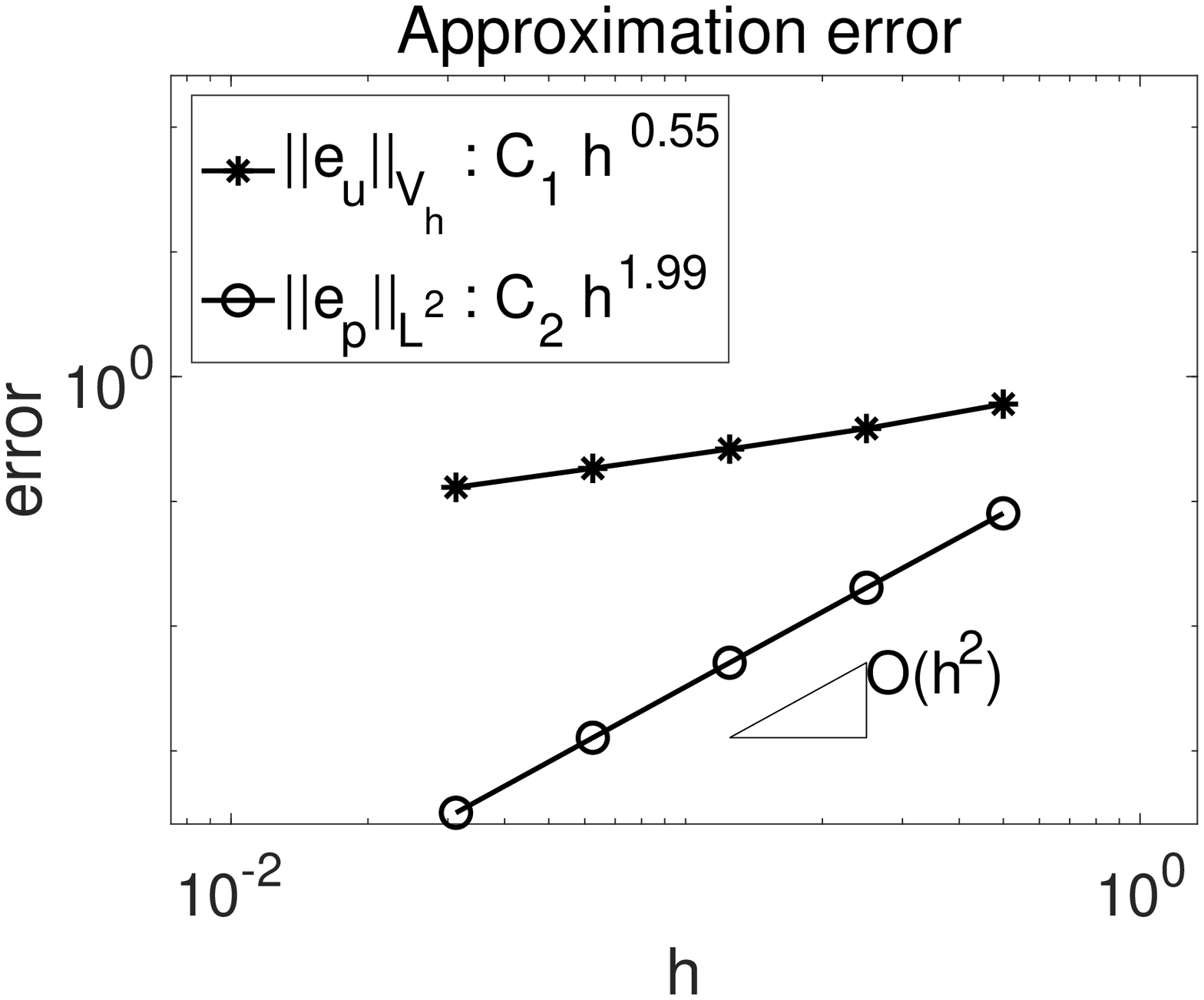} 
\caption{Convergence rates for the $P_1$-$P_1$-$P_0$ and $P_2$-$P_2$-$P_1$ original WG scheme (\ref{dispro})
  on the left mesh in Figure \ref{circlemesh}.}
\label{circlewglong}
\end{center}
\end{figure}
\begin{figure}[H]
\begin{center}
\includegraphics[width=2in]{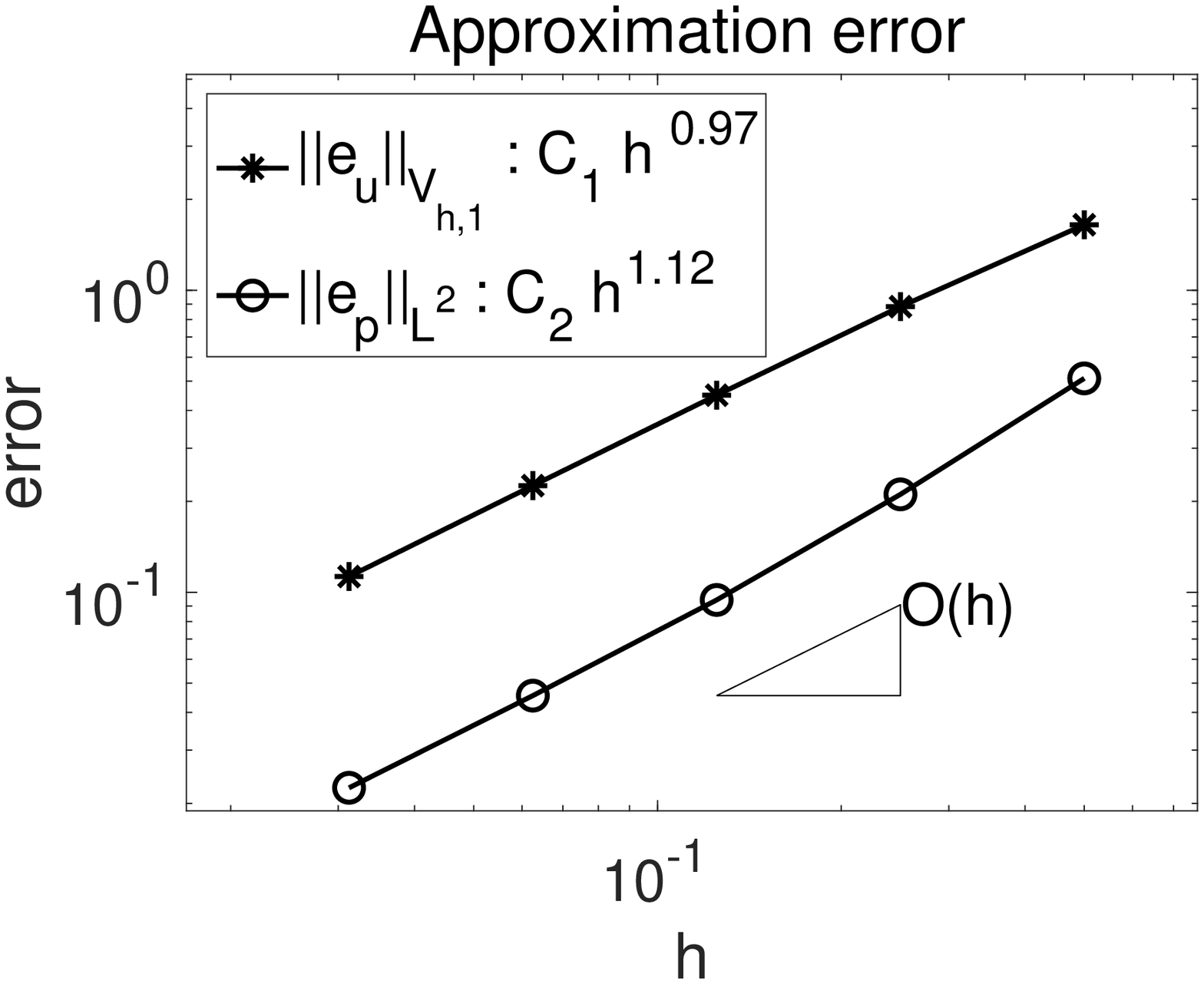}\includegraphics[width=2in]{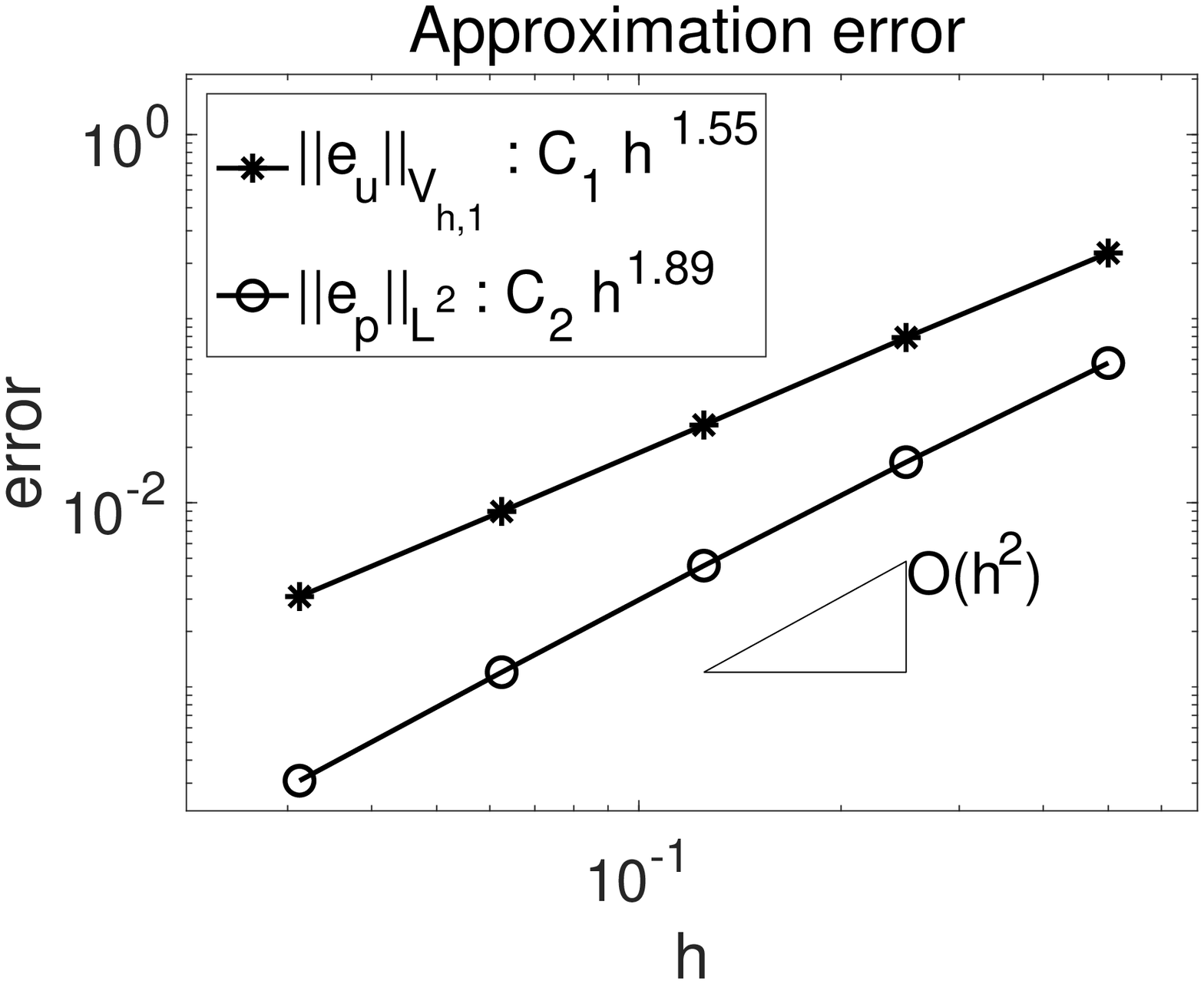}
\caption{Convergence rates for the $P_1$-$P_1$-$P_0$ and $P_2$-$P_2$-$P_1$ modified WG scheme (\ref{dispro2})
  on the left mesh in Figure \ref{circlemesh}. }
  
\label{p1p2p3p4}
\end{center}
\end{figure}

We then test the original and the modified WG scheme, respectively, on the middle and the right meshes in Figure \ref{circlemesh}.
The results are reported in figures \ref{p1p1p0}-\ref{p2p3p4m}.
Both are optimal, as predicted in corollaries \ref{lemma3.16} and \ref{lemma4.10}.

\begin{figure}[H]
\begin{center}
\includegraphics[width=2in]{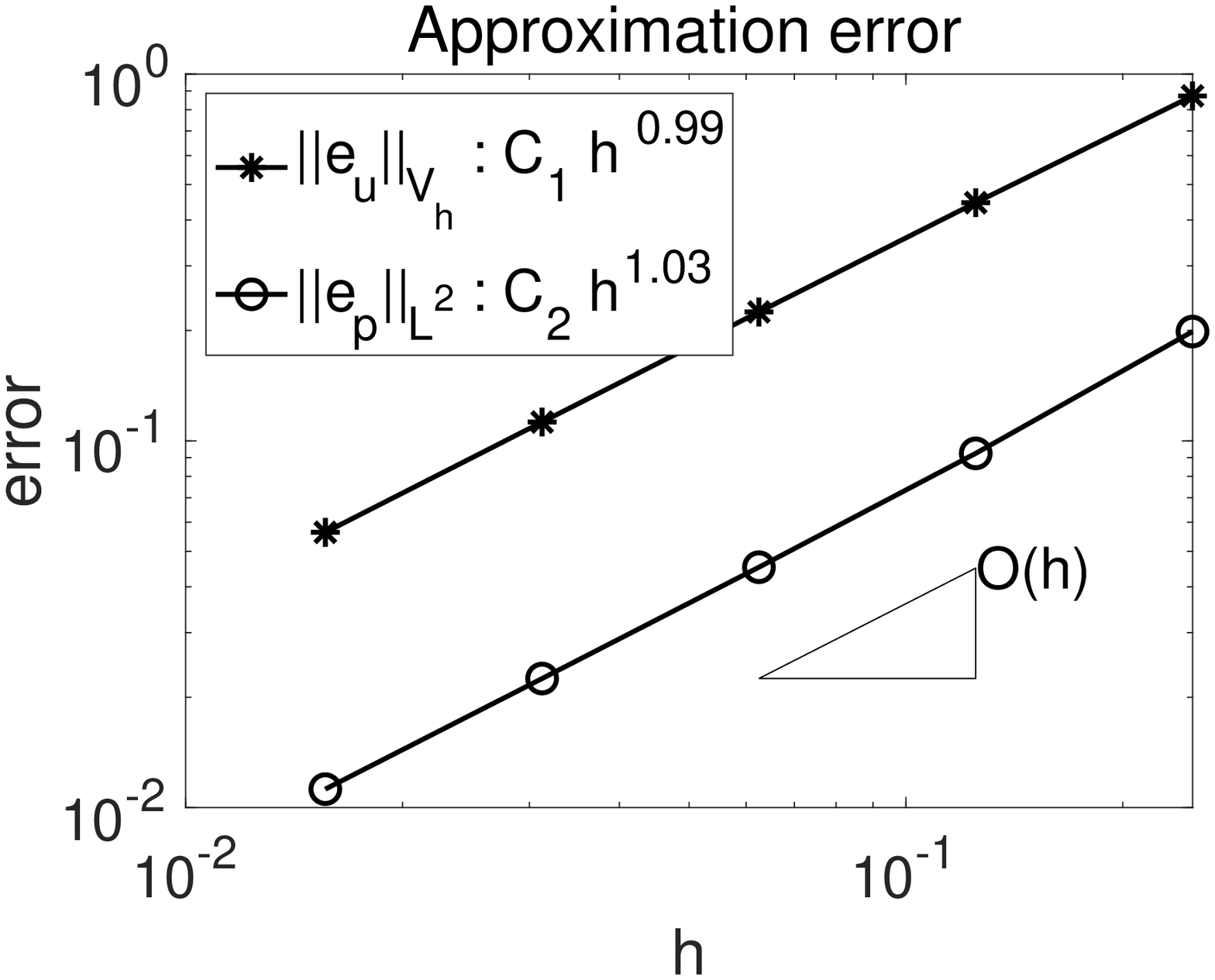} \includegraphics[width=2in]{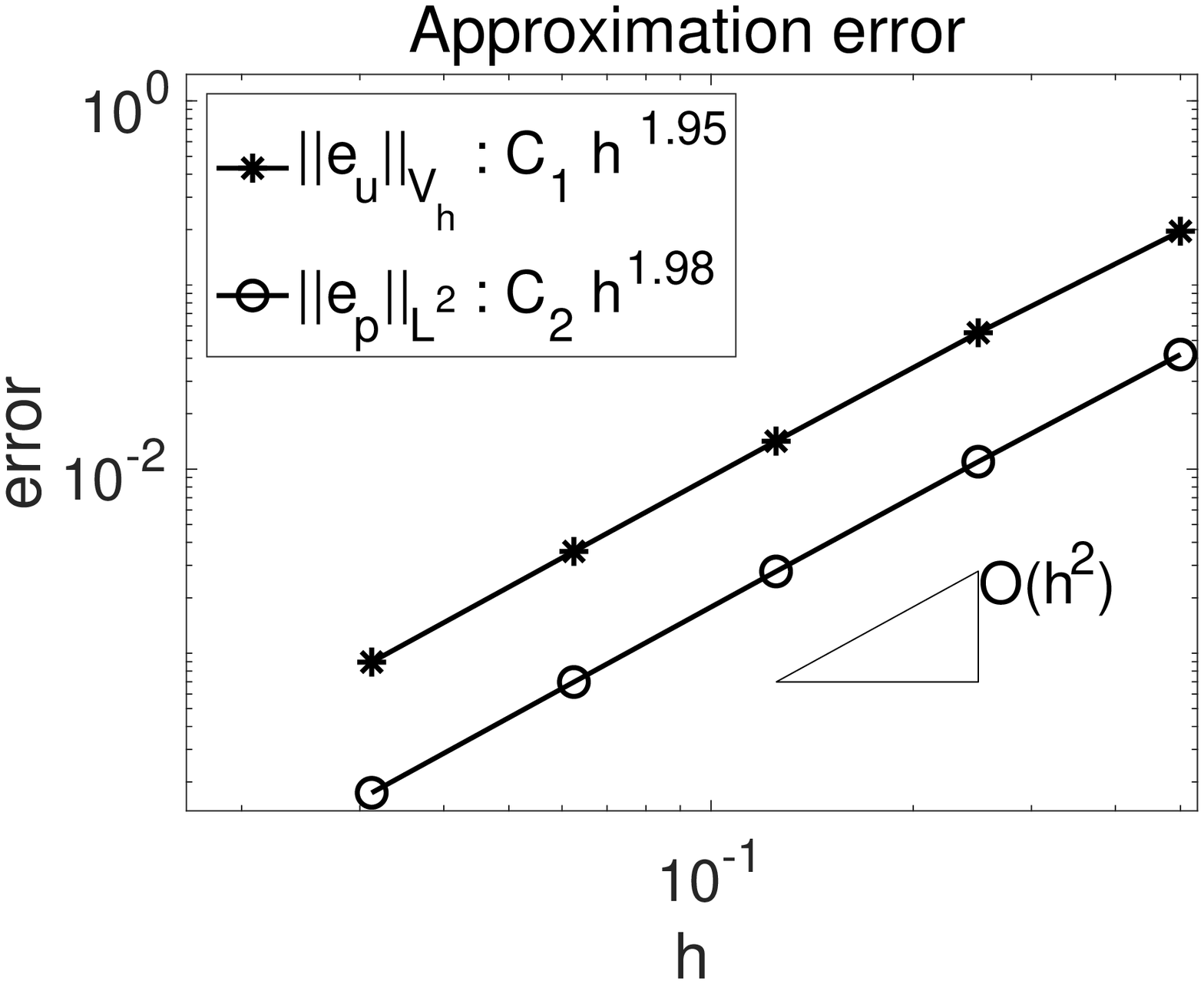} 
\caption{Convergence rates for the $P_1$-$P_1$-$P_0$ and $P_2$-$P_2$-$P_1$ original WG scheme  (\ref{dispro})
  on the middle mesh in Figure \ref{circlemesh}.}
\label{p1p1p0}
\end{center}
\end{figure}

\begin{figure}[H]
\begin{center}
\includegraphics[width=2in]{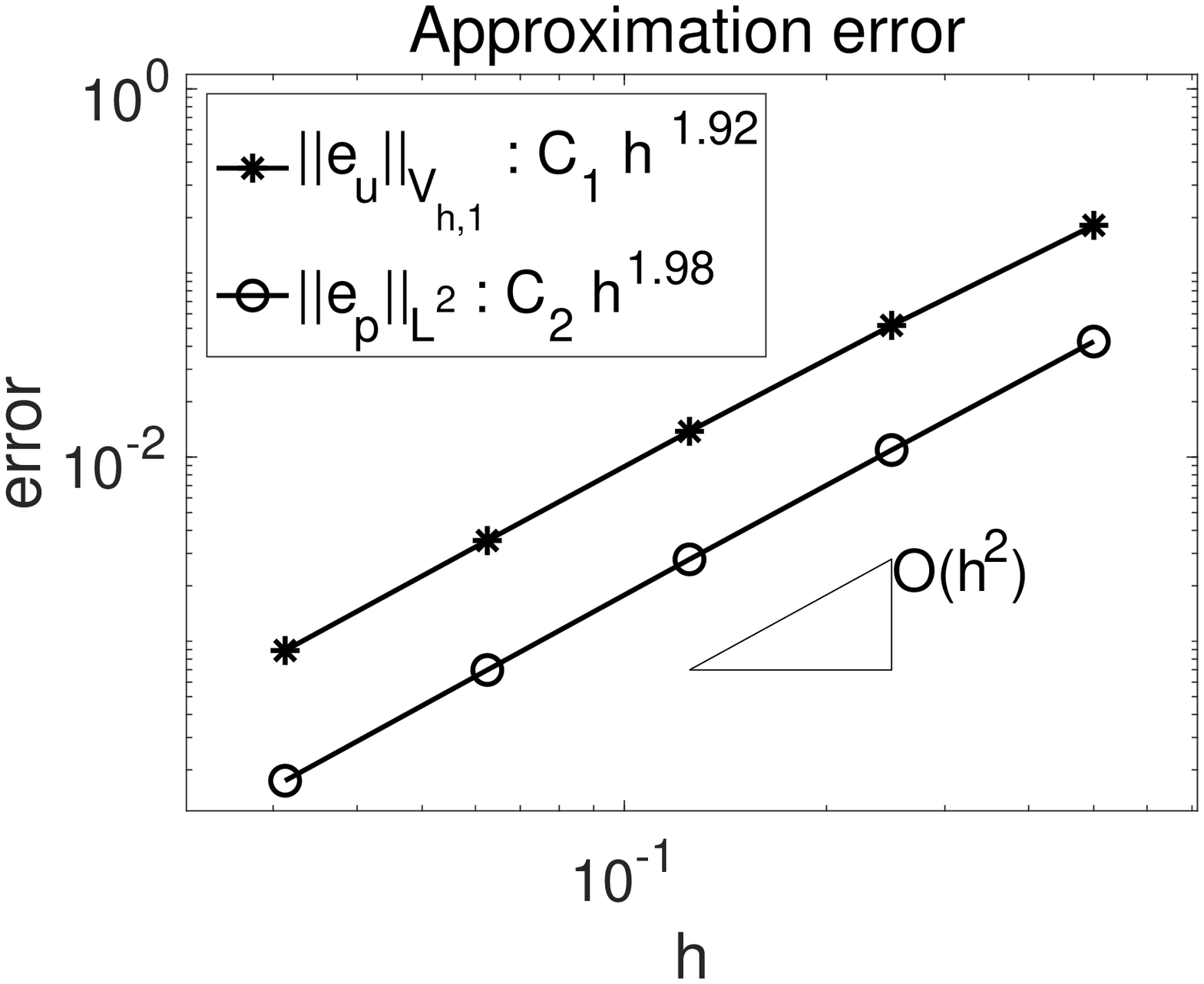}\includegraphics[width=2in]{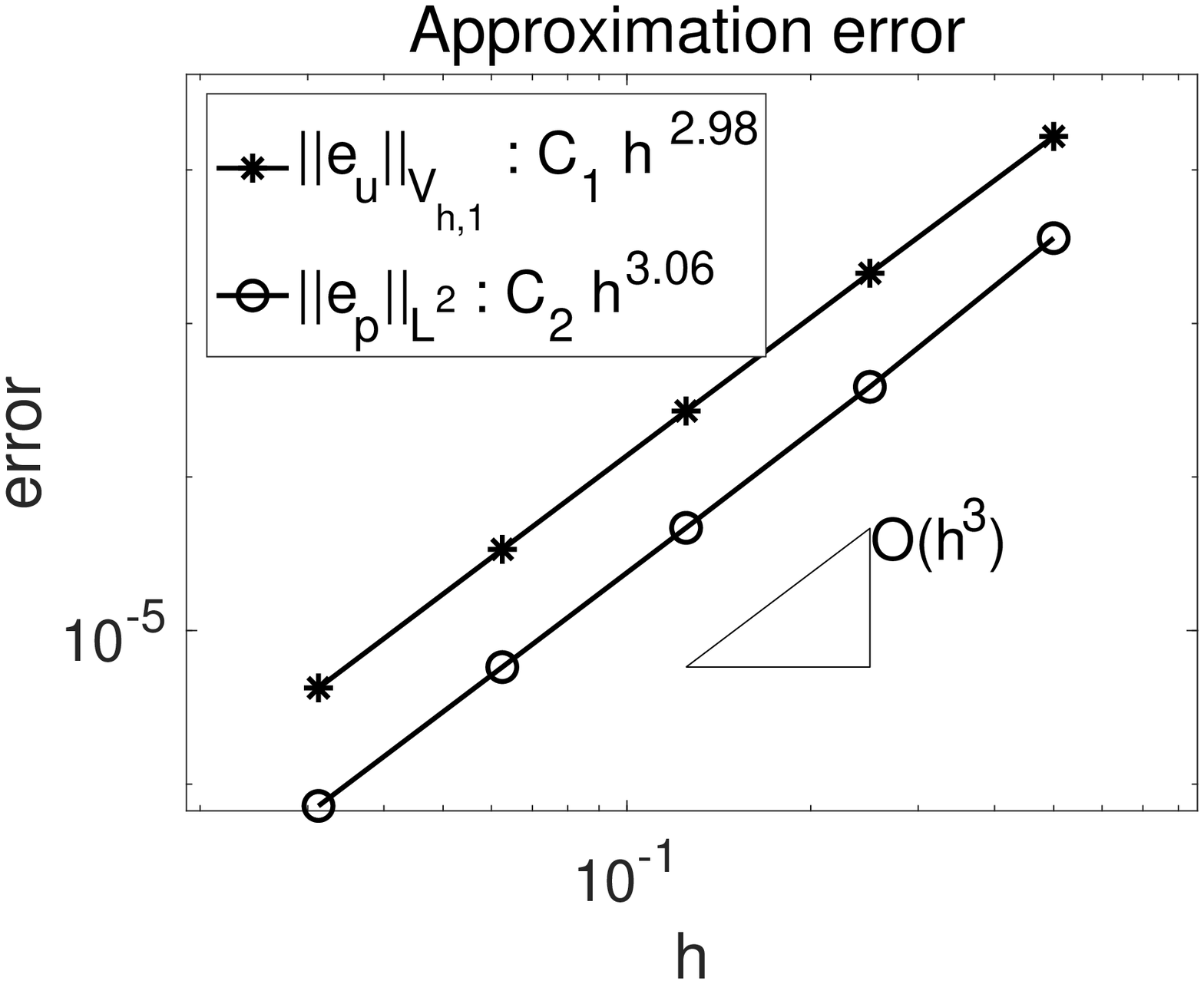} \includegraphics[width=2in]{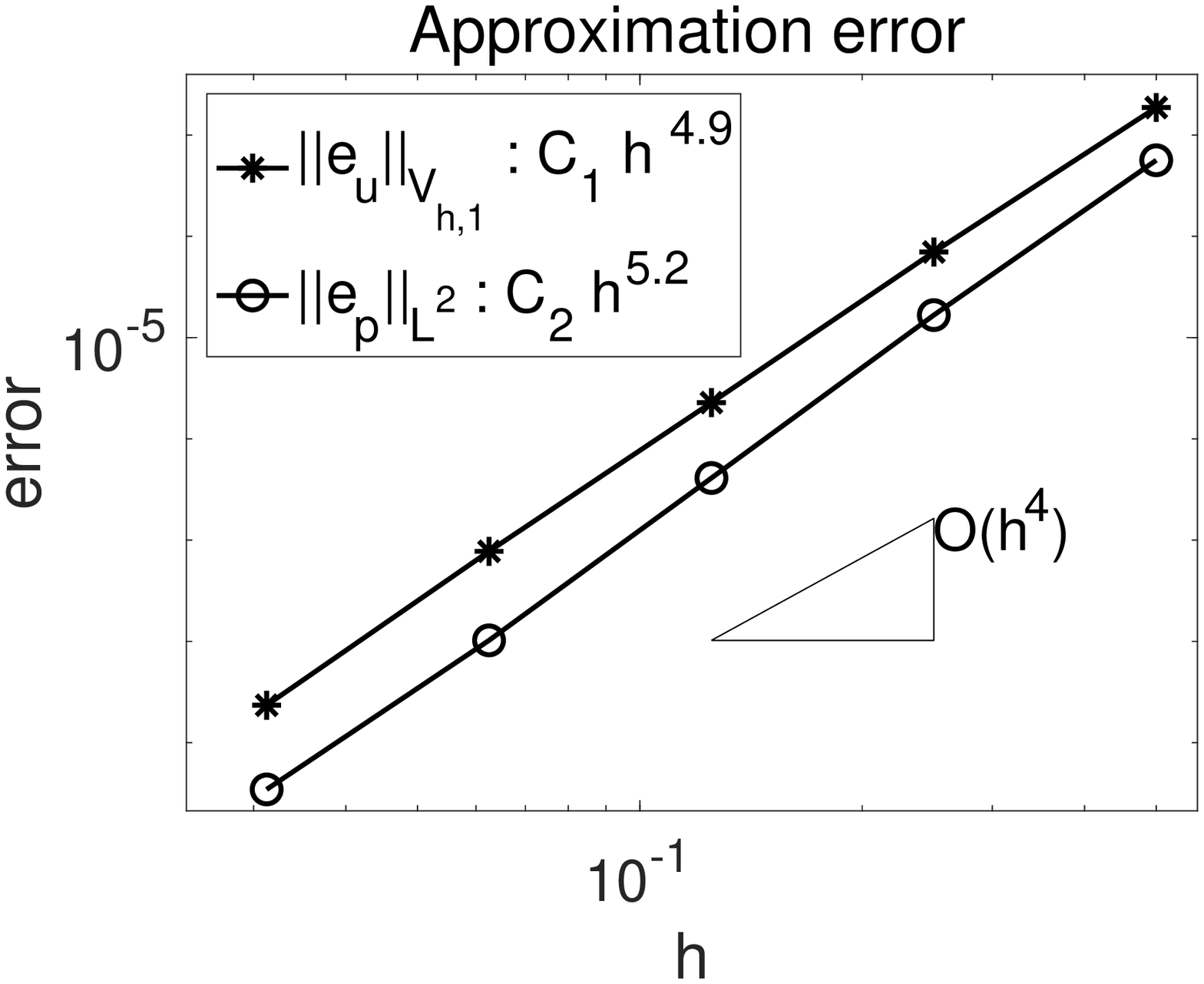} 
\caption{Convergence rates for the $P_2$-$P_2$-$P_1$, $P_3$-$P_3$-$P_2$ and $P_4$-$P_4$-$P_3$ modified WG scheme (\ref{dispro2})
  on the right mesh in Figure \ref{circlemesh}.}
\label{p2p3p4m}
\end{center}
\end{figure}

\noindent{\bf{ Example 5.3. Ring}}

Finally, we consider a ring domain $\Omega = \{(r,\theta)\,|\,r\in(\frac{1}{2},1)\text{, }\theta\in[0,2\pi]\}$.
This domain is non-convex and thus is not covered by the theoretical analysis in the paper.
However, numerical results to be presented next show that the convergence rates are the same as results in convex domains.
We set the exact solution to be
$$
\vu(r,\theta)=\left(
\begin{aligned}
&-\frac{r\sin 2\theta \,(8r - 9)}{2}\\
9r + 9r&(\sin\theta)^2 - 8r^2(\sin\theta)^2 - 4r^2 - 6
\end{aligned}\right),
\qquad
p(r,\theta) = (4r^3+6r-9r^2)  \sin \theta,
$$
which satisfies a homogeneous Neumann boundary condition and $\int_{\Omega} p\dx = 0$.

Again, three meshes are considered, as shown in Figure \ref{ringmesh}.
The leftmost is a triangular mesh with $s=O(h)$.
The middle and the right are polygonal meshes with $s=O(h^{j+\frac{1}{2}})$ and $s=O(\sqrt{h^{j+\frac{1}{2}}})$, respectively.
Numerical results of the original and the modified WG schemes on these meshes are reported in figures \ref{ringwglong}-\ref{p2p3p4ringm},
which exhibit the same convergence rates as in convex domains.

\begin{figure}[H]
\begin{center}
\includegraphics[width=2in]{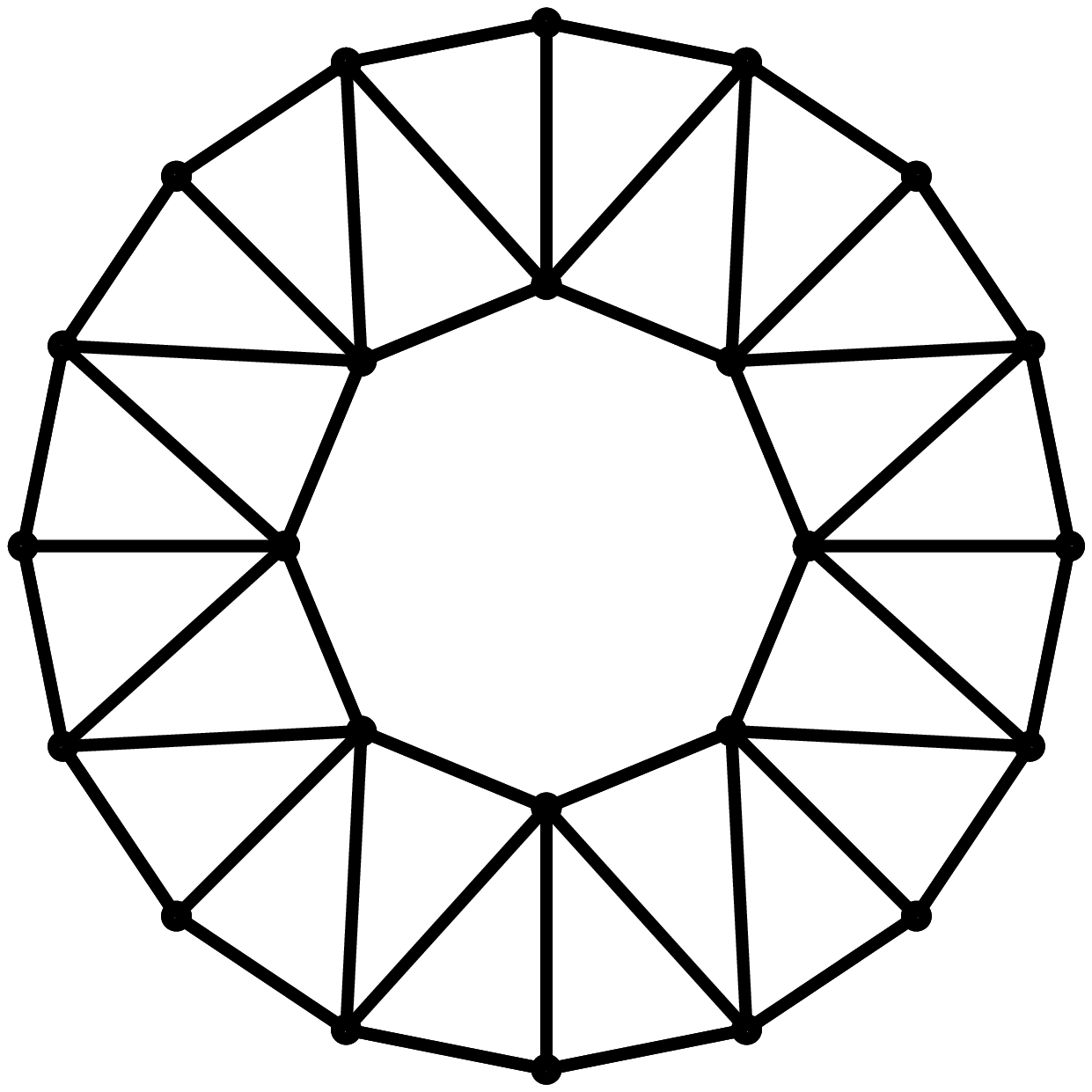} \includegraphics[width=2in]{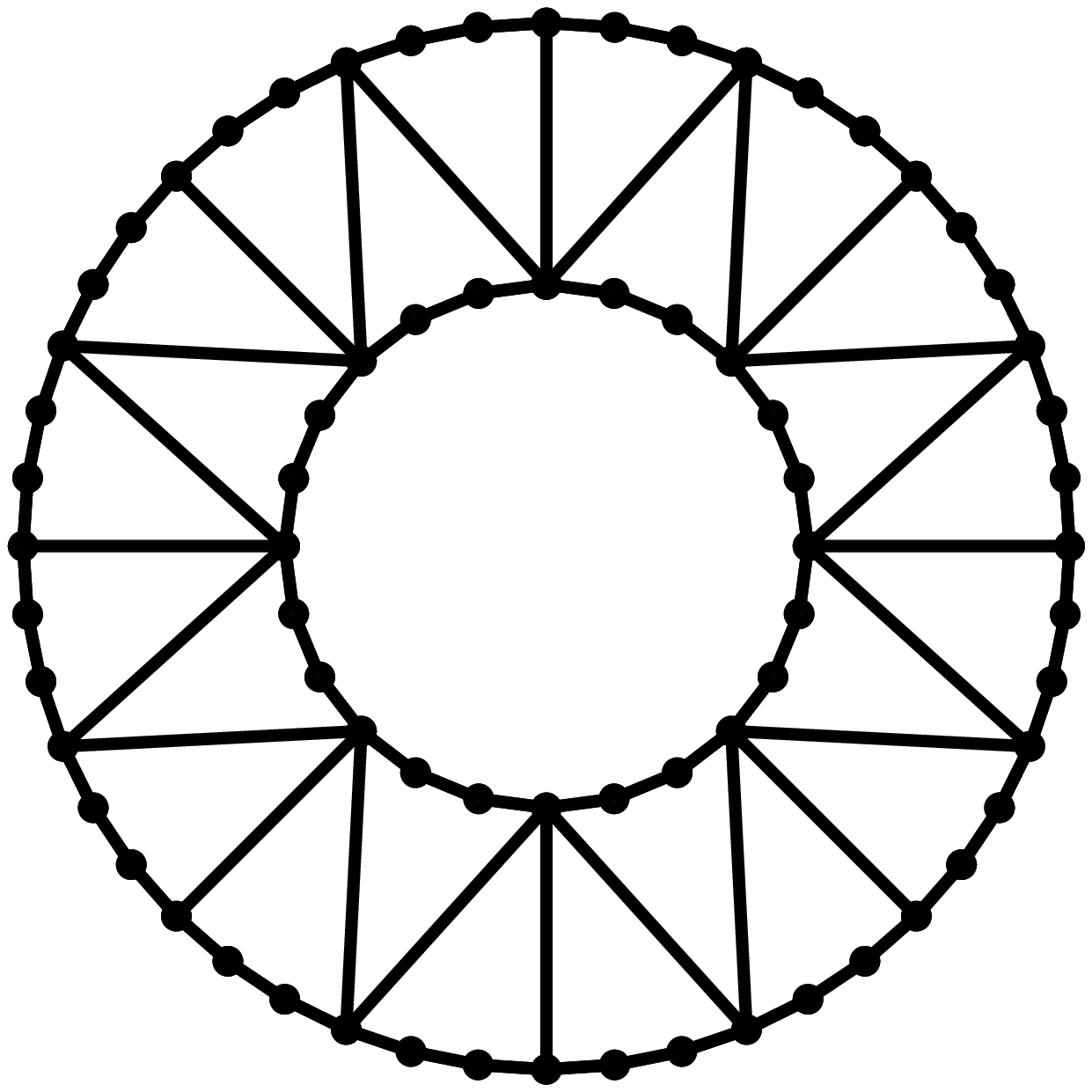} \includegraphics[width=2in]{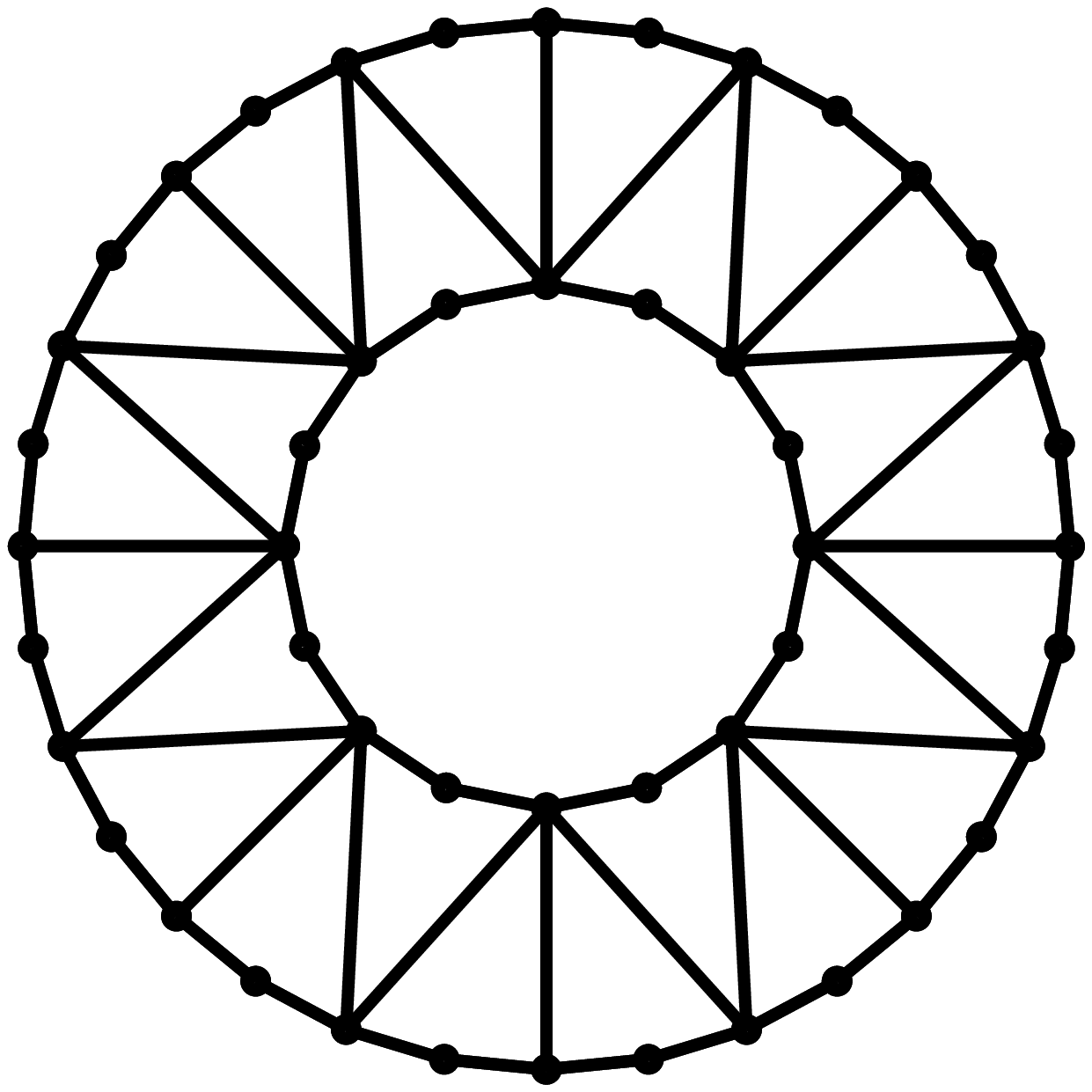} 
\caption{Triangular and polygonal meshes on the ring.}
\label{ringmesh}
\end{center}
\end{figure}

%

\begin{figure}[H]
\begin{center}
\includegraphics[width=2in]{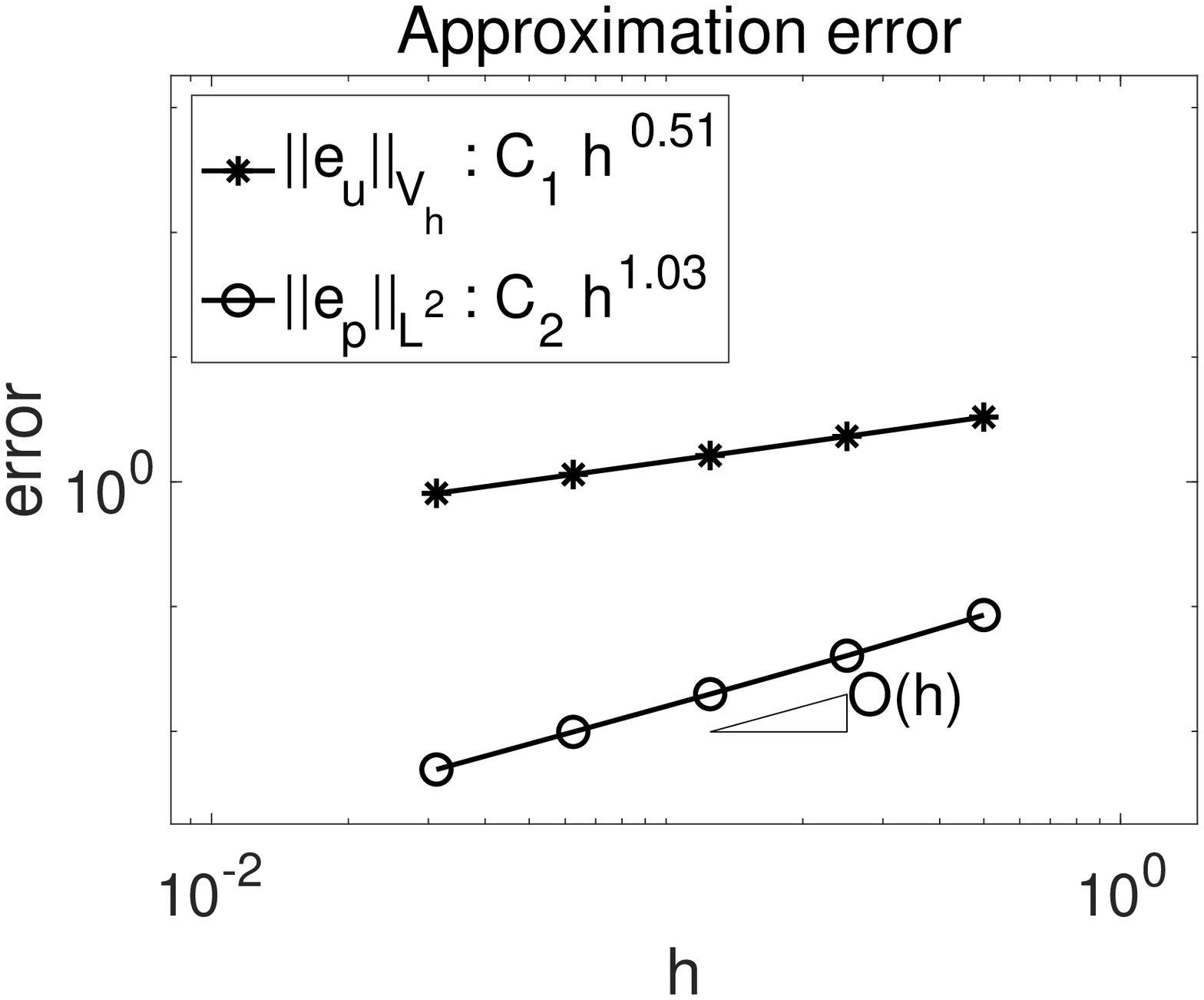} \includegraphics[width=2in]{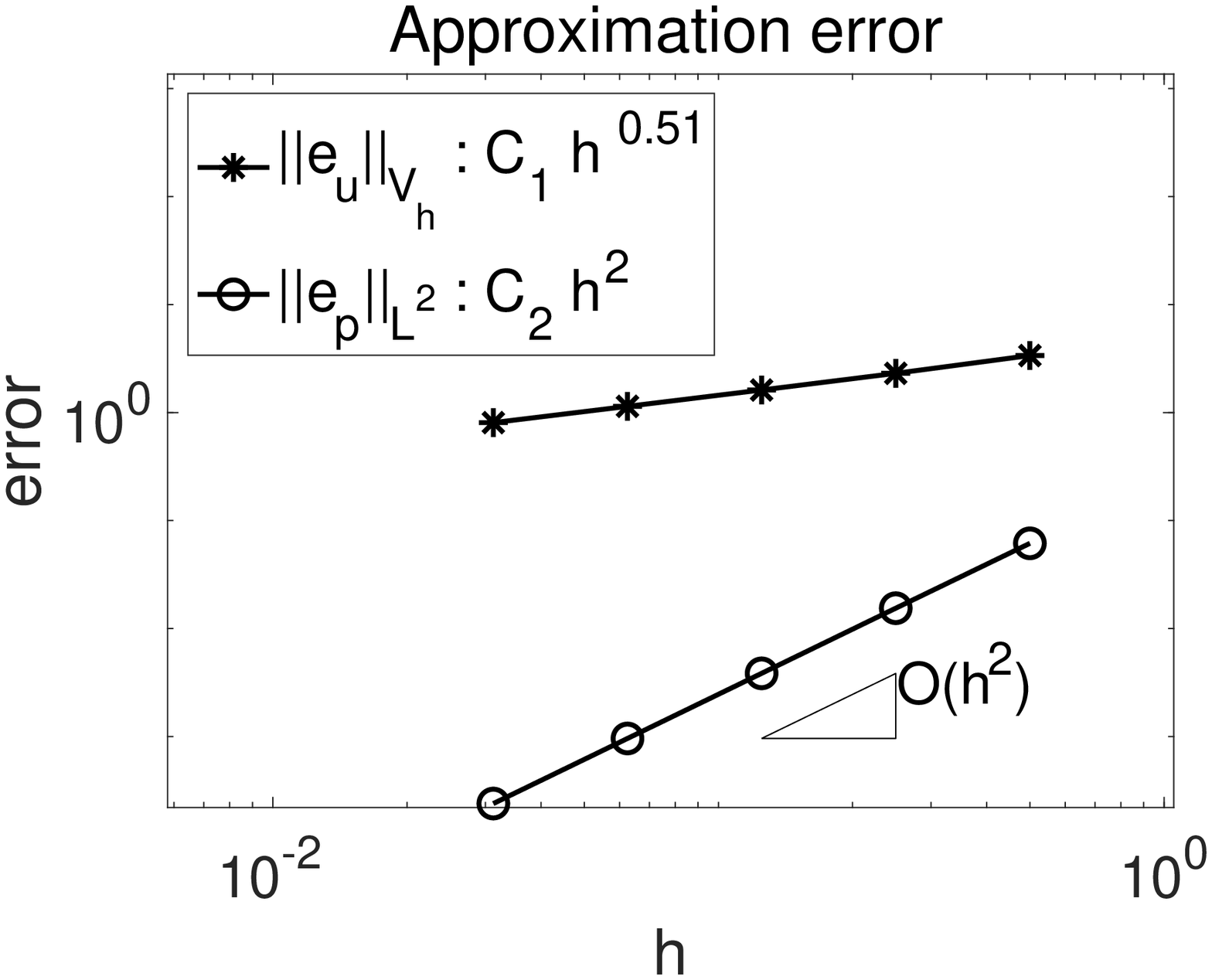} 
\caption{Convergence rates for the $P_1$-$P_1$-$P_0$ and $P_2$-$P_2$-$P_1$ original WG scheme (\ref{dispro})
  on the left mesh in Figure \ref{ringmesh}.}
\label{ringwglong}
\end{center}
\end{figure}
\begin{figure}[H]
\begin{center}
\includegraphics[width=2in]{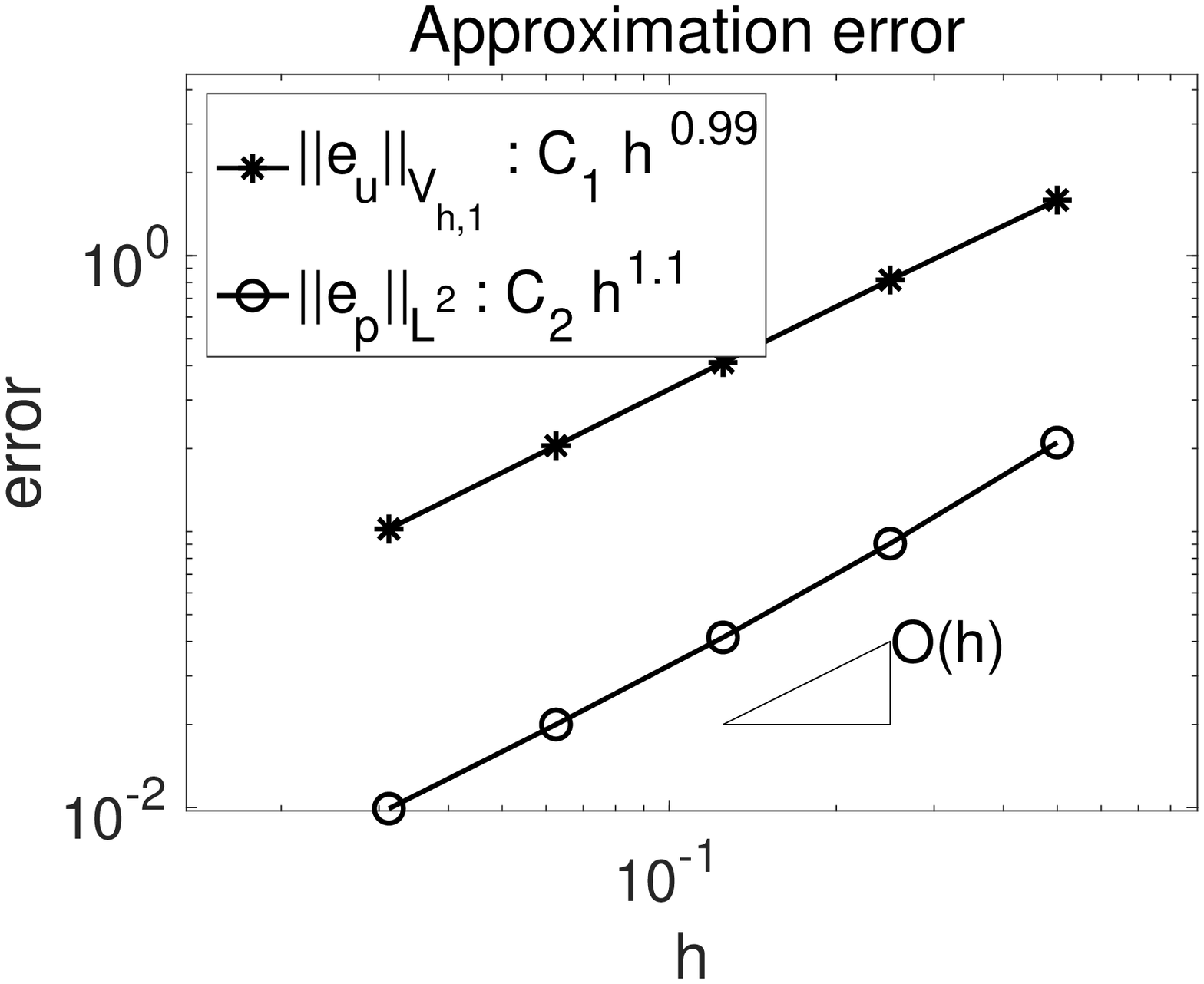}\includegraphics[width=2in]{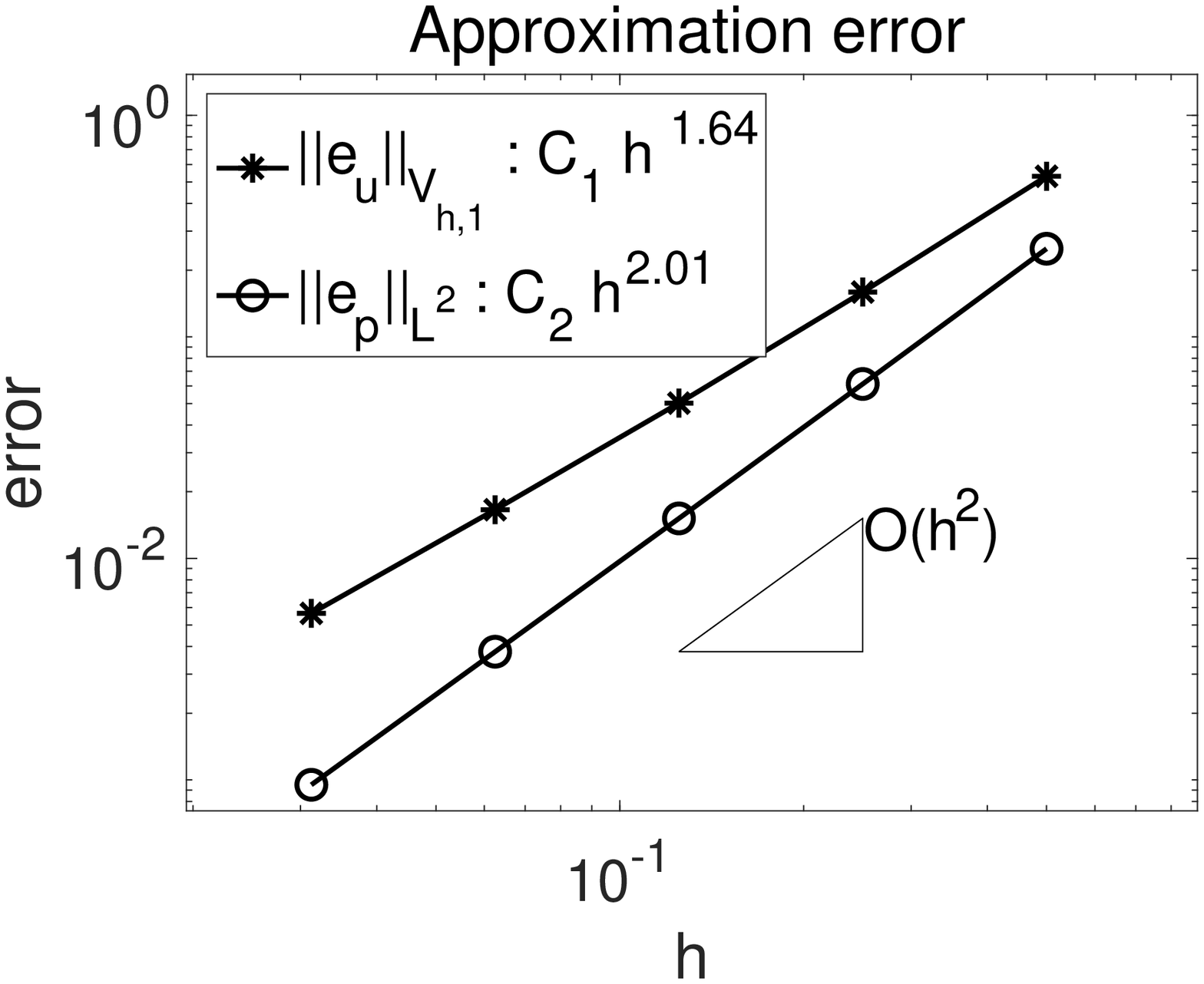}
\caption{Convergence rates for the $P_1$-$P_1$-$P_0$ and $P_2$-$P_2$-$P_1$ modified WG scheme (\ref{dispro2})
  on the left mesh in Figure \ref{ringmesh}.}
\label{p1p2p3p4ring}
\end{center}
\end{figure}


\begin{figure}[H]
\begin{center}
 \includegraphics[width=2in]{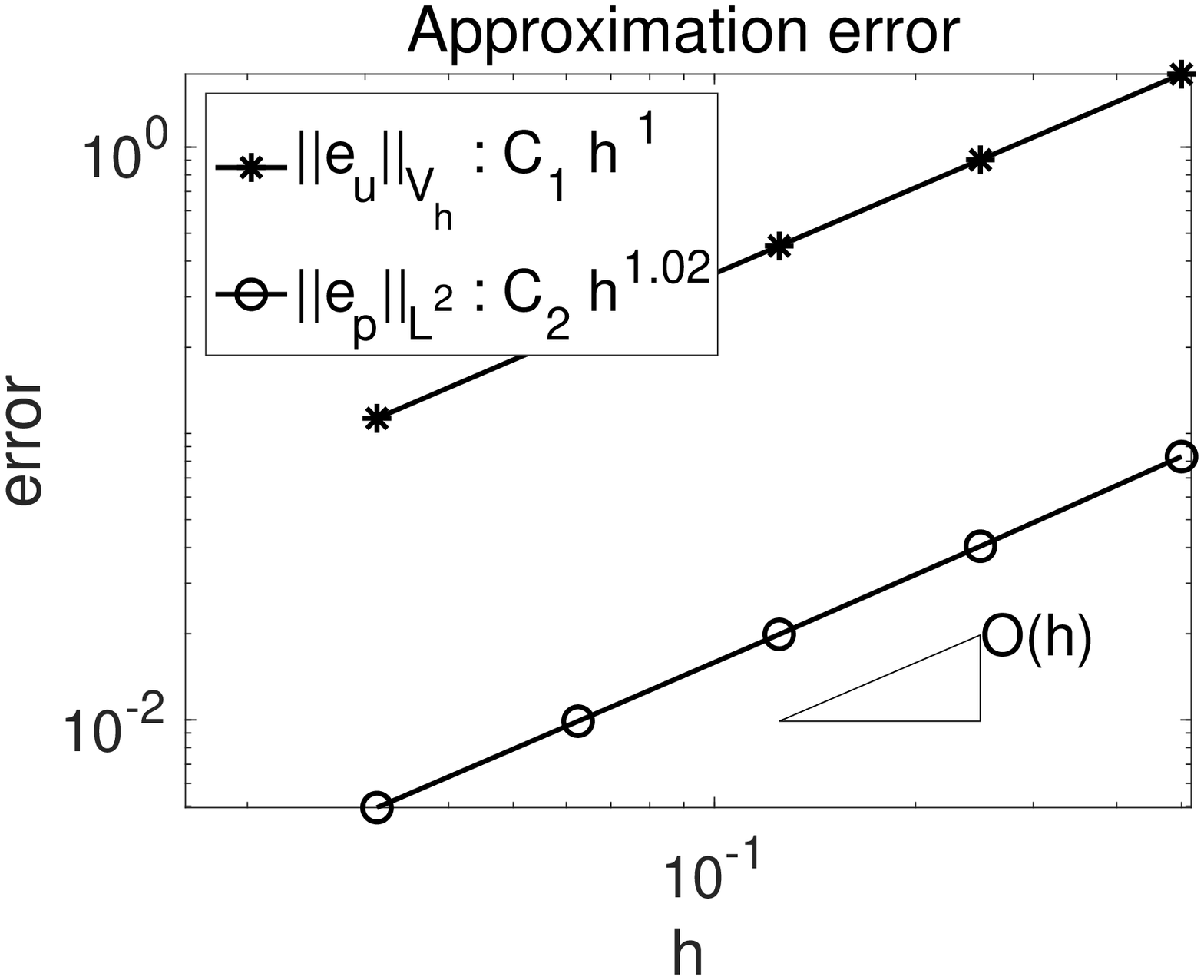} \includegraphics[width=2in]{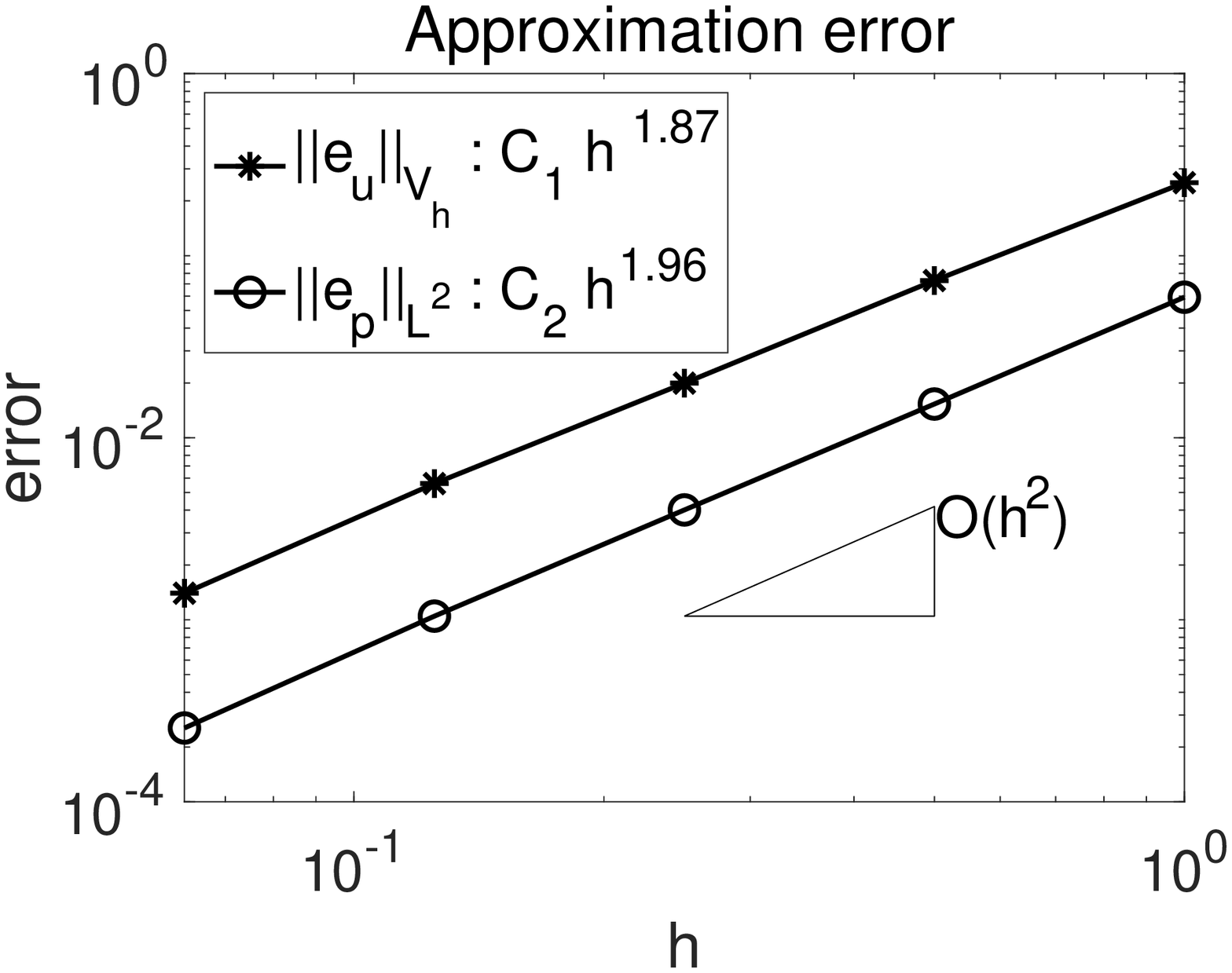}
\caption{Convergence rates for the $P_1$-$P_1$-$P_0$ and  $P_2$-$P_2$-$P_1$ original WG scheme  (\ref{dispro})
  on the middle mesh in Figure \ref{ringmesh}.}
\label{p1p1p0ring}
\end{center}
\end{figure}

\begin{figure}[H]
\begin{center}
 \includegraphics[width=2in]{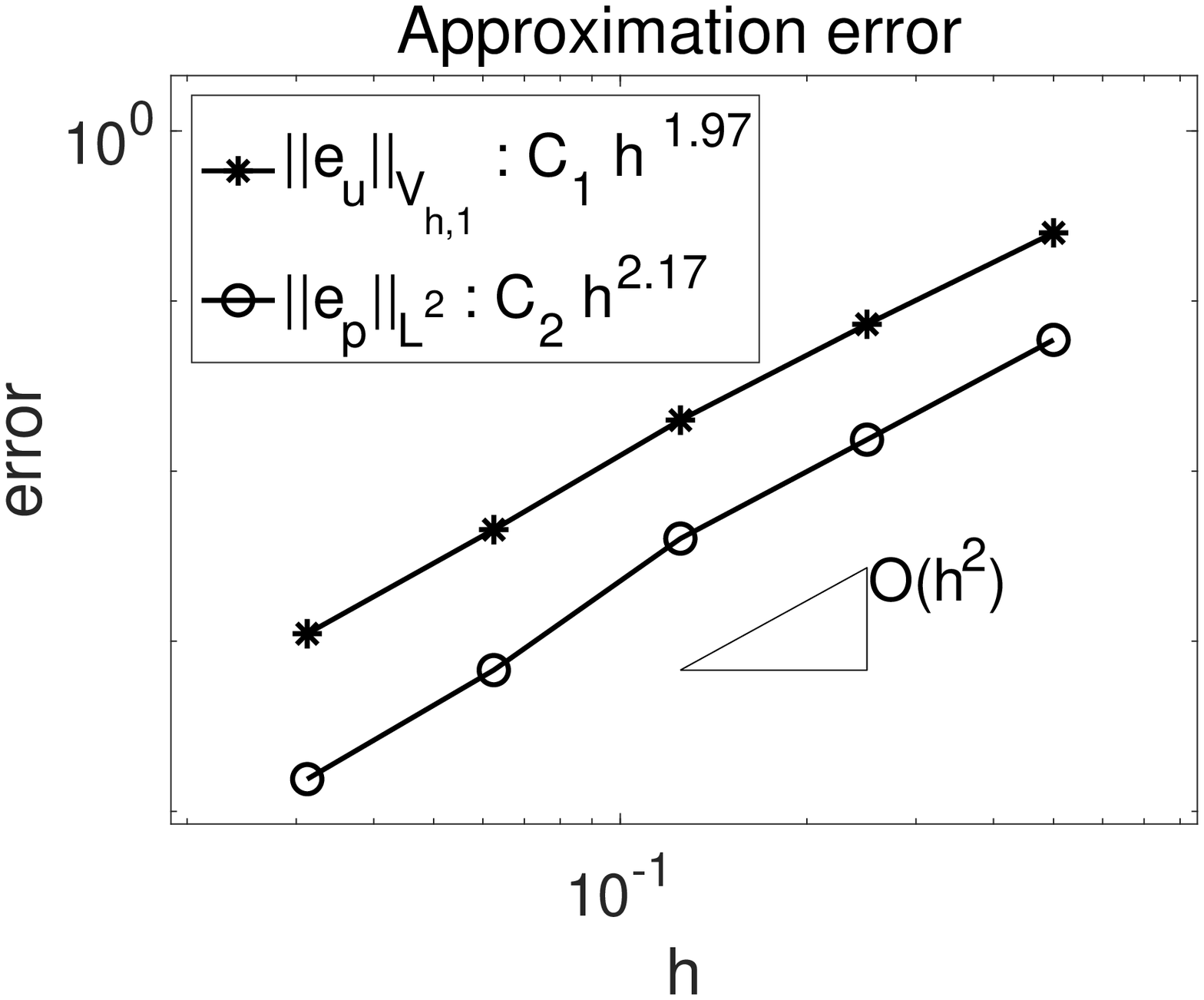} \includegraphics[width=2in]{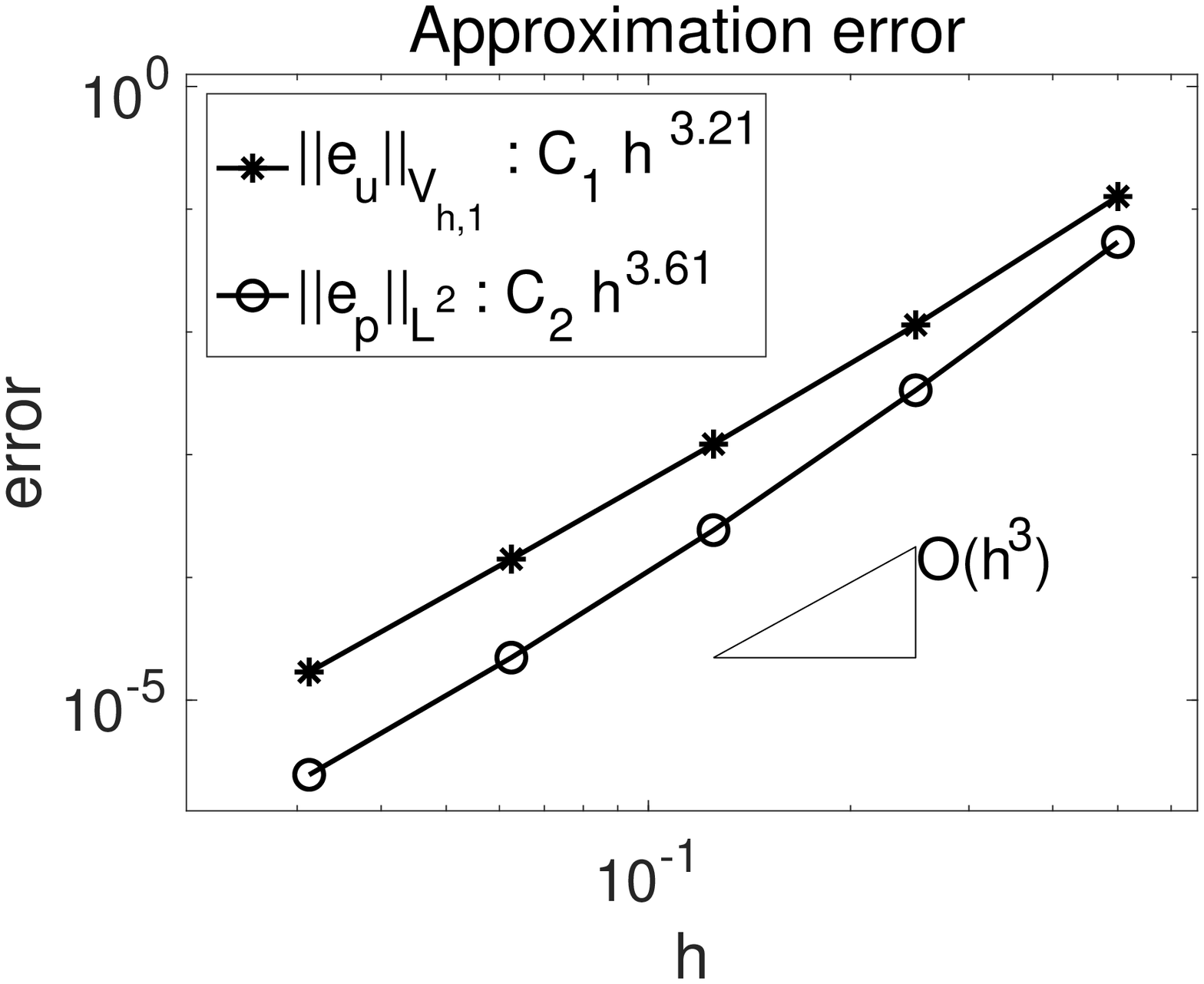} \includegraphics[width=2in]{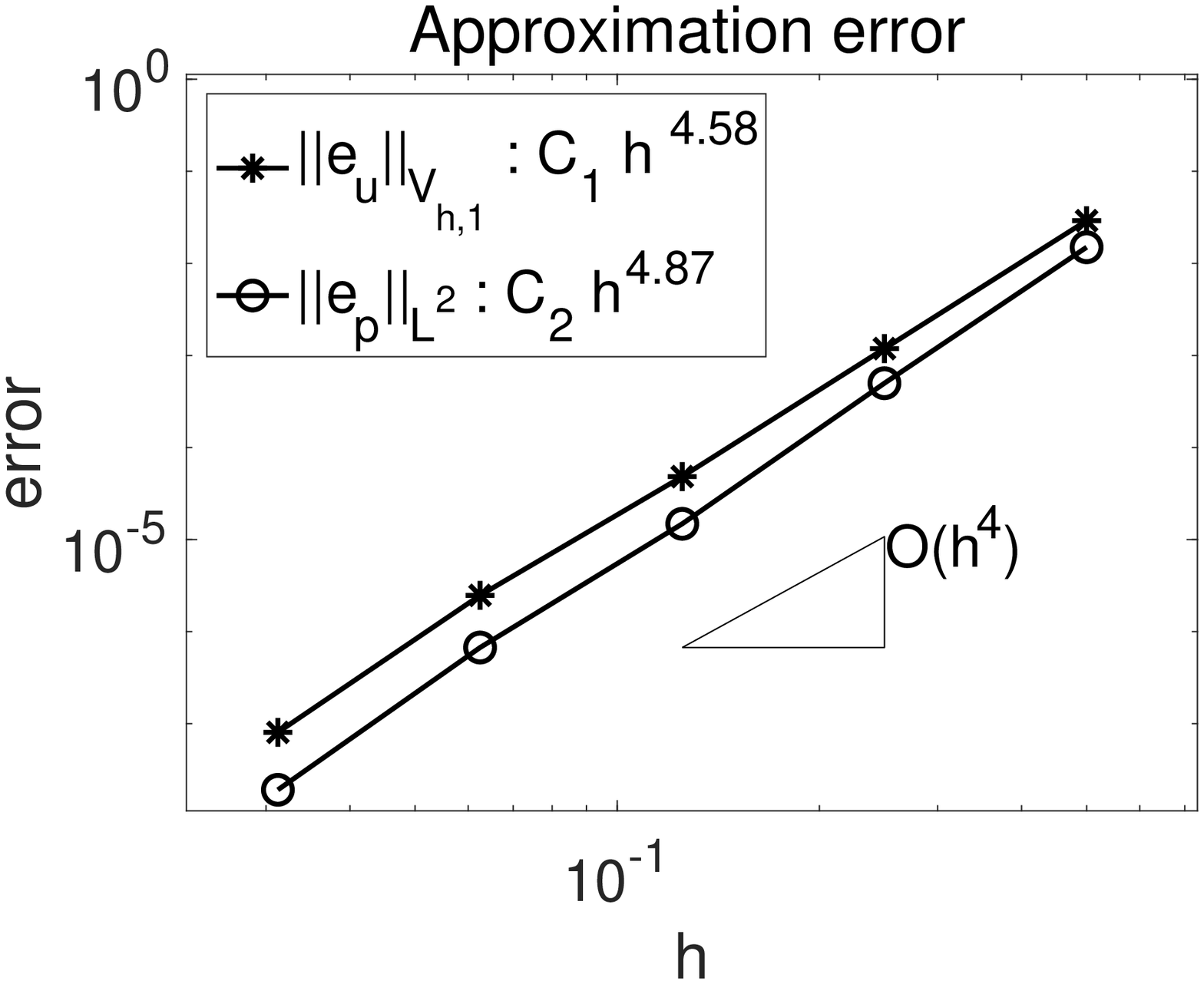} 
\caption{Convergence rates for the $P_2$-$P_2$-$P_1$, $P_3$-$P_3$-$P_2$ and $P_4$-$P_4$-$P_3$ modified WG scheme (\ref{dispro2})
  on the right mesh in Figure \ref{ringmesh}.}
\label{p2p3p4ringm}
\end{center}
\end{figure}

\appendix
\section{Proof of lemmas \ref{lemma1} and \ref{lemma2}} \label{appendix1}
{\bf Proof of Lemma \ref{lemma1}.} Testing the first equation in (\ref{primalproblem}) with $\vv_0$ in $\vv\in \vV_h$  and using integration by parts, one gets
\begin{equation}
\begin{aligned}\label{2.8}
0=& ( \vu, \vv_0)_{\Omega_h}+(\nabla p, \vv_0)_{\Omega_h} \\
=& ( \vu, \vv_0)_{\Omega_h}+\sum_{\e \in \calT_h}(-(\nabla \cdot \vv_0, p)_K+\langle\vv_0 \cdot \vn, p\rangle_{\partial K}) \\
=&( \vu, \vv_0)_{\Omega_h} +\sum_{K \in \calT_h}(-(\nabla \cdot \vv_0, p)_K+\langle(\vv_0-\vv_b) \cdot \vn, p\rangle_{\partial K}),\\
\end{aligned}
\end{equation}
where, in the last step we have used $\vv_b=\vzero$ on $\partial \Omega_h,$ and  the continuity of  $p$ across edges in $\calE_h^{I}$.
Since $\bbQ_h$ and $\vQ_h$ are projections, one gets
\begin{equation}
\begin{aligned}
  b_h(\vv, \bbQ_h  p) &= -(\nabla_{w} \cdot \vv, \bbQ_h  p ) \\
  &=\sum_{K \in \calT_h}\big( (\vv_0, \nabla(\bbQ_h  p))_K-\langle\vv_b \cdot \vn, \bbQ_h  p\rangle_{\partial K}\big) \\
&=\sum_{K \in \calT_h}\big(-(\nabla \cdot \vv_0, \bbQ_h p)_K+\langle(\vv_0-\vv_b) \cdot \vn, \bbQ_h  p\rangle_{\partial K}\big) \\
&=\sum_{K \in \calT_h}\big(-(\nabla \cdot \vv_0, p)_K+\langle(\vv_0-\vv_b) \cdot \vn, \bbQ_h  p\rangle_{\partial K}\big).
\end{aligned}
\end{equation}
Using the definitions of $\vtQ_h$ and $\vQ_h$, we have
\begin{equation}\label{2.11}
\begin{aligned}
c(\vtQ_h\vu, \vv) &= \rho \sum_{K \in \calT_h} h_K^{-1}\langle(\vQ_0\vu-\vtQ_b\vu) \cdot \vn,(\vv_0-\vv_b) \cdot \vn\rangle_{\partial K},&\\
&=c(\vQ_h\vu, \vv) + \rho \sum_{K \in \calT_h} h_K^{-1}\langle{\vQ}_b\vu \cdot \vn,(\vv_0-\vv_b) \cdot \vn\rangle_{\partial K\cap\partial\Omega_h},&\\
&=c(\vQ_h\vu, \vv)+ \rho \sum_{K \in \calT_h} h_K^{-1}\langle \vu \cdot \vn,(\vv_0-\vv_b) \cdot \vn\rangle_{\partial K\cap\partial\Omega_h} \\
&= c(\vQ_h\vu, \vv) + l_s(\vv).
\end{aligned}
\end{equation}

Combining \eqref{2.8}-\eqref{2.11}, one gets
\begin{equation}\begin{aligned}\label{2.10}
a_h( \vtQ_h  {\vu}, \vv)+b_h(\vv, \bbQ_h  p)&=  ( \vQ_0 {\vu}, \vv_0)_{\Omega_h}+c(\vtQ_h {\vu}, \vv)+b_h(\vv, \bbQ_h  p) \\
&=({\vu}, \vv_0)_{\Omega_h}+c(\vQ_h {\vu}, \vv) +l_s(\vv) \\
&\qquad +\sum_{K \in \calT_h}(-(\nabla \cdot \vv_0, p)_K+\langle(\vv_0-\vv_b) \cdot \vn, \bbQ_h  p\rangle_{\partial K}) \\
&=c(\vQ_h \vu, \vv)+l_s(\vv) -l_{\mathrm{div}}(\vv).
\end{aligned}\end{equation}
This completes the proof of the lemma.
\qed

\medskip
{\bf Proof of Lemma \ref{lemma2}.}
 By Lemma \ref{Projection}, we have
$$
b_h(\vQ_h  {\vu}, q)=-(\nabla_{w} \cdot(\vQ_h  {\vu}), q)_{\Omega_h}=-(\pi_h(\nabla \cdot {\vu}), q)_{\Omega_h}=-(\nabla \cdot  {\vu}, q)_{\Omega_h}=-(g , q)_{\Omega_h}.
$$
Using the definitions of $\vtQ_h$, $\vQ_h$ and $\nabla_w\cdot$, one gets
$$
\begin{aligned}
b_h(\vtQ_h  {\vu}, q)
&=b_h(\vQ_h {\vu}, q) - (\nabla_w\cdot(\vtQ_h  {\vu}-\vQ_h  {\vu}),q)_{\Omega_h} \\
&=-(g, q)_{\Omega_h} + \sum_{\e\in\calT_h^B}\langle \vQ_b {\vu}\cdot {\vn},q\rangle_{\partial\e\cap\partial\Omega_h}\\
&=-(g, q)_{\Omega_h} + l_b(q),\\
\end{aligned}
$$
where the last step follows from the property of the projection $\vQ_b$.
This completes the proof of the lemma.
\qed

%


\end{document}